\documentclass[12pt]{amsart}
\date{14 August 2012}
\usepackage{amscd,amssymb,mathdots}
\usepackage{hyperref}
\input xy
\xyoption{all}
\let\oldmarginpar\marginpar
\renewcommand\marginpar[1]{\oldmarginpar{\tiny\bf\begin{flushleft} #1
\end{flushleft}}}

\numberwithin{table}{section}

\newcommand{\la}{\langle}
\newcommand{\ra}{\rangle}

\newcommand{\xra}{\xrightarrow}
\theoremstyle{plain}  % default
\newtheorem{theorem}{Theorem}[section]

\newtheorem*{theorem*}{Theorem}
\newtheorem{corollary}[theorem]{Corollary}
\newtheorem{lemma}[theorem]{Lemma}
\newtheorem{proposition}[theorem]{Proposition}

\theoremstyle{definition}
\newtheorem{definition}[theorem]{Definition}
\newtheorem{notation}[theorem]{Notation}

\theoremstyle{remark}

\newtheorem{remark}[theorem]{Remark}
\newtheorem*{remark*}{Remark}

\newtheorem*{claim*}{Claim}

\newcommand{\curly}{\mathcal}

\newcommand{\bE}{{\bf{E}}}
\newcommand{\dbar}{\bar{\partial}}
\newcommand{\cR}{\mathcal{R}}
\newcommand{\cM}{\mathcal{M}}

\newcommand{\CC}{\mathbb{C}}

\newcommand{\FF}{\mathbb{F}}
\newcommand{\HH}{\mathbb{H}}

\newcommand{\RR}{\mathbb{R}}
\newcommand{\VV}{\mathbb{V}}
\newcommand{\WW}{\mathbb{W}}
\newcommand{\ZZ}{\mathbb{Z}}
\newcommand{\lie}{\mathfrak}

\newcommand{\glie}{\mathfrak{g}}
\newcommand{\hlie}{\mathfrak{h}}

\newcommand{\mlie}{\mathfrak{m}}
\newcommand{\olie}{\mathfrak{o}}

\newcommand{\sllie}{\mathfrak{sl}}
\newcommand{\splie}{\mathfrak{sp}}

\newcommand{\lieg}{\mathfrak{g}}
\newcommand{\lieh}{\mathfrak{h}}

\newcommand{\liem}{\mathfrak{m}}

\newcommand{\liez}{\mathfrak{z}}

\newcommand{\Ad}{\operatorname{Ad}}
\newcommand{\aut}{\operatorname{aut}}
\newcommand{\Aut}{\operatorname{Aut}}

\newcommand{\diag}{\operatorname{diag}}

\newcommand{\End}{\operatorname{End}}

\newcommand{\GL}{\operatorname{GL}}

\newcommand{\Hom}{\operatorname{Hom}}
\newcommand{\Id}{\operatorname{Id}}

\renewcommand{\O}{\operatorname{O}}

\newcommand{\rk}{\operatorname{rk}}
\newcommand{\rank}{\operatorname{rank}}
\newcommand{\SL}{\operatorname{SL}}
\newcommand{\SO}{\operatorname{SO}}

\newcommand{\Sp}{\operatorname{Sp}}
\newcommand{\Pin}{\operatorname{Pin}}

\newcommand{\SU}{\operatorname{SU}}

\newcommand{\Tr}{\operatorname{Tr}}

\newcommand{\U}{\operatorname{U}}

\newcommand{\Mg}{\mathcal{M}^{\operatorname{gauge}}}

\newcommand{\HHH}{{\curly H}}

\newcommand{\OOO}{{\curly O}}
\newcommand{\PPP}{{\curly P}}

\newcommand{\qu}{/\kern-.7ex/}
\newcommand{\exh}{\to\kern-1.8ex\to}

\newcommand{\VP}{{\curly V}\kern-0.9ex\PPP}

\newcommand{\spn}{\Sp(2n,\RR)}

\newcommand{\abs}[1]{\lvert#1\rvert}
\newcommand{\norm}[1]{\lVert#1\rVert}
\newcommand{\suchthat}{\;|\;}

\newcommand{\into}{\hookrightarrow}
\newcommand{\ad}{\mathrm{ad}}

\newcommand{\Kto}{\ar@{*+=[o][F]{\scriptscriptstyle{K}}->}[r]}

\newcommand{\HC}{H^{\CC}}

\newcommand{\mclie}{\mlie^{\CC}}

\title[Higgs bundles and surface group representations]
{Higgs bundles and surface group representations in the
real symplectic group}
\author{O. Garc\'{\i}a-Prada, P.~B. Gothen, I. Mundet i Riera}
\thanks{
  Members of the Research Group VBAC (Vector Bundles on Algebraic
  Curves).
  Research partially supported by Ministerio de Educaci\'{o}n y
  Ciencia, CSIC, Conselho de Reitores das Universidades Portuguesas
  and FCT (Portugal) through Spain--Portugal bilateral research
  projects.
  First and Third authors partially supported by Ministerio de
  Educaci{\'o}n y Ciencia (Spain) through Projects
  MTM2004-07090-C03-01 and MTM2007-67623.
  Second author partially supported by the FCT (Portugal) with EU
  (FEDER/COMPETE) and national funds through the projects
  PTDC/MAT/099275/2008 and PTDC/MAT/098770/2008, and through Centro de
  Matem\'atica da Universidade do Porto (PEst-C/MAT/UI0144/2011).
}

\begin{document}

\begin{abstract}
  In this paper we study the moduli space of representations of a
  surface group (i.e., the fundamental group of a closed oriented
  surface) in the real symplectic group $\Sp(2n,\RR)$. The moduli
  space is partitioned by an integer invariant, called the Toledo
  invariant. This invariant is bounded by a Milnor--Wood type
  inequality. Our main result is a count of the number of connected
  components of the moduli space of maximal representations, i.e.\
  representations with maximal Toledo invariant. Our approach uses the
  non-abelian Hodge theory correspondence proved in a companion paper
  \cite{garcia-prada-gothen-mundet:2009a}
  to identify the space of
  representations with the moduli space of polystable
  $\Sp(2n,\RR)$-Higgs bundles. A key step is provided by the discovery
  of new discrete invariants of maximal representations. These new
  invariants arise from an identification, in the maximal case, of the
  moduli space of $\Sp(2n,\RR)$-Higgs bundles with a moduli space of
  twisted Higgs bundles for the group $\GL(n,\RR)$.
\end{abstract}

\maketitle

%\newpage

%\tableofcontents

%%%%%%%%%%%%%%%%%%%%%%%
\section{Introduction}
%%%%%%%%%%%%%%%%%%%%%%%
 
\begin{flushright}
{\it Valeu a pena? Tudo vale a pena \\
Se a alma n\~ao \'e pequena.} \\ \vspace*{0.25cm} F. Pessoa
\end{flushright}

\vspace*{0.25cm}

In this paper we study representations of the fundamental group of
a compact oriented surface $X$ in $\Sp(2n,\RR)$ --- the group of
linear transformations of $\RR^{2n}$ which preserve the standard
symplectic form.  By a representation we mean a homomorphism from
$\pi_1(X)$ to $\Sp(2n,\RR)$.  Given a representation of $\pi_1(X)$
in $\Sp(2n,\RR)$ there is an integer, often referred to as the
{\em Toledo invariant}, associated to it. This integer can be
obtained geometrically by considering the flat
$\Sp(2n,\RR)$-bundle corresponding to the representation and
taking a reduction of the structure group of the underlying smooth
vector bundle to $\U(n)$
--- a maximal compact subgroup of $\Sp(2n,\RR)$. The degree of
the resulting $\U(n)$-bundle is the Toledo invariant.
As shown by Turaev \cite{turaev:1984} the Toledo invariant $d$ of
a representation satisfies the inequality
\begin{equation}\label{mw1}
\abs{d} \leq n(g-1),
\end{equation}
where $g$ is the genus of the surface.  When $n=1$, one has
$\Sp(2,\RR)\simeq\SL(2,\RR)$, the Toledo invariant coincides with
the Euler class of the $\SL(2,\RR)$-bundle, and (\ref{mw1}) is the
classical inequality of Milnor \cite{milnor:1957} which was later
generalized by Wood \cite{wood:1971}. We shall follow custom and
refer to (\ref{mw1}) as as the Milnor--Wood inequality.

Given two representations, a basic question to ask is whether one can
be continuously deformed into the other. Put in a more precise way, we are
asking for the connected components of the space of representations
\begin{displaymath}
  \Hom(\pi_1(X),\Sp(2n,\RR)).
\end{displaymath}
As shown in \cite{garcia-prada-mundet:2004}, this space has the same
number of connected components as the moduli space, or character
variety,
\begin{displaymath}
  \cR(\pi_1(X),\Sp(2n,\RR))
    = \Hom^{\mathrm{red}}(\pi_1(X),\Sp(2n,\RR)) / \Sp(2n,\RR)
\end{displaymath}
of reductive representations $\rho\colon \pi_1(X) \to
\Sp(2n,\RR)$, modulo the natural equivalence given by the action
of $\Sp(2n,\RR)$ by overall conjugation.

The Toledo invariant descends to the quotient so, for any $d$
satisfying (\ref{mw1}), we can define
\begin{displaymath}
  \cR_d(\pi_1(X),\Sp(2n,\RR)) \subset \cR(\pi_1(X),\Sp(2n,\RR))
\end{displaymath}
to be the subspace of representations with Toledo invariant $d$.  For
ease of notation, for the remaining part of the Introduction, we shall
write $\cR_d$ for $\cR_d(\pi_1(X),\Sp(2n,\RR))$ and $\cR$ for
$\cR(\pi_1(X),\Sp(2n,\RR))$.  Since the Toledo invariant varies
continuously with the representation, the subspace $\cR_d$ is a union
of connected components, and our basic problem is that of counting the
number of connected components of $\cR_d$ for $d$ satisfying
(\ref{mw1}). This has been done for $n=1$ by Goldman
\cite{goldman:1980,goldman:1988} and Hitchin \cite{hitchin:1987a}, and
for $n=2$ in \cite{gothen:2001} (in the cases $d=0$ and
$\abs{d}=2g-2$) and \cite{garcia-prada-mundet:2004} (in the cases
$\abs{d}<2g-2$). In this paper we count the number of connected
components of $\cR_d$ for $n>2$ when $d=0$ and $\abs{d}=n(g-1)$ ---
the maximal value allowed by the Milnor--Wood inequality. Our main
result is the following (Theorem~\ref{connected-components-rep}
below).

\begin{theorem}
  Let $X$ be a compact oriented surface of genus $g$.  Let $\cR_d$ be
  the moduli space of reductive representations of $\pi_1(X)$ in
  $\Sp(2n,\RR)$ with Toledo invariant $d$.  Let $n\geq 3$. Then
\begin{enumerate}
\item $\cR_0$ is non-empty and connected;
\item $\cR_{\pm n(g-1)}$ has $3. 2^{2g}$ non-empty connected components.
\end{enumerate}
\end{theorem}

The main tool we employ to count connected components is the theory of
Higgs bundles, as pioneered by Hitchin \cite{hitchin:1987a} for
$\SL(2,\RR) = \Sp(2,\RR)$. Fix a complex structure on $X$ endowing
it with a structure of a compact Riemann surface, which we will
denote, abusing notation, also by $X$. An $\Sp(2n,\RR)$-Higgs bundle
over $X$ is a triple $(V,\beta,\gamma)$ consisting of a rank $n$
holomorphic vector bundle $V$ and holomorphic sections $\beta\in
H^0(X,S^2V\otimes K)$ and $\gamma\in H^0(X,S^2V^*\otimes K)$, where
$K$ is the canonical line bundle of $X$. The sections $\beta$ and
$\gamma$ are often referred to as Higgs fields.  Looking at $X$ as an
algebraic curve, algebraic moduli spaces for $\Sp(2n,\RR)$-Higgs
bundle exist as a consequence of the work of Schmitt
\cite{schmitt:2005,schmitt:2008}. Fixing $d\in \ZZ$, we denote by
$\cM_d$ the moduli space of $\Sp(2n,\RR)$-Higgs bundles on $X$ with
$\deg(V)=d$. As usual, one must introduce an appropriate stability
condition (with related conditions of poly- and semistability) in
order to have good moduli spaces.  Thus $\mathcal{M}_d$ parametrizes
isomorphism classes of polystable $\Sp(2n,\RR)$-Higgs bundles. A basic
result of non-abelian Hodge theory, growing out of the work of
Corlette \cite{corlette:1988}, Donaldson \cite{donaldson:1987},
Hitchin \cite{hitchin:1987a} and Simpson
\cite{simpson:1988,simpson:1992,simpson:1994,simpson:1995}, is the
following (Theorem~\ref{na-Hodge} below).

\begin{theorem}
The moduli spaces $\cR_d$ and $\cM_d$ are homeomorphic.
\end{theorem}

An essential part of the proof of this Theorem follows from a
Hitchin--Kobayashi correspondence between polystable
$\Sp(2n,\RR)$-Higgs bundles and solutions to certain gauge theoretic
equations, known as Hitchin's equations (see
Theorem~\ref{higgs-hk}).
%Section~\ref{sec:g-higgs-bundles-hitchin}).
%
In the generality required
for stable $\Sp(2n,\RR)$-Higgs bundles, the Hitchin--Kobayashi
correspondence is provided by the general theory of
\cite{bradlow-garcia-prada-mundet:2003}, together with the extension to
the case of polystable (non-stable) pairs in general proved in
\cite{garcia-prada-gothen-mundet:2009a}  
%\footnote{The precise referencing to \cite{garcia-prada-gothen-mundet:2009a}
 % refers to the preprint {\bf ?????????????????}
  %{\tt arXiv:0909.4487v2} of the {\tt arXiv}.}.

Using the homeomorphism $\cR_d\simeq \cM_d$, our problem is reduced to
studying the connectedness properties of $\cM_d$. This is done by
using the Hitchin functional. This is a non-negative function
$f$ which is defined on $\cM_d$ using the solution to Hitchin's
equations. This function arises as the moment map for the Hamiltonian
circle action on the moduli space obtained by multiplying the Higgs
field by an element of $\U(1)$ and is (essentially) the $L^2$-norm of
the Higgs field. It was proved by Hitchin
\cite{hitchin:1987a,hitchin:1992} that $f$ is proper, and this implies
that $f$ has a minimum on each connected component of
$\mathcal{M}_d$. Using this fact, our problem essentially reduces to
characterizing the subvariety of minima of the Hitchin functional and
studying its connectedness properties.

While we characterize the minima for every value of $d$ satisfying the
Milnor--Wood inequality (see Theorem \ref{thm:all-minima}), we only
carry out the full programme for $d=0$ and $\abs{d}=n(g-1)$, the
extreme values of $d$.  For $d=0$, the subvariety of minima of the
Hitchin functional on $\cM_0$ coincides with the set of Higgs bundles
$(V,\beta,\gamma)$ with $\beta=\gamma=0$. This, in turn, can be
identified with the moduli space of polystable vector bundles of rank
$n$ and degree $0$. Since this moduli space is connected by the
results of Narasimhan--Seshadri \cite{narasimhan-seshadri:1965},
$\cM_0$ is connected and hence $\cR_0$ is connected.

The analysis for the \emph{maximal case}, $\abs{d}=n(g-1)$, is far
more involved and interesting.  It turns out that in this case one
of the Higgs fields $\beta$ or $\gamma$ for a semistable Higgs
bundle $(V,\beta,\gamma)$ becomes an isomorphism. Whether it is
$\beta$ or $\gamma$, actually depends on the sign of the Toledo
invariant. Since the map $(V,\beta,\gamma)\mapsto
(V^\ast,\gamma^t,\beta^t)$ defines an isomorphism $\cM_{-d}\simeq
\cM_d$, there is no loss of generality in assuming that $0\leq
d\leq n(g-1)$. Suppose that $d=n(g-1)$. Then $\gamma: V\to
V^\ast\otimes K$ is an isomorphism (see Proposition
\ref{mw-higgs}). Since $\gamma$ is furthermore symmetric, it
equips $V$ with a $K$-valued non-degenerate quadratic form. In
order to have a proper quadratic bundle, we fix a square root
$L_0=K^{-1/2}$  of $K^{-1}$, and
define $ W:=V\otimes L_0$. Then 
$q:=\gamma\otimes I_{L_0}: W \to W^*$ 
is a symmetric isomorphism defining a non-degenerate 
symmetric bilinear form $Q$ on $W$, in other words,  $(W,Q)$ is an
$\O(n,\CC)$-holomorphic bundle. The $K^2$-twisted endomorphism
$\psi:W\to W \otimes K^2$ defined by $\psi: = 
\beta \otimes I_{L_0^{-1}}\circ  (\gamma\otimes I_{L_0})$ is $Q$-symmetric and
hence $(W,Q,\psi)$ defines
what we
call a $K^2$-twisted $\GL(n,\RR)$-Higgs pair, from which we can
recover the original $\Sp(2n,\RR)$-Higgs bundle.  The main result
is the following (Theorem~\ref{thm:cayley} below).

\begin{theorem}\label{cayley}
Let $\cM_{\max}$ be the moduli space of polystable
$\Sp(2n,\RR)$-Higgs bundles with $d=n(g-1)$,
 and let $\cM'$ be
the moduli space of polystable $K^2$-twisted $\GL(n,\RR)$-Higgs pairs.
The map $(V,\beta,\gamma)\mapsto (W,Q,\psi)$ defines an
isomorphism of complex algebraic varieties
$$
\cM_{\max}\simeq \cM'.
$$
\end{theorem}

We refer to this isomorphism as the {\em Cayley correspondence}. This
name is motivated by the geometry of the bounded symmetric domain
associated to the Hermitian symmetric space $\Sp(2n,\RR)/\U(n)$. The
Cayley transform defines a biholomorphism between this domain and a
tube type domain defined over the symmetric cone $\GL(n,\RR)/\O(n)$
--- the Siegel upper half-space.  In fact, there is a similar
correspondence to that given in Theorem~\ref{cayley} for every
semisimple Lie group $G$ which, like $\Sp(2n,\RR)$, is the group of
isometries of a Hermitian symmetric space of tube type (see
\cite{bradlow-garcia-prada-gothen:2005} for a survey on this subject).

A key point is that the Cayley correspondence brings to the surface
new topological invariants. These invariants, hidden a priori, are
naturally attached to an $\Sp(2n,\RR)$-Higgs bundle with maximal
Toledo invariant and generalize those obtained in the case $n=2$ in
\cite{gothen:2001}.  The invariants are the first and second
Stiefel-Whitney classes $(w_1,w_2)$ of a reduction to $\O(n)$ of the
$\O(n,\CC)$-bundle defined by $(W,Q)$. It turns out that there is a
connected component for each possible value of $(w_1,w_2)$, containing
$K^2$-twisted $\GL(n,\RR)$-Higgs pairs $(W,Q,\psi)$ with
$\psi=0$. This accounts for $2.2^{2g}$ of the $3.2^{2g}$ connected
components of $\mathcal{M}_{\max}$. Thus it remains to account for the
$2^{2g}$ ``extra'' components. As already mentioned, the group
$\Sp(2n,\RR)$ is the group of isometries of a Hermitian symmetric
space, but it also has the property of being a split real form. In
fact, up to finite coverings, it is the only Lie group with this
property.  In \cite{hitchin:1992} Hitchin shows that for every
semisimple split real Lie group $G$, the moduli space of reductive
representations of $\pi_1(X)$ in $G$ has a topological component which
is isomorphic to $\RR^{\dim G (2g-2)}$, and which naturally contains
Teichm\"uller space. Indeed, when $G=\SL(2,\RR)$, this component can
be identified with Teichm\"uller space, via the Riemann uniformization
theorem. Since $\Sp(2n,\RR)$ is split, the moduli space for
$\Sp(2n,\RR)$ must have a Hitchin component. It turns out that there
are $2^{2g}$ isomorphic Hitchin components (this is actually true for
arbitrary $n$). As follows from Hitchin's construction, the
$K^2$-twisted Higgs pairs $(W,Q,\psi)$ in the Hitchin component all
have $\psi \neq 0$.

As already mentioned, in the cases $n=1$ (see Goldman
\cite{goldman:1988} and Hitchin \cite{hitchin:1987a}) and $n=2$
(see \cite{garcia-prada-mundet:2004}) the subspaces $\mathcal{R}_d$ are
connected for $0 < d < n(g-1)$. It appears natural to expect that this
should be true for general $n$. Indeed, given the analysis of the
minima of the Hitchin this functional carried out in this paper (cf.\
Theorem~\ref{thm:all-minima}), this would follow if one could prove
connectedness of the $\beta=0$ locus of $\mathcal{M}_d$. Of course,
this locus can be viewed as a moduli space of quadric bundles and as
such is a natural object to study. However, not much is known about
this question for general rank and we will leave a detailed study
for another occasion.

A second reason for our focus on maximal representations in the
present paper is that from many points of view they are the most
interesting ones. They have been the object of intense study in recent
years, using methods from diverse branches of geometry, and it has
become clear that they enjoy very special properties. In particular,
at least in many cases, maximal representations have a close
relationship to geometric structures on the surface. The prototype of
this philosophy is Goldman's theorem \cite{goldman:1980,goldman:1985}
that the maximal representations in $\SL(2,\RR)$ are exactly the
Fuchsian ones. In the following, we briefly mention some results of
this kind.

Using bounded cohomology methods, maximal representations in
general Hermitian type groups have been studied by Burger--Iozzi
\cite{burger-iozzi:2004,burger-iozzi:2007} and
Burger--Iozzi--Wienhard \cite{burger-iozzi-wienhard:2003,
burger-iozzi-wienhard:2006,burger-iozzi-wienhard:2007}. Among many
other results, they have given a very general Milnor--Wood
inequality and they have shown that maximal representations are
discrete, faithful and completely reducible. One consequence of
this is that the restriction to reductive representations is
unnecessary in the case of the moduli space $\mathcal{R}_{\max}$
of maximal representations. Building on this work and the work of
Labourie \cite{labourie:2006}, Burger--Iozzi--Labourie--Wienhard
\cite{burger-iozzi-labourie-wienhard:2005} have shown that maximal
representations in $\Sp(2n,\RR)$ are Anosov (in the sense of
\cite{labourie:2006}).  Furthermore, it has been shown that the
action of the mapping class group on $\mathcal{R}_{\max}$ is
proper, by Wienhard \cite{wienhard:2006} (for classical simple Lie
groups of Hermitian type), and by Labourie \cite{labourie:2005}
(for $\Sp(2n,\RR)$), who also proves further geometric properties
of maximal representations in $\Sp(2n,\RR)$.

From yet a different perspective, representations in the
Hitchin component have been studied in the work on higher
Teichm\"uller theory of Fock--Goncharov
\cite{fock-goncharov:2006}, using methods of tropical geometry. In
particular, the fact that representations in the Hitchin
component for $\Sp(2n,\RR)$ are faithful and discrete also follows
from their work

Thus, while Higgs bundle techniques are very efficient in the study of
topological properties of the moduli space (like counting components),
these other approaches have been more powerful in the study of special
properties of individual representations. It would be very interesting
indeed to gain a better understanding of the relation between these
distinct methods.

We describe now briefly the content of the different sections of the paper.

In Section \ref{sec:G-higgs-and-reps} we review the general theory of
$L$-twisted $G$-Higgs pairs, of which $G$-Higgs bundles are a
particular case. We explain the general Hitchin--Kobayashi
correspondence from \cite{garcia-prada-gothen-mundet:2009a} and the
corresponding non-abelian Hodge theorem. We also review some general
properties about moduli spaces and deformation theory of
$G$-Higgs bundles.

In Section \ref{spn-reps}, we specialize the non-abelian Hodge theory
correspondence of Section~\ref{sec:na-hodge} to $G=\Sp(2n,\RR)$ --- our case
of interest in this paper. We recall some basic facts about the moduli
space of $\Sp(2n,\RR)$-Higgs bundles, including the Milnor--Wood
inequality and we carry out a careful study of stable, non-simple
$\Sp(2n,\RR)$-Higgs bundles. To do this, we study and exploit the
relation between the polystability of a $\Sp(2n,\RR)$-Higgs bundles
and the $\SL(2n,\CC)$-Higgs bundle naturally associated to it.

In Section~\ref{maximal} we study the Cayley correspondence between
$\Sp(2n,\RR)$-Higgs bundles with maximal Toledo invariant and
$K^2$-twisted $\GL(n,\RR)$-Higgs pairs.

The rest of the paper is mostly devoted to the study of the
connectedness properties of the moduli space of $\Sp(2n,\RR)$-Higgs
bundles and, in particular, to prove
Theorem~\ref{connected-components-higgs}.  In
Section~\ref{sec:connected-components} we introduce the Hitchin
functional on the moduli space of $\Sp(2n,\RR)$-Higgs bundles and
characterize its minima.  We then use this and the Cayley
correspondence of Section~\ref{maximal} to count the number of
connected components of the moduli space of $\Sp(2n,\RR)$-Higgs
bundles for $d=0$ and $\abs{d}=n(g-1)$.  The proof of the
characterization of the minima is split in two cases: the case of
minima in the smooth locus of the moduli space, given in
Section~\ref{sec:stable-simple-minima} and the case of the remaining
minima, treated in Section~\ref{sec:poly-stable-minima}.

The results of this paper have been announced in several conferences
over the last few years, while several preliminary versions of
this paper have been circulating.  The main results, together with
analogous results for other groups of Hermitian type have appeared in
the review paper \cite{bradlow-garcia-prada-gothen:2005}.  The authors
apologize for having taken so long in producing this final version.

\subsubsection*{Acknowledgments}

The authors thank Olivier Biquard, Steven Bradlow,  Marc Burger, 
Bill Goldman, Nigel
Hitchin,  Alessandra Iozzi, Fran\c{c}ois Labourie, S.~Ramanan,
Roberto Rubio, Domingo Toledo, and Anna Wienhard for numerous useful
conversations and shared insights. We also thank Olivier Guichard
for carefully reading an earlier version of this paper and
pointing out a number of typos and mistakes.

%%%%%%%%%%%%%%%%%%%%%%%%%%%%%%%%%%%%%%%%%%%%%%%%%%%%%%%%%%%%%%%%%%%%%%%%%%%%%%%
\section{$L$-twisted $G$-Higgs pairs,  $G$-Higgs bundles and
  representations of surface groups}
\label{sec:G-higgs-and-reps}
%%%%%%%%%%%%%%%%%%%%%%%%%%%%%%%%%%%%%%%%%%%%%%%%%%%%%%%%%%%%%%%%%%%%%%%%%%%%%%%

%%%%%%%%%%%%%%%%%%%%%%%%%%%%%%%%%%%%
\subsection{$L$-twisted $G$-Higgs pairs, $G$-Higgs bundles,
stability  and moduli spaces}
\label{sec:g-higgs-defs}
%%%%%%%%%%%%%%%%%%%%%%%%%%%%%%%%%%%%

Let $G$ be a real reductive Lie group, let $H\subset G$ be a
maximal compact subgroup and let $\glie=\hlie\oplus\mlie$ be a
Cartan decomposition, so that the Lie algebra structure on $\glie$
satisfies
$$[\hlie,\hlie]\subset\hlie,\qquad
[\hlie,\mlie]\subset\mlie,\qquad [\mlie,\mlie]\subset\hlie.$$ The
group $H$ acts linearly on $\mlie$ through the adjoint
representation, and this action extends to a linear holomorphic
action of $\HC$ on $\mclie=\mlie\otimes\CC$. This is the \textbf{isotropy
representation:}
\begin{equation}
    \label{eq:isotropy-rep}
  \iota\colon \HC \to \GL(\mclie).
\end{equation}
Furthermore, the Killing form on $\glie$ induces
on $\mclie$ a Hermitian structure which is preserved by the action
of $H$.

Let $X$ be a compact Riemann surface and let $E$ be a holomorphic
principal $H^\CC$-bundle on $X$.
Let $E(\mclie)= E \times_{\HC}\mclie$ be the
$\mclie$-bundle associated to $E$ via the isotropy representation.
Let $K$ be the canonical bundle of $X$ and  $L$ be a holomorphic line
bundle on $X$.

\begin{definition}
  \label{def:g-higgs}
  An \textbf{$L$-twisted $G$-Higgs pair} on $X$ is a pair
  $(E,\varphi)$, where $E$ is a holomorphic $\HC$-principal bundle
  over $X$ and $\varphi$ is a holomorphic section of $E(\mclie)\otimes
  L$.  A \textbf{$G$-Higgs bundle} on $X$ is a $K$-twisted $G$-Higgs
  pair. Two $L$-twisted $G$-Higgs pairs $(E,\varphi)$ and
  $(E',\varphi')$ are \textbf{isomorphic} if there is an isomorphism
  $f\colon E \xra{\simeq} E'$ such that $\varphi = f^*\varphi'$.
\end{definition}

A general notion of (poly,semi)stability for $G$-Higgs bundles is
given in \cite{garcia-prada-gothen-mundet:2009a}.  This
notion depends on an element $\alpha\in \sqrt{-1}\liez$, where $\liez$
is the Lie algebra of the centre of $H$, and is formulated (similarly
to \cite{ramanathan:1975}) in terms of parabolic subgroups of $H^\CC$,
antidominant characters of their Lie algebras and reductions of
structure group of the $H^\CC$-bundle $E$ to the parabolic subgroups.
A workable notion can be obtained in terms of filtrations of a certain
vector bundle associated to $E$.  In particular, when $H^\mathbb{C}$
is a classical group ---which is the situation for the particular
groups considered in this paper---, this vector bundle is the one
associated to $E$ via the standard representation
$\rho:H^\mathbb{C}\rightarrow\GL(n,\mathbb{C})$.  For details, see
Sections~2.8 and 4 of \cite{garcia-prada-gothen-mundet:2009a}.

When studying $G$-Higgs bundles we shall mainly
be interested in the case when $\alpha = 0$, since this is the
relevant value for the applications to non-abelian Hodge theory. Thus
we will talk about \textbf{stability} of a $G$-Higgs bundle, meaning
$0$-stability, and analogously for \textbf{semistability} and
\textbf{polystability}.

Henceforth, we shall assume that $G$ is connected. Then the
topological classification of $H^\CC$-bundles $E$ on $X$ is given by a
characteristic class $c(E)\in H^2(X,\pi_1(H^\CC))=\pi_1(H^\CC)=\pi_1(H)=\pi_1(G)$.  For a
fixed $d \in \pi_1(G)$, the \textbf{moduli space of polystable
  $G$-Higgs bundles} $\mathcal{M}_d(G)$ is the set of isomorphism
classes of polystable $G$-Higgs bundles $(E,\varphi)$ such that
$c(E)=d$.  When $G$ is compact, the moduli space $\cM_d(G)$ coincides
with $M_d(G^\CC)$, the moduli space of polystable $G^\CC$-bundles
with topological invariant $d$.

The moduli space $\cM_d(G)$ has the structure of a complex analytic
variety.  This can be seen by the standard slice method (see, e.g.,
Kobayashi \cite{kobayashi:1987}).  Geometric Invariant Theory
constructions are available in the literature for $G$ real compact
algebraic (Ramanathan \cite{ramanathan:1996}) and for $G$ complex
reductive algebraic (Simpson \cite{simpson:1994,simpson:1995}).  The
case of a real form of a complex reductive algebraic Lie group follows
from the general constructions of Schmitt
\cite{schmitt:2005,schmitt:2008}. We thus have the following.
\begin{theorem}\label{alg-moduli}
The moduli space $\cM_d(G)$ is a complex analytic variety, which is
algebraic when $G$ is algebraic.
\end{theorem}

\begin{remark}
  More generally, moduli spaces of $L$-twisted $G$-Higgs pairs can be
  constructed (see Schmitt \cite{schmitt:2008}). We shall need this in
  Section~\ref{maximal} below.
\end{remark}

%%%%%%%%%%%%%%%%%%%%%%%%%%%%%%%%%%%%%%%%%%%%%%%%%%%%%%%%%%%%%%%%%%%%%%%%
\subsection{Surface group representations and non-abelian Hodge theory}
\label{sec:na-hodge}
%%%%%%%%%%%%%%%%%%%%%%%%%%%%%%%%%%%%%%%%%%%%%%%%%%%%%%%%%%%%%%%%%%%%%%%%

Let $X$ be a closed oriented surface of genus $g$ and let
\begin{displaymath}
  \pi_{1}(X) = \langle a_{1},b_{1}, \dotsc, a_{g},b_{g} \suchthat
  \prod_{i=1}^{g}[a_{i},b_{i}] = 1 \rangle
\end{displaymath}
be its fundamental group.  Let $G$ be a connected reductive real Lie
group and let $H \subseteq G$ be a maximal compact subgroup
By a \textbf{representation} of $\pi_1(X)$ in
$G$ we understand a homomorphism $\rho\colon \pi_1(X) \to G$.
The set of all such homomorphisms,
$\Hom(\pi_1(X),G)$, can be naturally identified with the subset
of $G^{2g}$ consisting of $2g$-tuples
$(A_{1},B_{1}\dotsc,A_{g},B_{g})$ satisfying the algebraic equation
$\prod_{i=1}^{g}[A_{i},B_{i}] = 1$.  This shows that
$\Hom(\pi_1(X),G)$ is a real analytic  variety, which is algebraic
if $G$ is algebraic.

The group $G$ acts on $\Hom(\pi_1(X),G)$ by conjugation:
\[
(g \cdot \rho)(\gamma) = g \rho(\gamma) g^{-1}
\]
for $g \in G$, $\rho \in \Hom(\pi_1(X),G)$ and $\gamma\in
\pi_1(X)$. If we restrict the action to the subspace
$\Hom^{\mathrm{red}}(\pi_1(X),G)$ consisting of reductive
representations, the orbit space is Hausdorff (see Theorem 11.4
in \cite{Ri}).  By a \textbf{reductive representation} we mean one
that composed with the adjoint representation in the Lie algebra
of $G$ decomposes as a sum of irreducible representations. If $G$
is algebraic this is equivalent to the Zariski closure of the
image of $\pi_1(X)$ in $G$ being a reductive group. (When $G$ is
compact every representation is reductive.) Define the
\emph{moduli space of representations} of $\pi_1(X)$ in $G$ to be
the orbit space
\[
\mathcal{R}(G) = \Hom^{\mathrm{red}}(\pi_1(X),G) / G. \]

One has the following (see e.g.\ Goldman \cite{goldman:1984}).

\begin{theorem}
The moduli space $\cR(G)$ has the structure of a real analytic variety, which
is algebraic if $G$ is algebraic and is a complex variety if $G$ is complex.
\end{theorem}

Given a representation $\rho\colon\pi_{1}(X) \to
G$, there is an associated flat $G$-bundle on
$X$, defined as
\begin{math}
  E_{\rho} = \widetilde{X}\times_{\rho}G
\end{math}, where $\widetilde{X} \to X$ is the universal cover and
$\pi_{1}(X)$ acts on $G$ via $\rho$.  We can then assign a topological
invariant to a representation $\rho$ given by the characteristic class
$c(\rho):=c(E_{\rho})\in \pi_1(G) \simeq \pi_1(H)$ corresponding to $E_{\rho}$.
For a fixed $d\in \pi_1(G)$, the \textbf{moduli space of reductive
  representations} $\mathcal{R}_d(G)$ with topological invariant $d$
is defined as the subvariety
\begin{equation}\label{eq:RdG}
  \mathcal{R}_d(G):=\{[\rho] \in \mathcal{R}(G) \suchthat
  c(\rho)=d\},
\end{equation}
where as usual $[\rho]$ denotes the $G$-orbit $G\cdot \rho$ of
$\rho\in\Hom^{\mathrm{red}}(\pi_1(X),G)$.

The non-abelian Hodge theorem states the following.

\begin{theorem}\label{na-Hodge}
  Let $G$ be a connected semisimple real Lie group with maximal
  compact subgroup $H \subseteq G$. Let $d \in \pi_1(G) \simeq
  \pi_1(H)$. Then there is a homeomorphism $\mathcal{R}_d(G) \simeq
  \mathcal{M}_d(G)$.
\end{theorem}

One half of this theorem is proved by solving Hitchin's equations.
To explain this, let $(E,\varphi)$ be a
$G$-Higgs bundle over a compact Riemann surface $X$. By a slight abuse
of notation, we shall denote the $C^\infty$-objects underlying $E$ and
$\varphi$ by the same symbols. In particular, the Higgs field can be
viewed as a $(1,0)$-form: $\varphi \in
\Omega^{1,0}(E(\liem^{\CC}))$. Let $\tau\colon
\Omega^{1}(E(\lieg^{\CC})) \to \Omega^{1}(E(\lieg^{\CC}))$ be the
compact conjugation of $\glie^\CC$ combined with complex conjugation
on complex $1$-forms. Given a reduction $h$ of structure group to $H$ in
the smooth $H^{\CC}$-bundle $E$, we denote by $F_h$ the curvature of
the unique connection compatible with $h$ and the holomorphic
structure on $E$.

\begin{theorem} \label{higgs-hk}
  There exists a reduction $h$ of the structure group of $E$
  from $H^\CC$ to $H$ satisfying the Hitchin equation
  $$
  F_h -[\varphi,\tau(\varphi)]= 0
  $$
  if and only if $(E,\varphi)$ is polystable.
\end{theorem}

Theorem \ref{higgs-hk} was proved by Hitchin \cite{hitchin:1987a} for
$G=\SL(2,\CC)$ and Simpson \cite{simpson:1988,simpson:1992} for an
arbitrary semisimple complex Lie group $G$.  The proof for an
arbitrary reductive real Lie group $G$ when $(E,\varphi)$ is stable is
given in \cite[Theorem~2.13]{bradlow-garcia-prada-mundet:2003}, and the general
polystable case follows as a particular case of the more general
Hitchin--Kobayashi correspondence given in
\cite{garcia-prada-gothen-mundet:2009a} (see \emph{loc.\ cit.}\ Theorem~3.21).

To explain the other half of Theorem~\ref{na-Hodge}, recall that
$\mathcal{R}_d(G)$ can be identified with the moduli space of flat
reductive connections on a fixed $G$-bundle of topological class $d
\in \pi_1(G)$ (see, e.g., Simpson
\cite{simpson:1994,simpson:1995}). Now any solution $h$ to Hitchin's
equations defines a flat reductive $G$-connection
\begin{equation}
  \label{eq:flat-hitchin}
  D = D_h + \varphi - \tau(\varphi),
\end{equation}
where $D_h$ is the
unique $H$-connection on $E$ compatible with its holomorphic
structure.

It is a fundamental theorem of Corlette \cite{corlette:1988} and
Donaldson \cite{donaldson:1987} (for $G=\SL(2,\CC)$) that this process
can be inverted: given a flat reductive connection $D$ in a $G$-bundle
$E_G$, there exists a \emph{harmonic metric}, i.e.\ a reduction of
structure group to $H \subset G$ corresponding to a harmonic section
of $E_G/H \to X$. This reduction produces a solution $h$ to Hitchin's
equations such that (\ref{eq:flat-hitchin}) holds.  (See
\cite[Section~3.6]{garcia-prada-gothen-mundet:2009a} for a fuller
explanation).

\begin{remark}
These moduli spaces may be empty.
\end{remark}

\begin{remark}
On the open subvarieties defined by the smooth
points of $\mathcal{R}_d$ and $\mathcal{M}_d$, this correspondence is in
fact an isomorphism of real analytic varieties.
\end{remark}

\begin{remark} There is a similar correspondence when $G$ is reductive
but not semisimple. In this case, it makes sense to consider
nonzero values of the stability parameter $\alpha$. The resulting
Higgs bundles can be geometrically interpreted in terms of
representations of the universal central extension of the
fundamental group of $X$, and the value of $\alpha$ prescribes the
image of a generator of the centre in the representation.
\end{remark}

%%%%%%%%%%%%%%%%%%%%%%%%%%%%%%%%%%%%%%%%%%%%%
\subsection{Deformation theory of $G$-Higgs bundles}
\label{sec:deformation-theory}
%%%%%%%%%%%%%%%%%%%%%%%%%%%%%%%%%%%%%%%%%%%%%

In this section we recall some standard facts about the deformation
theory of $G$-Higgs bundles, following Biswas--Ramanan
\cite{biswas-ramanan:1994}. The results summarized here are explained
in more detail in
\cite[Section~3.3]{garcia-prada-gothen-mundet:2009a}.

\begin{definition}\label{def:def-complex}
Let $(E,\varphi)$ be a $G$-Higgs bundle.  The \emph{deformation complex}
of $(E,\varphi)$ is the following complex of sheaves:
\begin{equation}\label{eq:def-complex}
  C^{\bullet}(E,\varphi)\colon E(\lieh^\CC) \xrightarrow{\ad(\varphi)}
  E(\liem^\CC)\otimes K.
\end{equation}
\end{definition}
This definition makes sense because $\varphi$ is a section of
$E(\lie{m}^{\CC})\otimes K$ and $[\lie{m}^{\CC},\lie{h}^{\CC}]
\subseteq \lie{m}^{\CC}$.

The following result generalizes the fact that the infinitesimal
deformation space of a holomorphic vector bundle $V$ is
isomorphic to $H^1(\End V)$.

\begin{proposition}
  \label{prop:deform}
  The space of infinitesimal deformations of a $G$-Higgs bundle
  $(E,\varphi)$ is naturally isomorphic to the hypercohomology group
  $\HH^1(C^{\bullet}(E,\varphi))$.
\end{proposition}

For any $G$-Higgs bundle there is a natural long exact sequence
\begin{equation}
  \label{eq:hyper-les}
  \begin{split}
  0 &\to \HH^0(C^{\bullet}(E,\varphi)) \to H^0 (E(\lieh^\CC))
  \xrightarrow{\ad(\varphi)} H^0(E(\liem^\CC)\otimes K) \\
  &\to \HH^1(C^{\bullet}(E,\varphi))
  \to H^1 (E(\lieh^\CC)) \xrightarrow{\ad(\varphi)}
  H^1(E(\liem^\CC)\otimes K) \to
  \HH^2(C^{\bullet}(E,\varphi)) \to 0.
  \end{split}
\end{equation}
This justifies the following definition.

\begin{definition}
  \label{def:inf-aut-space}
  The \textbf{infinitesimal automorphism space} of $(E,\varphi)$ is
  \begin{displaymath}
    \aut(E,\varphi) = \HH^0(C^{\bullet}(E,\varphi)).
  \end{displaymath}
\end{definition}
Note that this agrees with the general notion of the infinitesimal
automorphism space of a pair introduced in
\cite[Section~2.9]{garcia-prada-gothen-mundet:2009a}.

Let $d\iota\colon \lieh^{\CC} \to \End(\liem^{\CC})$ be the derivative
at the identity of the complexified isotropy representation $\iota =
\Ad_{|H^{\CC}}\colon H^{\CC} \to \Aut(\liem^{\CC})$
(cf. (\ref{eq:isotropy-rep})). Let $\ker d\iota \subseteq \lieh^{\CC}$
be its kernel and let $E(\ker d\iota) \subseteq E(\lieh^{\CC})$ be the
corresponding subbundle.  Then there is an inclusion $H^0(E(\ker
d\iota)) \into \HH^0(C^{\bullet}(E,\varphi))$.
\begin{definition}
  \label{def:i-simple}
  A $G$-Higgs bundle $(E,\varphi)$ is said to be
  \textbf{infinitesimally simple} if the infinitesimal automorphism
  space $\aut(E,\varphi)$ is isomorphic to
$ H^0(E(\ker d\iota \cap \liez))$.
\end{definition}

\begin{remark}
  \label{rem-gr-gc-ss}
  If $\ker d \iota = 0$, then $(E,\varphi)$ is infinitesimally simple
  if and only if the vanishing $\HH^0(C^{\bullet}(E,\varphi)) = 0$
  holds.  A particular case of this situation is when the group $G$
  is a complex semisimple group: indeed, in this case
  the isotropy representation is just the adjoint
  representation.
\end{remark}

Similarly, we have an inclusion $\ker\iota\cap Z(H^{\CC})\into
\Aut(E,\varphi)$.

\begin{definition}
  \label{def:simple}
  A $G$-Higgs bundle $(E,\varphi)$ is said to be
  \textbf{simple} if $\Aut(E,\varphi)=\ker \iota \cap Z(H^\CC)$,
where $Z(H^\CC)$ is the centre of $H^\CC$.
\end{definition}

Unlike the case of ordinary vector bundles, a stable $G$-Higgs bundle
is not necessarily simple. However, we have the following
infinitesimal result.

\begin{proposition}
  \label{prop:stable-simple-G-higgs}
  Any stable $G$-Higgs bundle $(E,\varphi)$ with $\varphi\neq 0$
  is infinitesimally simple.
\end{proposition}

With respect to the question of the vanishing of $\HH^2$ of the
deformation complex, we have the following useful result.

\begin{proposition}
  \label{prop:hh-vanishing}
  Let $G$ be a real semisimple group and let $G^{\CC}$ be its
  complexification.  Let $(E,\varphi)$ be a $G$-Higgs bundle which is
  stable viewed as a $G^{\CC}$-Higgs bundle.  Then the vanishing
  \begin{displaymath}
    \HH^0(C^{\bullet}(E,\varphi)) = 0 =
    \HH^2(C^{\bullet}(E,\varphi))
  \end{displaymath}
  holds.
\end{proposition}

The following result on smoothness of the moduli space can be proved,
for example, from the standard slice method construction referred to
above.
\begin{proposition}
  \label{prop:smoothness}
  Let $(E,\varphi)$ be a stable $G$-Higgs bundle. If $(E,\varphi)$ is
  simple and
$$
\HH^2(C^{\bullet}(E,\varphi))=0,
$$ then $(E,\varphi)$
  is a smooth point in the moduli space. In particular, if
  $(E,\varphi)$ is a simple $G$-Higgs bundle which is stable as a
  $G^{\CC}$-Higgs bundle, then it is a smooth point in the moduli space.
\end{proposition}

Suppose that $G$ is semisimple and $(E,\varphi)$ is stable and
simple. If a (local) universal family exists then the dimension of the
component of the moduli space containing $(E,\varphi)$ equals the
dimension of the infinitesimal deformation space
$\HH^1(C^{\bullet}(E,\varphi))$.  We shall refer to this dimension as
the \textbf{expected dimension} of the moduli space.

Moreover, since in this situation
$\HH^0(C^{\bullet}(E,\varphi))=\HH^2(C^{\bullet}(E,\varphi))=0$, the
expected dimension can be calculated from Riemann--Roch as follows.

\begin{proposition}
  \label{prop:dim-moduli}
  Let $G$ be semisimple.  Then the expected dimension of the moduli space
  of $G$-Higgs bundles is
  \begin{math}
    (g-1)\dim G^{\CC}.
  \end{math}
\end{proposition}

For a proper understanding of many aspects of the geometry of the
moduli space of Higgs bundles, one needs to consider the moduli space
as the gauge theory moduli space $\Mg_d(G)$. To define this, fix a
$C^\infty$ principal $H$-bundle $\bE_H$ with fixed topological class
$d\in \pi_1(H)$ and consider the moduli space of solutions to
\textbf{Hitchin's equations} for a pair $(A,\varphi)$ consisting of an
$H$-connection $A$ and $\varphi\in \Omega^{1,0}(X,\bE_H(\mlie^\CC))$:
\begin{equation}\label{hitchin}
\begin{array}{l}
F_A -[\varphi,\tau(\varphi)]= 0\\
\dbar_A\varphi=0.
\end{array}
\end{equation}
Here $d_A$ is the covariant derivative associated to $A$ and
$\dbar_A$ is the $(0,1)$ part of $d_A$, which defines a
holomorphic structure on $\bE_H$. The gauge group $\HHH$  of
$\bE_H$ acts on the space of solutions and the moduli space of
solutions is
$$
\Mg_d(G):= \{ (A,\varphi)\;\;\mbox{satisfying}\;\;
(\ref{hitchin})\}/\HHH.
$$
Now, as a consequence of Theorem \ref{higgs-hk}, we have
that there is a homeomorphism $\cM_d(G)\simeq \Mg_d(G)$.

This point of view is very relevant in
    particular for the Morse theoretic approach to the count of
    connected components, as explained in Section~\ref{sec:morse}. Thus
    one should consider the infinitesimal deformation space of a
    solution $(A,\varphi)$ to Hitchin's equations, which can be
    described as the first cohomology group of a certain elliptic
    deformation complex (cf.\ Hitchin~\cite{hitchin:1987a}). On the
    other hand, the formulation of the deformation theory in terms of
    hypercohomology is very convenient. Fortunately, at a smooth point
    of the moduli space, there is a natural isomorphism between the
    gauge theory deformation space and the infinitesimal deformation
    space $\HH^1(C^\bullet(E,\varphi))$ (where the holomorphic
    structure on the Higgs bundle $(E,\varphi)$ is given by
    $\dbar_A$).  As in Donaldson--Kronheimer
    \cite[\S~6.4]{donaldson-kronheimer:1990} this can be seen by using
    a Dolbeault resolution to calculate $\HH^1(C^\bullet(E,\varphi))$
    and using harmonic representatives of cohomology classes, via
    Hodge theory.  For this reason we can freely apply the complex
    deformation theory described in this Section to the gauge theory
    situation.
 % \end{remark}

%%%%%%%%%%%%%%%%%%%%%%%%%%%%%%%%%%%%%%%%%%%%%%%%%%%%%
\section{$\Sp(2n,\RR)$-Higgs bundles}
\label{spn-reps}
%%%%%%%%%%%%%%%%%%%%%%%%%%%%%%%%%%%%%%%%%%%%%%%%%%%%%

%%%%%%%%%%%%%%%%%%%%%%%%%%%%%%%%%%%%%%%%%%%%%%%%%%%%%%%%%%%%%%%%%%%%%%%%%%%%
\subsection{Stability and simplicity of $\Sp(2n,\RR)$-Higgs bundles}
\label{sec:spnr-higgs}
%%%%%%%%%%%%%%%%%%%%%%%%%%%%%%%%%%%%%%%%%%%%%%%%%%%%%%%%%%%%%%%%%%%%%%%%%%%%

Let $G=\Sp(2n,\RR)$. The maximal compact subgroup of $G$ is
$H=\U(n)$ and hence  $H^\CC=\GL(n,\CC)$. Now, if  $\VV=\CC^n$ is
the fundamental representation of $\GL(n,\CC)$, then the isotropy
representation space is:
$$
\lie{m}^\CC=S^2\VV\oplus  S^2\VV^*.
$$
Let $X$ be a compact Riemann surface.  According to
Definition~\ref{def:g-higgs}, an $\Sp(2n,\RR)$-Higgs bundle over $X$
is a triple $(V,\beta,\gamma)$ consisting of a rank $n$ holomorphic
vector bundle $V$ and holomorphic sections $\beta\in H^0(X,S^2V\otimes
K)$ and $\gamma\in H^0(X,S^2V^*\otimes K)$, where $K$ is the canonical
line bundle of $X$. Some times we denote $\varphi=\beta +\gamma$.

\begin{remark}
When $H^\CC$ is a classical group, like for $G=\Sp(2n,\RR)$,
we  prefer to work with the vector bundle $V$ associated to
the standard representation rather than the
$H^\CC$-principal bundle. It is indeed in terms of $V$ that we will
describe the stability condition as we will see below.
\end{remark}

\begin{notation}
  \label{not:symmetrized-tensor}
  Before giving a precise statement we introduce some notation. If $W$
  is a vector bundle and $W',W''\subset W$ are subbundles, then
  $W'\otimes_S W''$ denotes the subbundle of the second symmetric
  power $S^2W$ which is the image of $W'\otimes W''\subset W\otimes W$
  under the symmetrization map $W\otimes W\to S^2W$ (of course this
  should be defined in sheaf theoretical terms to be sure that
  $W'\otimes_SW''$ is indeed a subbundle, since the intersection of
  $W'\otimes W''$ and the kernel of the symmetrization map might
  change dimension from one fibre to the other). Also, we denote by
  $W'^{\perp}\subset W^*$ the kernel of the restriction map $W^*\to
  W'^*$.
\end{notation}

According to  \cite[Theorem~4.9]{garcia-prada-gothen-mundet:2009a},
the (semi)stability
condition for an $\Sp(2n,\RR)$-Higgs bundle is equivalent to the following.

\begin{definition}\label{spnr-stability}
  An $\Sp(2n,\RR)$-Higgs bundle $(V,\beta,\gamma)$ over $X$ is {\bf
    semistable} if for any filtration of subbundles
\begin{displaymath}
  0 \subset V_1 \subset V_2 \subset V
\end{displaymath}
such that
\begin{equation}
  \beta\in H^0(K\otimes(S^2V_2 + V_1\otimes_S V)),
  \quad \gamma\in H^0(K\otimes(S^2V_1^{\perp} + V_2^{\perp}\otimes_S V^*)),
  \label{eq:condicions2}
\end{equation}
we have 
\begin{equation}
  %\label{eq:spn-simplified-stab}
  \deg(V) - \deg(V_1) - \deg(V_2) \geq 0.
\end{equation}

Furthermore, $(V,\beta,\gamma)$ is {\bf stable} if for any filtration
as above, except the filtration $0=V_1\subset V_2=V$ we have
\begin{equation}
  \label{eq:spn-simplified-stab}
  \deg(V) - \deg(V_1) - \deg(V_2) > 0.
\end{equation}
\end{definition}

\begin{remark}\label{phi=0}
Note that when $\beta=\gamma = 0$, the
semistability of $(V,\beta,\gamma)$ is equivalent to the semistability of
$V$ with $\deg(V)=0$.
\end{remark}

\begin{definition}
  \label{def:beta-gamma-invariant}
  A filtration $0 \subset V_1 \subset V_2 \subset V$ satisfying the
  first condition in (\ref{eq:condicions2}) will be called
  $\beta$-{\bf invariant}; if it satisfies the second condition will
  be called $\gamma$-{\bf invariant}; and if it satisfies both will be
  said to be $(\beta,\gamma)$-{\bf invariant}, or $\varphi$-{\bf invariant},
  where $\varphi=\beta+\gamma$.
\end{definition}

The following will be useful many times below and  explains the
terminology just introduced.

\begin{remark}
\label{rem:matrix-form}
Let $V_i^{\perp}$ be the kernel of the projection
$V^*\to V_i^*$ and view $\beta$ and $\gamma$ as symmetric maps
$\beta:V^*\to K\otimes V$ and $\gamma:V\to K\otimes V^*$.
If $0\subset V_1\subset V_2\subset V$ is a filtration of vector
bundles then for any $\beta\in H^0(K\otimes S^2V)$ the condition $\beta\in H^0(K\otimes(S^2V_2 +
V_1\otimes_S V))$ is equivalent to $\beta (V_2^{\perp})\subset
K\otimes V_1$ which, by symmetry of $\beta$, is equivalent to $\beta
(V_1^{\perp})\subset K\otimes V_2$. Similarly, for any and $\gamma\in
H^0(K\otimes S^2V^*)$, the condition $\gamma\in
H^0(K\otimes(S^2V_1^{\perp} + V_2^{\perp}\otimes_S V^*))$ is
equivalent to $\gamma (V_1)\subset K\otimes V_2^{\perp}$ which, by
symmetry of $\gamma$, is equivalent to $\gamma(V_2)\subset K\otimes
V_1^{\perp}$.  Thus, if we
use a local basis of $V$ adapted to the filtration $0\subseteq V_1
\subseteq V_2 \subseteq V$ and the dual basis of $V^*$, then the
matrix of $\gamma$ is of the form
\begin{displaymath}
  \begin{pmatrix}
    0 & 0 & * \\
    0 & * & * \\
    * & * & *
  \end{pmatrix},
\end{displaymath}
while the matrix of $\beta$ has the form
\begin{displaymath}
  \begin{pmatrix}
    * & * & * \\
    * & * & 0 \\
    * & 0 &0
  \end{pmatrix}.
\end{displaymath}
\end{remark}

\begin{remark}\label{K-pairing}
There is yet another useful interpretation of (\ref{eq:condicions2})
that will be used later. To explain this, let $Q_\gamma:V\times V\to K$
be the $K$-twisted symmetric bilinear pairing defined by $\gamma$ as
$$
Q_{\gamma}(u, v):=\gamma(v)(u), \;\;\text{for}\;\; u,v\in V,
$$ 
and denote, for a subbundle  $V'\subset V$, 
$$
V'^{\perp_\gamma}:=\{v\in V\suchthat Q_\gamma(u,v)=0\;\;\mbox{for every}\;\;
u\in V'\} = \ker(V \to {V'}^*\otimes K; v\mapsto \gamma(v)_{|V'}).
$$
Then it is immediate that
\begin{equation}
  \label{eq:perp-gamma}
  \gamma(V'^{\perp_\gamma}) = V'^\perp\otimes K \subset V^*\otimes K.
\end{equation}
From this it follows that, if we have a filtration $0\subset
V_1\subset V_2\subset V$ then the condition 
$\gamma(V_1)\subset K\otimes V_2^\perp$ is equivalent to $V_1\subset
V_2^{\perp_\gamma}$. This is clearly equivalent to $V_2\subset
V_1^{\perp_\gamma}$ which, in turn, is equivalent to
$\gamma(V_2)\subset K\otimes V_1^\perp$.  Similar reasoning applies to
$\beta$.
\end{remark}

The deformation complex (\ref{eq:def-complex}) for an
$\Sp(2n,\RR)$-Higgs bundle $(V,\varphi=\beta + \gamma)$ is
\begin{equation}
  \label{eq:spnr-def-complex}
  \begin{aligned}
    C^\bullet(V,\varphi)\colon \End(V) &\xrightarrow{\ad(\varphi)} S^2 V
    \otimes K \oplus S^2V^* \otimes K \\
    \psi & \mapsto (-\beta\psi^t-\psi\beta,\gamma\psi+\psi^t\gamma)
  \end{aligned}
\end{equation}

\begin{proposition}
  \label{prop:i-simple-spnr}
  An $\Sp(2n,\RR)$-Higgs bundle $(V,\varphi)$ is infinitesimally
  simple if and only if $\HH^0(C^{\bullet}(V,\varphi)) =
  0$. Equivalently, $(V,\varphi)$ is infinitesimally simple if and
  only if there is a non-zero $\psi \in H^0(\End(V))$ such that
  \begin{displaymath}
    \ad(\varphi)(\psi) =
    (-\beta\psi^t-\psi\beta,\gamma\psi+\psi^t\gamma) = (0,0).
  \end{displaymath}
\end{proposition}

\begin{proof}
  For $\Sp(2n,\RR)$-Higgs bundles one has that $\ker(d\iota) =
  0$. Thus the first statement is immediate from
  Definition~\ref{def:i-simple}. The equivalent statement now follows
  from the long exact sequence (\ref{eq:hyper-les}), recalling that
  in this case the deformation complex (\ref{eq:def-complex}) is given
  by (\ref{eq:spnr-def-complex}).
\end{proof}

\begin{proposition}
  \label{prop:simple-spnr}
  An $\Sp(2n,\RR)$-Higgs bundle $(V,\varphi)$ is simple if and only if
  $\Aut(V,\varphi) = \{\pm \Id \}$.
\end{proposition}

\begin{proof}
  Since $\lambda \in \CC^* = Z(H^{\CC})$ acts
  on the isotropy representation $\lie{m}^\CC=S^2\VV\oplus S^2\VV^*$
  by $(\beta,\gamma) \mapsto (\lambda^2\beta,\lambda^{-2}\gamma)$
  we have $\ker \iota\cap Z(H^{\CC})=\{\pm 1\}$, so the statement
  follows directly from Definition~\ref{def:simple}.
\end{proof}

%%%%%%%%%%%%%%%%%%%%%%%%%%%%%%%%%%%%%%%%%%%%%%%%%%%%%%%%%%%%%%%%%%%%%%%%%%%%
\subsection{Stable and non-simple $\Sp(2n,\RR)$-Higgs
  bundles}
\label{sec:stable-non-simple-spnr-higgs}
%%%%%%%%%%%%%%%%%%%%%%%%%%%%%%%%%%%%%%%%%%%%%%%%%%%%%%%%%%%%%%%%%%%%%%%%%%%%

The goal of this section is to obtain a complete understanding of how
a stable $\Sp(2n,\RR)$-Higgs bundle can fail to be simple. The main
result is Theorem~\ref{thm:stable-not-simple-spnr-higgs}.
In order to state this theorem,  we need to describe some special
$\Sp(2n,\RR)$-Higgs bundles arising from $G$-Higgs bundles associated
to certain real subgroups $G \subseteq \Sp(2n,\RR)$.

\subsubsection*{The subgroup $G = \U(n)$}

Observe that a $\U(n)$-Higgs bundle is nothing but a holomorphic
vector bundle $V$ of rank $n$. The standard inclusion
$\upsilon^{\U(n)}\colon\U(n) \into \Sp(2n,\RR)$ gives the
correspondence
\begin{equation}
  \label{eq:un-sp2n}
  V \mapsto \upsilon^{\U(n)}_*V = (V,0)
\end{equation}
associating the $\Sp(2n,\RR)$-Higgs bundle $\upsilon^{\U(n)}_*V =
(V,0)$ to the holomorphic vector bundle $V$.

\subsubsection*{The subgroup $G = \U(p,q)$}

In the following we assume that $p,q \geq 1$. As is easily seen, a
$\U(p,q)$-Higgs bundle (cf.\
\cite[Definition~3.3]{bradlow-garcia-prada-gothen:2003}) is given by the data
$(\tilde{V},\tilde{W},\tilde\varphi=\tilde\beta+\tilde\gamma)$,
where $\tilde{V}$ and $\tilde{W}$ are holomorphic vector bundles
of rank $p$ and $q$, respectively, $\tilde\beta\in H^0(K\otimes
\Hom(\tilde{W},\tilde{V}))$ and $\tilde\gamma\in H^0(K\otimes
\Hom(\tilde{V},\tilde{W}))$. Let $n=p+q$. The imaginary part of
the standard indefinite Hermitian metric of signature $(p,q)$ on
$\CC^n$ is a symplectic form, and thus there is an inclusion
$\upsilon^{\U(p,q)}\colon\U(p,q) \into \Sp(2n,\RR)$. At the level
of $G$-Higgs bundles, this gives rise to the correspondence
\begin{equation}
\label{eq:upq-sp2n}
(\tilde{V},\tilde{W},\tilde\varphi=\tilde\beta+\tilde\gamma)
\mapsto \upsilon^{\U(p,q)}_*(\tilde{V},\tilde{W},\tilde\varphi)
= (V, \varphi=\beta+\gamma),
\end{equation}
where
\begin{displaymath}
    V = \tilde{V} \oplus \tilde{W}^*,\quad
    \beta =
    \begin{pmatrix}
      0 & \tilde{\beta} \\
      \tilde\beta & 0
    \end{pmatrix}\quad\text{and}\quad
    \gamma =
    \begin{pmatrix}
      0 & \tilde\gamma \\
      \tilde\gamma & 0
    \end{pmatrix}.
\end{displaymath}

In the following we shall occasionally slightly abuse language, saying
simply that $\upsilon^{\U(n)}_*V$ is a $\U(n)$-Higgs bundle and that
$\upsilon^{\U(p,q)}_*(\tilde{V},\tilde{W},\tilde\varphi)$ is a
$\U(p,q)$-Higgs bundle.

Another piece of convenient notation is the following. Let
$(V_i,\varphi_i)$ be $\Sp(2n_i,\RR)$-Higgs bundles and let $n = \sum
n_i$. We can define an $\Sp(2n,\RR)$-Higgs bundle $(V,\varphi)$ by
setting
\begin{displaymath}
  V = \bigoplus V_i \quad\text{and}\quad
  \varphi = \sum \varphi_i
\end{displaymath}
by using the canonical inclusions $H^0(K \otimes (S^2V_i \oplus
S^2V_i^*)) \subset H^0(K \otimes (S^2V \oplus S^2V^*))$. We shall
slightly abuse language and write $(V,\varphi) = \bigoplus
(V_i,\varphi_i)$, referring to this as \textbf{the direct sum} of the
$(V_i,\varphi_i)$.

\begin{remark}
  \label{rem:V-non-simple}
  Note that $\upsilon^{\U(n)}_*V = (V,0)$  is never simple as an
  $\Sp(2n,\RR)$-Higgs bundle, since its automorphism group contains
  the non-zero scalars $\CC^*$.
  Similarly, the $\Sp(2n,\RR)$-Higgs bundle
  $\upsilon^{\U(p,q)}_*(\tilde{V},\tilde{W},\tilde\varphi)$ is not
  simple, since it has the automorphism $ \left(
    \begin{smallmatrix}
      1 & 0 \\
      0 & -1
    \end{smallmatrix}
  \right)
  $.
\end{remark}

We shall need a few lemmas for the proof of
Theorem~\ref{thm:stable-not-simple-spnr-higgs}.
\begin{lemma}
  \label{lem:summand-nonstable}
  Let $(V,\varphi)$ be an $\Sp(2n,\RR)$-Higgs bundle and assume that
  there is a non-trivial splitting $(V,\varphi) = (V_a \oplus
  V_b,\varphi_a + \varphi_b)$ such that $\varphi_{\nu} \in H^0(K \otimes
  (S^2V_\nu \oplus S^2V_\nu^*))$ for $\nu=a,b$.  Assume that the
  $\Sp(2n_a,\RR)$-Higgs bundle $(V_a,\varphi_a)$ is not stable. Then
  $(V,\varphi)$ is not stable.
\end{lemma}

\begin{proof}
  Since $(V_a,\varphi_a)$ is not stable there is a filtration $0
  \subset V_{a1} \subset V_{a2} \subset V_a$ such that
  \begin{displaymath}
  \beta\in H^0(K\otimes(S^2V_{a2} + V_{a1}\otimes_S V)),
  \quad \gamma\in H^0(K\otimes(S^2V_{a1}^{\perp} + V_{a2}^{\perp}\otimes_S V^*))
  \end{displaymath}
  and
  \begin{equation}
    \label{eq:degVa}
    \deg(V_a) - \deg(V_{a1}) - \deg(V_{a1}) \leq 0.
  \end{equation}
  Consider the filtration $0 \subset V_1 \subset V_2 \subset V$
  obtained by setting
  \begin{displaymath}
    V_1 = V_{a1},\quad V_2 = V_{a2} \oplus V_b.
  \end{displaymath}
  Using Remark~\ref{rem:matrix-form} one readily sees that this
  filtration satisfies the conditions (\ref{eq:condicions2}). Since
  \begin{displaymath}
    \deg(V) - \deg(V_1) - \deg(V_2) = \deg(V_a) - \deg(V_{a1}) - \deg(V_{a1}),
  \end{displaymath}
  it follows from (\ref{eq:degVa}) that $(V,\varphi)$ is not stable.
\end{proof}

\begin{lemma}
  \label{lem:gln-summand-nonstable}
  Let $(V,\varphi)$ be an $\Sp(2n,\RR)$-Higgs bundle and assume that
  there is a non-trivial splitting $V = V_a \oplus V_b$ such that
  $\varphi \in H^0(K \otimes (S^2V_a \oplus S^2V_a^*))$. In other
  words, $(V,\varphi) = (V_a\oplus V_b, \varphi_a + 0)$ with $(V_b,0) =
  \upsilon_*^{\U(n_b)}V_b$.  Then $(V,\varphi)$ is not stable.
\end{lemma}

\begin{proof}
  It is immediate from Lemma~\ref{lem:summand-nonstable} and
  Remark~\ref{phi=0} that $V_b$ is a stable vector bundle with
  $\deg(V_b) = 0$. Hence
  \begin{displaymath}
    \deg(V) = \deg(V_a).
  \end{displaymath}
  Consider the filtration $0 \subset V_1
  \subset V_2 \subset V$ obtained by setting $V_1 = 0$ and $V_2 = V_{a}$.
  As before this filtration satisfies
  (\ref{eq:condicions2}). Therefore the calculation
  \begin{displaymath}
    \deg(V) - \deg (V_1) -\deg(V_2) = \deg(V) - \deg(V_a) = 0
  \end{displaymath}
  shows that $(V,\varphi)$ is not stable.
\end{proof}

\begin{lemma}
  \label{lem:upq-nonstable}
  Let $(V,\varphi) = \upsilon^{\U(p,q)}_*(V_a, V_b^*, \tilde
  \varphi)$ be an $\Sp(2n,\RR)$-Higgs bundle arising from a
  $\U(p,q)$-Higgs bundle $(V_a, V_b^*, \tilde
  \varphi)$ with $p,q \geq 1$. Then $(V,\varphi)$ is not
  stable.
\end{lemma}

\begin{proof}
  The $\Sp(2n,\RR)$-Higgs bundle $(V,\varphi)$ is given by
  \begin{displaymath}
    V = V_a \oplus V_b,\quad
    \beta =
    \begin{pmatrix}
      0 & \tilde\beta \\
      \tilde\beta & 0
    \end{pmatrix}
    \quad\text{and}\quad
    \gamma =
    \begin{pmatrix}
      0 & \tilde\gamma \\
      \tilde\gamma & 0
    \end{pmatrix}.
  \end{displaymath}
  Let $V_1 = V_2 = V_a$ and consider the filtration
  \begin{math}
    0 \subset V_1 \subset V_2 \subset V.
  \end{math}
  Again this filtration satisfies the conditions
  (\ref{eq:condicions2}). Thus, if $(V,\varphi)$ is stable, we have
  from (\ref{eq:spn-simplified-stab})
  \begin{displaymath}
    \deg(V) - 2\deg(V_a) < 0.
  \end{displaymath}
  Similarly, considering $V_1 = V_2 = V_b$, we obtain
  \begin{displaymath}
    \deg(V) - 2\deg(V_b) < 0,
  \end{displaymath}
  so we conclude that
  \begin{displaymath}
    \deg(V) = \deg(V_a) + \deg(V_b) < \deg(V),
  \end{displaymath}
  which is absurd.

\end{proof}

\begin{lemma}
  \label{lem:iso-unstable}
  Let $(\tilde V,\tilde\varphi)$ be an $\Sp(2\tilde{n},\RR)$-Higgs
  bundle. Then the $\Sp(4\tilde n,\RR)$-Higgs bundle $(\tilde V \oplus
  \tilde V, \tilde\varphi + \tilde\varphi)$ is not stable.
\end{lemma}
\begin{proof}
  Consider the automorphism
  $
    f =
    \frac{1}{\sqrt{2\mathbf{i}}}
    \left(
      \begin{smallmatrix}
      1 & \mathbf{i} \\
      \mathbf{i} & 1
    \end{smallmatrix}
    \right)
  $
  of $V = \tilde V \oplus \tilde V$. Write
  $\beta = \left(
    \begin{smallmatrix}
      \tilde\beta & 0 \\
      0 & \tilde\beta
    \end{smallmatrix}
    \right)$
  and
  $\gamma = \left(
    \begin{smallmatrix}
      \tilde\gamma & 0 \\
      0 & \tilde\gamma
    \end{smallmatrix}
  \right)$.
  Then we have that
  $$
  (V,\varphi) \simeq (\tilde V \oplus \tilde
  V,f\cdot\beta+f\cdot\gamma),
  $$
  where
  \begin{displaymath}
      f \cdot \beta = f\beta f^t =
      \begin{pmatrix}
        0 & \tilde\beta \\
        \tilde\beta & 0
      \end{pmatrix}\quad\text{and}\quad
      f \cdot\gamma  = (f^t)^{-1}\gamma f^{-1} =
      \begin{pmatrix}
        0 & \tilde\gamma \\
        \tilde\gamma & 0
      \end{pmatrix}.
  \end{displaymath}
  We shall see that $(\tilde V \oplus \tilde
  V,f\cdot\beta+f\cdot\gamma)$ is not stable. To this end, consider
  the filtration $0 \subset V_1 \subset V_2 \subset \tilde V \oplus
  \tilde V$ obtained by setting $V_1 = V_2 = \tilde V$. This satisfies
  (\ref{eq:condicions2}). But, on the other hand,
  \begin{displaymath}
    \deg(\tilde V \oplus \tilde V) - \deg(V_1) - \deg(V_2) = 0
  \end{displaymath}
  so $(\tilde V \oplus \tilde V,f\cdot\beta+f\cdot\gamma)$ is not
  stable.
\end{proof}

\begin{theorem}
  \label{thm:stable-not-simple-spnr-higgs}
  Let $(V,\varphi)$ be a stable $\Sp(2n,\RR)$-Higgs bundle. If
  $(V,\varphi)$ is not simple, then one of the following alternatives
  occurs:
  \begin{enumerate}
  \item The vanishing $\varphi=0$ holds and $V$ is a stable vector
    bundle of degree zero. In this case, $\Aut(V,\varphi) \simeq \CC^*$.
    \label{item:phi-eq-0}
  \item There is a nontrivial decomposition, unique up to reordering,
    $$(V,\varphi) = (\bigoplus_{i=1}^{k} V_i,\sum_{i=1}^{k}
    \varphi_i)$$
    with $\varphi_i = \beta_i + \gamma_i \in H^0(K \otimes (S^2V_i \oplus
    S^2V_i^*))$, such that each $(V_i,\varphi_i)$ is a stable and simple
    $\Sp(2n_i,\RR)$-Higgs bundle. Furthermore, each $\varphi_i \neq 0$
    and $(V_i,\varphi_i) \not\simeq (V_j,\varphi_j)$ for $i \neq j$.
    The automorphism group of $(V,\varphi)$ is
    \begin{displaymath}
      \Aut(V,\varphi) \simeq \Aut(V_1,\varphi_1) \times \dots \times
      \Aut(V_k,\varphi_k) \simeq (\ZZ/2)^k.
    \end{displaymath}
    \label{item:phi-neq-0}
  \end{enumerate}
\end{theorem}

\begin{proof}
  First of all, we note that if $\varphi = 0$ then it is immediate
  from Remark~\ref{phi=0} that alternative
  (\ref{item:phi-eq-0}) occurs.

  Next, consider the case $\varphi\neq 0$. Since $(V,\varphi)$ is not
  simple, there is an automorphism $\sigma\in
  \Aut(V,\varphi)\setminus\{\pm 1\}$.  We know from Lemma~2.25 in
  \cite{garcia-prada-gothen-mundet:2009a} that $\Aut(V,\varphi)$ is
  reductive. This implies that $\sigma$ may be chosen to be
  semisimple, so that there is a splitting $V = \bigoplus V_i$ in
  eigenbundles of $\sigma$ such that the action of $\sigma$ on $V_i$
  is given by multiplication by some $\sigma_i \in \CC^*$.  If
  $\sigma$ were a multiple of the identity, say $\sigma=\lambda\Id$
  with $\lambda\in \CC^*$, then it would act on $\varphi =
  \beta+\gamma$ by $\beta\mapsto \lambda^2\beta$ and $\gamma\mapsto
  \lambda^{-2}\gamma$. Since $\varphi\neq 0$ this would force $\sigma$
  to be equal to $1$ or $-1$, in contradiction with our choice.  Hence
  $\sigma$ is not a multiple of the identity, so the decomposition $V
  = \bigoplus V_i$ has more than one summand.  The action of $\sigma$
  on $S^2V \oplus S^2V^*$ is given by
  \begin{equation}\label{eq:sigma-ij-1}
    \sigma = \sigma_i\sigma_j\colon V_i\otimes V_j \to V_i\otimes V_j
    \qquad\text{and}\qquad
    \sigma = \sigma_i^{-1}\sigma_j^{-j}\colon V_i^*\otimes V_j^* \to
    V_i^*\otimes V_j^*.
  \end{equation}
  If we denote by $\varphi_{ij}= \beta_{ij} + \gamma_{ij}$ the
  component of $\varphi$ in $H^0(K \otimes (V_i \otimes V_j \oplus V^*_i
  \otimes V^*_j))$ (symmetrizing the tensor product if $i=j$), then
  \begin{equation}
    \label{eq:sigma-ij}
    \sigma_i\sigma_j \neq 1 \implies \varphi_{ij} = 0.
  \end{equation}

  Suppose that $\varphi_{i_0j_0} \neq 0$ for some $i_0 \neq j_0$. From
  (\ref{eq:sigma-ij}) we conclude that $\sigma_{i_0}\sigma_{j_0} =
  1$. But then $\sigma_{i}\sigma_{j_0} \neq 1$ for $i \neq i_0$ and
  $\sigma_{i_0}\sigma_{j} \neq 1$ for $j \neq j_0$. Hence, again by
  (\ref{eq:sigma-ij}), $\varphi_{ij_0} = 0 = \varphi_{i_0j}$ if $i\neq
  i_0$ or $j \neq j_0$. Thus $(V_{i_0}, V_{j_0}^*,\varphi_{i_0j_0})$
  is a $\U(p,q)$-Higgs bundle and we have a non-trivial decomposition
  $(V,\varphi) = (V_a\oplus V_b,\varphi_a + \varphi_b)$ with
  $(V_a,\varphi_a) = \upsilon_*^{\U(p,q)}(V_{i_0},
  V_{j_0}^*,\varphi_{i_0j_0})$. By Lemma~\ref{lem:upq-nonstable}
  the $\Sp(2n_a,\RR)$-Higgs bundle $(V_a,\varphi_a)$ is not stable so, by
  Lemma~\ref{lem:summand-nonstable}, $(V,\varphi)$ is not stable.
  This contradiction shows that $\varphi_{ij} = 0$ for $i\neq j$.

  It follows that $\varphi = \sum \varphi_i$ with $\varphi_i \in H^0(K
  \otimes (S^2V_i \oplus S^2V_i^*))$. By
  Lemma~\ref{lem:summand-nonstable} each of the summands
  $(V_i,\varphi_i)$ is a stable $\Sp(2n,\RR)$-Higgs bundle and by
  Lemma~\ref{lem:gln-summand-nonstable} each $\varphi_i$ must be
  non-zero. Also, from (\ref{eq:sigma-ij-1}), $\sigma\cdot\beta_i =
  \sigma_i^2\beta_i $ and $\sigma\cdot\gamma_i = \sigma_i^{-2}\gamma_i$ so
  we conclude that the only possible eigenvalues of $\sigma$ are $1$
  and $-1$. Thus the decomposition $(V,\varphi) = \bigoplus
  (V_i,\varphi_i)$ has in fact only two summands and, more
  importantly, $\sigma^2 = 1$. This means that all non-trivial
  elements of $\Aut(V,\varphi)$ have order two and therefore
  $\Aut(V,\varphi)$ is abelian (indeed: if $\sigma,\tau\in\Aut(V,\varphi)$
  then we have $\sigma^2=\tau^2=(\tau\sigma)^2=1$, but
  $(\tau\sigma)^2=\tau\sigma\tau\sigma=1$ implies, multiplying
  both sides on the left by $\tau$ and then by $\sigma$, that
  $\tau\sigma=\sigma\tau$).

  Now, the summands $(V_i,\varphi_i)$ may not be simple but, applying
  the preceding argument inductively to each of the $(V_i,\varphi_i)$,
  we eventually obtain a decomposition $(V,\varphi) = (\bigoplus V_i,
  \sum \varphi_i)$ where each $(V_i,\varphi_i)$ is stable and simple,
  and $\varphi_i \neq 0$. Since $\Aut(V,\varphi)$ is abelian, the
  successive decompositions of $V$ in eigenspaces can in fact be
  carried out simultaneously for all $\sigma\in \Aut(V,\varphi)
  \setminus \{\pm 1 \}$. From this the uniqueness of the decomposition
  and the statement about the automorphism group of $(V,\varphi)$ are
  immediate.

  Finally, Lemma~\ref{lem:gln-summand-nonstable} and
  Lemma~\ref{lem:iso-unstable} together imply that the
  $(V_i,\varphi_i)$ are mutually non-isomorphic.
\end{proof}

%%%%%%%%%%%%%%%%%%%%%%%%%%%%%%%%%%%%%%%%%%%%%%%%%%%%
\subsection{Polystable  $\Sp(2n,\RR)$-Higgs bundles}
\label{sec:polystability}
%%%%%%%%%%%%%%%%%%%%%%%%%%%%%%%%%%%%%%%%%%%%%%%%%%%%

According to  \cite[Theorem~4.9]{garcia-prada-gothen-mundet:2009a} the
polystability condition for an
$\Sp(2n,\RR)$-Higgs bundle is equivalent to the following.

\begin{definition}\label{spn-poly}
  A $\Sp(2n,\RR)$-Higgs bundle $(V,\varphi)$ is {\bf polystable} if is
  semistable and for any filtration of subbundles
$$
  0 \subset V_1 \subset V_2 \subset V
$$
distinct from the filtration $0=V_1\subset V_2=V$
such that
\begin{equation}
  \beta\in H^0(K\otimes(S^2V_2 + V_1\otimes_S V)),
  \quad \gamma\in H^0(K\otimes(S^2V_1^{\perp} + V_2^{\perp}\otimes_S V^*)),
\end{equation}
and
$$
  \deg(V) - \deg(V_1) - \deg(V_2) = 0,
$$
there exists an isomorphism of holomorphic vector bundles

$$
\sigma: V\to  V_1\oplus V_2/V_1\oplus V/V_2
$$
satisfying the following properties:

(a) $V_1=\sigma^{-1}(V_1)$, $V_2=\sigma^{-1}(V_1\oplus V_2/V_1)$,

(b)
$ \beta\in H^0(K\otimes 
(S^2(\sigma^{-1}(V_2/V_1)) \oplus
\sigma^{-1}(V_1)\otimes_S\sigma^{-1}(V/V_2)))$,

(c) $ \gamma\in H^0(K\otimes
(S^2(\sigma^*(V_2/V_1)^*)\oplus
\sigma^*(V_1^*)\otimes_S\sigma^*(V/V_2)^*)))$.
\end{definition}

\begin{remark}\label{f-decomposition}
  The polystability condition can be rephrased by saying that there
  are subbundles $F_1$, $F_2$ and $F_3$ of $V$, such that $V=F_1\oplus
  F_2\oplus F_3$, with $V_1=F_1$, $V_2=F_1\oplus F_2$, and
  $\beta(F_i^*)\subset F_{4-i}\otimes K$ and $\gamma(F_i)\subset
  F_{4-i}^*\otimes K$ for $i=1,2,3$.
\end{remark}

We have the following result
\cite[Proposition 4.14]{garcia-prada-gothen-mundet:2009a} on the structure of
polystable $\Sp(2n,\RR)$-Higgs bundles.

\begin{proposition}
  \label{prop:polystable-spnr-higgs}
A $\Sp(2n,\RR)$-Higgs bundle $(V,\varphi)$ is  polystable if and
only if there is a decomposition
  $$(
  V,\varphi) = (V_1,\varphi_1) \oplus \dots \oplus (V_k,\varphi_k),
  $$ unique up to reordering, such that each
  $(V_i,\varphi_i)$ is a stable and simple $G_i$-Higgs bundle
  of one of the following mutually exclusive types:

(1) $G_i=\Sp(2n_i,\RR)$ and non-vanishing Higgs field;

(2)  $G_i=\U(p_i,q_i)$ with $\deg \tilde{V_i} + \deg  \tilde{W_i}=0$,
where $\tilde{V_i}$ and    $\tilde{W_i}  $ are defined by (\ref{eq:upq-sp2n});

(3) $G_i=\U(n_i)$  and  $\deg(V_i)=0$.

\end{proposition}

%%%%%%%%%%%%%%%%%%%%%%%%%%%%%%%%%%%%%%%%%%%%%%%%%
\subsection{$L$-twisted $\GL(n,\RR)$-Higgs pairs}
\label{glnr-higgs}
%%%%%%%%%%%%%%%%%%%%%%%%%%%%%%%%%%%%%%%%%%%%%%%%%

We study  now $L$-twisted $G$-Higgs pairs for $G=\GL(n,\RR)$.
They will appear for $L=K$ in Section \ref{sec:spnr-sl-higgs}.
When
$L=K^2$, these will be related to maximal degree
$\Sp(2n,\RR)$-Higgs bundles as we will see in Section \ref{maximal}.

A maximal compact subgroup of $\GL(n,\RR)$ is $H=\O(n)$ and hence
$H^\CC=\O(n,\CC)$. Now, if  $\WW$ is the standard $n$-dimensional
complex vector  space representation of $\O(n,\CC)$, and $Q$ is the 
bilinear form defining $\O(n,\CC)$,  then the
isotropy representation space is:
$$\liem^{\CC}=S^2_{Q}\mathbb{W}:=\{\xi\in\End(\mathbb{W})\,|\,Q(\xi\cdot,\cdot)=
Q(\cdot,\xi\cdot)\}\subset\End(\mathbb{W}).$$ 
Given a $\O(n,\CC)$-bundle  $(W,Q)$, denote by $S^{2}_{Q}W$ the bundle
of endomorphisms $\xi$ of $W$ which are symmetric with respect to $Q$ i.e.\ 
such that $Q(\xi\,\cdot,\cdot)=Q(\cdot,\xi\,\cdot)$.
An $L$-twisted $\GL(n,\RR)$-Higgs pair over $X$  is thus a triple
$(W,Q,\psi)$ consisting of a holomorphic $\O(n,\CC)$-bundle,
i.e.\ a rank $n$ holomorphic vector bundle $W$ over $X$ equipped
with a non-degenerate quadratic form $Q$, and a section
$$
\psi \in H^0(L\otimes S^2_{Q}W).
$$
Note that when $\psi = 0$ a twisted $\GL(n,\RR)$-Higgs  pair is
simply  an orthogonal bundle.

\begin{remark}
Given a $\GL(n,\RR)$-Higgs pair  $(W,Q,\psi)$, we can consider the symmetric
isomorphism  $q:W\stackrel{\simeq}{\longrightarrow}W^*$ given by $q(v)=Q(v,\cdot)$
and define the homomorphism $$\tilde\psi:W^*\longrightarrow W\otimes L$$  by
\begin{displaymath}%\label{tildepsi}
\tilde\psi:=\psi q^{-1}.
\end{displaymath}
It follows from the symmetry of $\psi$ with respect to $Q$ that
$\tilde\psi$ is 
symmetric i.e. $\tilde\psi^t\otimes I_L=\tilde\psi$.
In other words, $$\tilde\psi\in H^0(S^{2}W\otimes L).$$
We can thus regard a $\GL(n,\RR)$-Higgs bundle as a triple $(W,Q,\tilde\psi)$, 
where $(W,Q)$ is a $\O(n,\CC)$-bundle and $\tilde\psi\in H^0(S^{2}W\otimes
L)$.
\end{remark}

\begin{remark}
Since the centre of $\olie(n)$ is trivial, $\alpha=0$ is the only
possible value for which stability of an $L$-twisted
$\GL(n,\RR)$-Higgs pair  is defined.
\end{remark}

According to \cite[Theorem 4.17]{garcia-prada-gothen-mundet:2009a}
the stability conditions (for $\alpha=0$) for an
$L$-twisted $\GL(n,\RR)$-Higgs pair can be defined as follows.

\begin{definition}\label{thm:orthogonal-stability}
  Let $(W,Q,\psi)$ be a $L$-twisted  $\GL(n,\RR)$-Higgs pair. Then 
$(W,Q,\psi)$ is 
{\bf semistable} if and only if $  \deg(W')\leq 0$ for any isotropic and 
$\psi$-invariant subbundle $W'\subset W$. Furthermore,
$(W,Q,\psi)$ is  
{\bf stable} if and only if it is semistable and $\deg(W')<0$ for any isotropic and 
$\psi$-invariant strict subbundle $0\neq W' \subset W$.
Finally, $(W,Q,\psi)$ is  
 {\bf polystable} if and only if it is semistable and, for any isotropic 
 (respectively.\ coisotropic) and $\psi$-invariant strict
  subbundle $0\neq W'\subset W$ such that $\deg(W')=0$, there is another
  coisotropic (resp.\ isotropic) and $\psi$-invariant subbundle 
$0\neq W''\subset W$ such that $W\simeq W'\oplus W''$.
\end{definition}

\begin{remark}
  \label{rem:iso-co-iso}
  Let $W' \subset W$ be a subbundle and let ${W'}^{\perp_Q}$ be its
  orthogonal complement with respect to $Q$. Then $W'$ is isotropic if
  and only if ${W'}^{\perp_Q}$ is coisotropic. Moreover, since $\psi$
  is symmetric with respect to $Q$, the subbundle $W'$ is
  $\psi$-invariant if and only if ${W'}^{\perp_Q}$ is
  $\psi$-invariant. 
\end{remark}

\begin{proposition}
  \label{prop:GL-poly-decomposition}
  Let $(W,Q,\psi)$ be a polystable $L$-twisted $\GL(n,\RR)$-Higgs pair
  and let $W' \subset W$ be an isotropic $\psi$-invariant subbundle
  such that $\deg(W') = 0$. Then there is a decomposition
  \begin{displaymath}
    W \simeq W' \oplus W_2 \oplus \tilde{W}'
  \end{displaymath}
  in which
  \begin{enumerate}
  \item each of the subbundles $W'$, $W_2$ and $\tilde{W}'$ is
    $\psi$-invariant,
  \item the quadratic form $Q$ restricted to $W_2$ is non-degenerate, and
  \item the subbundles $W'$ and $\tilde{W}'$ are isotropic and in
    duality with each other under $Q$.
  \end{enumerate}
  Moreover, if we let $W_1 = W' \oplus \tilde{W}'$ and, for $i=1,2$,
  $\psi_i = \psi_{|W_i}$ and $Q_i = Q_{|W_i}$, then the $(W_i,Q_i,\psi_i)$
  are polystable $L$-twisted $\GL(n_i,\RR)$-Higgs pairs and
  \begin{displaymath}
    (W,Q,\psi) = (W_1 \oplus W_2, Q_1 + Q_2, \psi_1 + \psi_2).
  \end{displaymath}
\end{proposition}

\begin{proof}
  We have from Remark~\ref{rem:iso-co-iso} that $W'^{\perp_Q}$ is
  coisotropic and $\psi$-invariant. Thus, since $(W,Q,\psi)$ is
  polystable, $W'^{\perp_Q}$ has an isotropic $\psi$-invariant
  complement $\tilde{W}'$. Clearly, the subbundles $W'$ and
  $\tilde{W}'$ are isotropic and in duality with each other under $Q$,
  so that the restriction $Q_1 = Q_{|W_1}$ is non-degenerate (where
  $W_1 = W' \oplus \tilde{W}'$). Now define $W_2 = (W' \oplus
  W'^{\perp_Q})^{\perp_Q}$. Then, clearly, the restriction $Q_2 =
  Q_{|W_2}$ is non-degenerate and $\psi$-invariant.

  For the final statement, it only remains to prove polystability of
  $(W_i,Q_i,\psi_i)$. But this follows from polystability of
  $(W,Q,\psi)$, noting the following facts: firstly, that any
  $\psi$-invariant isotropic subbundle of $(W_i,Q_i,\psi_i)$ is also a
  $\psi$-invariant isotropic subbundle of $(W,Q,\psi)$ and, secondly,
  that any $\psi$-invariant coisotropic subbundle of
  $(W_i,Q_i,\psi_i)$ extends naturally to a $\psi$-invariant
  coisotropic subbundle of $(W,Q,\psi)$.
\end{proof}

%%%%%%%%%%%%%%%%%%%%%%%%%%%%%%%%%%%%%%%%%%%%%%%%%%%%%%%%%%%%%%%%%%%%%%%%%%%%
\subsection{Comparison of stability conditions for  $\Sp(2n,\RR)$-,
$\Sp(2n,\CC)$-  and  $\SL(2n,\CC)$-Higgs
  bundles}
\label{sec:spnr-sl-higgs}
%%%%%%%%%%%%%%%%%%%%%%%%%%%%%%%%%%%%%%%%%%%%%%%%%%%%%%%%%%%%%%%%%%%%%%%%%%%%

An $\Sp(2n,\RR)$-Higgs bundle can be viewed as a Higgs bundle for the
larger complex groups $\Sp(2n,\CC)$ and $\SL(2n,\CC)$. The goal of
this section is to understand the relation between the various
corresponding stability notions. The main results are
Theorems~\ref{thm:stability-equivalence} and
\ref{thm:stability-equivalence-sp2nC} below.

If $G=\SL(n,\CC)$ then the maximal compact subgroup of $G$ is
$H=\SU(n)$ and hence  $H^\CC$ coincides with $\SL(n,\CC)$. Now, if
$\WW=\CC^{n}$ is the fundamental representation of $\SL(n,\CC)$,
the isotropy representation space is given by the traceless
endomorphisms of $\WW$
$$
\lie{m}^\CC=\sllie(\WW)=\{\xi\in\End(\WW)\mid \Tr\xi
=0\}\subset\End\WW,
$$
so it coincides again with the adjoint representation of
$\SL(n,\CC)$ on its Lie algebra. An  $\SL(n,\CC)$-Higgs
bundle is thus a pair consisting of a rank $n$ holomorphic vector
bundle $W$ over $X$ endowed with a trivialization $\det
W\simeq\OOO$ and a holomorphic section
$$
\Phi \in H^0(K\otimes \End_0W),
$$
where $\End_0W$ denotes the bundle of traceless endomorphisms of
$W$.

Again we refer the reader to
\cite{garcia-prada-gothen-mundet:2009a} for the general
statement of the stability conditions for $\SL(n,\CC)$-Higgs
bundles. (Semi)stability is equivalent in this case to the original
notions given by Hitchin in \cite{hitchin:1987a} (see \cite[Theorem
4.4]{garcia-prada-gothen-mundet:2009a}).

\begin{definition}
\label{thm:sl(n,C)-stability} An  $\SL(n,\CC)$-Higgs
bundle $(W,\Phi)$ is {\bf semistable} if and only if for any subbundle
$W'\subset W$ such that $\Phi(W')\subset K\otimes W'$ we have
$\deg W'\leq 0$. Furthermore, $(W,\Phi)$ is {\bf stable} if for any
nonzero and strict subbundle $W'\subset W$ such that
$\Phi(W')\subset K\otimes W'$ we have $\deg W'<0$. Finally,
$(W,\Phi)$ is {\bf polystable} if it is semistable and for each
subbundle $W'\subset W$ such that $\Phi(W')\subset K\otimes W'$
and $\deg W'=0$ there is another subbundle $W''\subset W$
satisfying $\Phi(W'')\subset K\otimes W''$ and $W=W'\oplus W''$.
\end{definition}

Consider now the case $G=\Sp(2n,\CC)$. A maximal compact subgroup
of $G$ is $H=\Sp(2n)$ and hence  $H^\CC$ coincides with
$\Sp(2n,\CC)$. Now, if $\WW=\CC^{2n}$ is the fundamental
representation of $\Sp(2n,\CC)$ and $\omega$ denotes the standard
symplectic form on $\WW$, the isotropy representation space is
$$
\lie{m}^\CC=\splie(\WW)=\splie(\WW,\omega):= \{\xi\in\End(\WW)\mid
\omega(\xi\cdot,\cdot)+\omega(\cdot, \xi\cdot)=0\}\subset\End\WW,
$$
so it coincides with the adjoint representation of $\Sp(2n,\CC)$
on its Lie algebra. An $\Sp(2n,\CC)$-Higgs bundle is
thus a pair consisting of a rank $2n$ holomorphic symplectic
vector bundle $(W,\Omega)$ over $X$ (so $\Omega$ is a holomorphic
section of $\Lambda^2W^*$ whose restriction to each fibre of $W$
is non degenerate) and a section
$$
\Phi \in H^0(K\otimes \splie(W)),
$$
where $\splie(W)$ is the vector bundle whose fibre over $x$ is
given by $\splie(W_x,\Omega_x)$.

As for $\SL(n,\CC)$, we refer the reader to
\cite{garcia-prada-gothen-mundet:2009a} for the general statement of
the stability conditions for $\Sp(2n,\CC)$-Higgs bundles. We now have
the following analogue of Proposition \ref{spnr-stability}, which
implies that the definition of (semi)stability from \cite[Theorem
4.2]{garcia-prada-gothen-mundet:2009a} coincides with the usual one in
the literature in the case $\Phi = 0$ (cf.\ Ramanathan
\cite{ramanathan:1975}).  Recall that if $(W,\Omega)$ is a symplectic
vector bundle, a subbundle $W'\subset W$ is said to be isotropic if
the restriction of $\Omega$ to $W'$ is identically zero.

\begin{definition}
  \label{thm:sp(2n,C)-stability} An $\Sp(2n,\CC)$-Higgs bundle
  $((W,\Omega),\Phi)$ is {\bf semistable} if and only if for any
  isotropic subbundle $W'\subset W$ such that $\Phi(W')\subset
  K\otimes W'$ we have $\deg W'\leq 0$. Furthermore,
  $((W,\Omega),\Phi)$ is {\bf stable} if for any nonzero and strict
  isotropic subbundle $0\neq W'\subset W$ such that $\Phi(W')\subset
  L\otimes W'$ we have $\deg W'<0$. Finally, $((W,\Omega),\Phi)$ is
  {\bf polystable} if it is semistable and for any nonzero and strict
  isotropic (resp., coisotropic) subbundle $W'\subset W$ such that
  $\Phi(W')\subset L\otimes W'$ and $\deg W'=0$ there is a coisotropic
  (resp., isotropic) subbundle $W''\subset W$ such that
  $\Phi(W'')\subset L\otimes W''$ and $W=W'\oplus W''$.
\end{definition}

Given an $\Sp(2n,\RR)$-Higgs bundle  $(V,\varphi)$ with
$\varphi=(\beta,\gamma)\in H^0(K\otimes(S^2V\oplus S^2V^*))$ one
can associate to it an $\Sp(2n,\CC)$-Higgs bundle
$((W,\Omega),\Phi)$ given by
\begin{equation}
  \label{eq:sp2nC-associated-higgs}
        W = V \oplus V^*,\quad
              \Phi =
      \begin{pmatrix}
        0 & \beta \\
        \gamma & 0
      \end{pmatrix}
      \quad\text{and}\quad
      \Omega\bigl((v,\xi),(w,\eta)\bigr) = \xi(w) - \eta(v),
\end{equation}
for local holomorphic sections $v,w$ of $V$ and $\xi,\eta$ of $V^*$
(i.e.\ $\Omega$ is the canonical symplectic structure on $V \oplus
V^*$).

Since $\Sp(2n,\CC)\subset \SL(2n,\CC)$, every $\Sp(2n,\CC)$-Higgs
bundle $((W,\Omega),\Phi)$ gives rise to an $\SL(2n,\CC)$-Higgs
bundle $(W,\Phi)$. If $((W,\Omega),\Phi)$ is obtained from an
$\Sp(2n,\RR)$-Higgs bundle $(V,\varphi)$ we denote the associated
$\SL(2n,\CC)$-Higgs bundle by
\begin{displaymath}
  H(V,\varphi) = (W,\Phi) =
  (V \oplus V^*,
  \begin{pmatrix}
    0 & \beta \\
    \gamma & 0
  \end{pmatrix}).
\end{displaymath}

\begin{theorem}\label{thm:stability-equivalence}
Let $(V,\varphi=(\beta,\gamma))$ be an $\Sp(2n,\RR)$-Higgs bundle
and let $(W,\Phi)=H(V,\varphi)$ be the corresponding
$\SL(2n,\CC)$-Higgs bundle. Then
\begin{enumerate}
\item if $(W,\Phi)$ is stable then $(V,\varphi)$ is stable;
\item\label{item:intertwined} if $(V,\varphi)$ is stable and simple
  then $(W,\Phi)$ is stable unless there is an isomorphism
  $f:V\xrightarrow{\simeq} V^*$ such that $\beta f=f^{-1}\gamma$, in
  which case $(W,\Phi)$ is polystable;
\item $(W,\Phi)$ is semistable if and only if $(V,\varphi)$ is semistable.
\item $(W,\Phi)$ is polystable if and only if $(V,\varphi)$ is polystable;
\end{enumerate}
In particular, if $\deg(V)\neq 0$ then $(W,\Phi)$ is stable if and
only if $(V,\varphi)$ is stable.
\end{theorem}

For the statement of the following theorem, recall from
Section~\ref{glnr-higgs} that a $\GL(n,\RR)$-Higgs bundle is given
by $((W,Q),\psi)$, where $(W,Q)$ is rank $n$ orthogonal bundle and
$\psi \in H^0(K \otimes S^2 W)$. The stability condition for
$\GL(n,\RR)$-Higgs bundles is given in
Proposition~\ref{thm:orthogonal-stability}.

\begin{theorem}\label{thm:stability-equivalence-sp2nC}
  Let $(V,\varphi)$ be a stable and simple $\Sp(2n,\RR)$-Higgs
  bundle. Then $(V,\varphi)$ is stable as an $\Sp(2n,\CC)$-Higgs
  bundle, unless there is a symmetric isomorphism $f\colon V
  \xrightarrow{\simeq} V^*$ such that $\beta f =
  f^{-1}\gamma$. Moreover, if such an $f$ exists, let $\psi = \beta =
  f^{-1}\gamma f^{-1} \in H^0(K \otimes S^2 V)$. Then the
  $\GL(n,\RR)$-Higgs bundle $((V,f),\psi)$ is stable, even as a
  $\GL(n,\CC)$-Higgs bundle.
\end{theorem}

The proof of Theorem~\ref{thm:stability-equivalence} is given below in
Section~\ref{sec:proof-thm:stability-equivalence} and the proof of
Theorem~\ref{thm:stability-equivalence-sp2nC} is given below in
Section~\ref{sec:proof-thm:stability-equivalence-sp2nC}.

The following observation is not essential for our main line of
argument. We include it since it might be of independent interest.
\begin{remark}
  \label{rem:intertwined-cases}
  Suppose we are in Case~(\ref{item:intertwined}) of
  Theorem~\ref{thm:stability-equivalence}. Decompose $f = f_s +
  f_a\colon V \xrightarrow{\simeq} V$ in its symmetric and
  anti-symmetric parts, given by $f_s = \frac{1}{2}(f+f^t)$ and $f_a =
  \frac{1}{2}(f-f^t)$. Let $V_a = \ker(f_s)$ and $V_s =
  \ker(f_a)$. Both $V_a$ and $V_s$ are vector bundles, since the ranks
  of $f_s$ and $f_a$ (which coincide with the multiplicities of $-1$
  and $1$ respectively as eigenvalues of $f$) are constant.
  There is then a decomposition $V = V_a \oplus
  V_s$ and $f$ decomposes as
  \begin{displaymath}
    f =
    \begin{pmatrix}
      f_s & 0 \\
      0 & f_a
    \end{pmatrix}
    \colon V_s \oplus V_a \to V_s^* \oplus V_a^*.
  \end{displaymath}
  Write $\gamma_{sa}\colon V_a \to V_s^* \otimes K$ for the component
  of $\gamma$ in $H^0(K \otimes V_a^*\otimes V_s^*)$ and similarly for
  the other mixed components of $\beta$ and $\gamma$. Since $f$
  intertwines $\beta$ and $\gamma$, one has that $\gamma_{as} = f_a
  \beta_{as} f_s$. Hence
  \begin{displaymath}
    \gamma_{sa} = \gamma_{as}^t = f_s^t\beta_{as}^tf_a^t =
    -f_s\beta_{sa}f_a = -\gamma_{sa}.
  \end{displaymath}
  It follows that $\gamma_{sa} = 0$ and similarly for the other mixed
  terms. Thus there is a decomposition $(V,\varphi) = (V_s \oplus V_a,
  \varphi_s + \varphi_a)$. If $(V,\varphi)$ is simple then one of the
  summands must be trivial. The case when $(V,\varphi) =
  (V_s,\varphi_s)$ is the one covered in
  Theorem~\ref{thm:stability-equivalence-sp2nC}. In the other case,
  when $(V,\varphi) = (V_a,\varphi_a)$, the antisymmetric map $f$
  defines a symplectic form on $V$. If we let $\psi = \beta f =
  f^{-1}\gamma$, one easily checks that $\psi$ is symplectic. Thus, in this case, $(V,\varphi)$ comes in
  fact from an $\Sp(n,\CC)$-Higgs bundle $((V,f),\psi)$. This is a
  stable $\Sp(n,\CC)$-Higgs bundle, since $(V,\psi)$ is a stable
  $\GL(n,\CC)$-Higgs bundle (cf.\ the proof of
  Theorem~\ref{thm:stability-equivalence-sp2nC} below).
\end{remark}

%%%%%%%%%%%%%%%%%%%%%%%%%%%%%%%%%%%%%%%%%%%%%%%%%%%%%%%%%%%%%
\subsection{Proof of Theorem \ref{thm:stability-equivalence}}
%%%%%%%%%%%%%%%%%%%%%%%%%%%%%%%%%%%%%%%%%%%%%%%%%%%%%%%%%%%%%
\label{sec:proof-thm:stability-equivalence}

The proof of the theorem is split into several lemmas.

The following lemma proves (1) of Theorem \ref{thm:stability-equivalence}.

\begin{lemma}
Let $(V,\varphi=(\beta,\gamma))$ be an $\Sp(2n,\RR)$-Higgs bundle,
and let
$$\Phi=\left(
\begin{array}{cc} 0 & \beta \\ \gamma & 0\end{array}\right)
\colon V\oplus V^*\to K\otimes (V\oplus V^*).$$

The pair $(V,\varphi)$ is semistable if and only if for any
pair of subbundles
$A\subset V$ and $B\subset V^*$ satisfying
$B^{\perp}\subset A$ (equivalently, $\ A^{\perp}\subset B$) and
$\Phi(A\oplus B)\subset K\otimes (A\oplus B)$
we have $\deg(A\oplus B) \leq 0$.

The pair $(V,\varphi)$ is stable if and only if it is semistable
and for any pair of subbundles $A \subset V$ and $B \subset V^*$, 
such that $A\neq V$ and $B\neq V^*$,  satisfying $B^{\perp}\subset A$
(equivalently, $A^{\perp}\subset B$) and $\Phi(A\oplus B)\subset
K\otimes (A\oplus B)$, the inequality $\deg(A\oplus B) < 0$ holds.
\label{primerpas}
\end{lemma}
\begin{proof}
Suppose that $A\subset V$ and $B\subset V^*$ satisfy
the conditions of the lemma. Then
setting $V_2:=A$ and $V_1:=B^{\perp}$ we obtain a
filtration $0\subset V_1\subset V_2\subset V$ which,
thanks to Remark~\ref{rem:matrix-form}, satisfies
(\ref{eq:condicions2}).

Conversely, given a filtration $0\subset V_1\subset V_2\subset V$
for which (\ref{eq:condicions2}) holds, we get subbundles
$A:=V_2\subset V$ and $B:=V_1^{\perp}\subset V^*$
satisfying the conditions of the lemma.
Finally, we have
$$\deg (A\oplus B)=\deg (V_1^{\perp}\oplus V_2)=
\deg(V_1)+\deg(V_2)-\deg(V),$$
so the lemma follows from Proposition~\ref{spnr-stability}.
(For the case of stability, note that the condition $A\neq V$ and
$B\neq V^*$ corresponds to the filtration $0\subset V_1 \subset V_2\subset V$
not being   the trivial filtration
$0=V_1 \subset V_2= V$.)
\end{proof}

\begin{remark}
In the proof we have used the following formula:
if $F\subset E$ is an inclusion of vector bundles,
then $\deg F^{\perp}=\deg F-\deg E$.
To check this, observe that there is an exact sequence
$0\to F^{\perp}\to E^*\to F^*\to 0,$
and apply the additivity of the degree w.r.t. exact sequences
together with $\deg E^*=-\deg E$ and $\deg F^*=-\deg F$.
\label{formuleta}
\end{remark}

The following lemma proves
(2) of Theorem \ref{thm:stability-equivalence}.

\begin{lemma}
  Suppose that $(V,\varphi)$ is semistable, and define $\Phi\colon V\oplus
  V^*\to K\otimes (V\oplus V^*)$ as previously.  Then any subbundle
  $0\neq W'\subsetneq V\oplus V^*$ such that $\Phi(W')\subset K\otimes
  W'$ satisfies $\deg W'\leq 0$.  Furthermore, if $(V,\varphi)$ is
  stable and simple, one can get equality only if there is an
  isomorphism $f\colon V\to V^*$ such that $\beta f=f^{-1}\gamma$, and in this
  case $(W,\Phi)=H(V,\varphi)$ is polystable.
\label{segonpas}
\end{lemma}
\begin{proof}
Fix a subbundle $W'\subset V\oplus V^*$
satisfying $\Phi(W')\subset K\otimes W'$.
We prove the lemma in various steps.

\noindent{\bf 1.}
Denote by $p:V\oplus V^*\to V$ and $q:V\oplus V^*\to V^*$
the projections, and define subsheaves $A=p(W')$ and
$B=q(W')$.
It follows from $\Phi (W')\subset K\otimes W'$
that $\beta (B)\subset K\otimes A$ and
$\gamma (A)\subset K\otimes B$ (for example, using that
$\Phi p = q\Phi$ and $\Phi q=p\Phi$).
Since both $\beta$ and $\gamma$ are symmetric we deduce that
$\beta (A^{\perp})\subset K\otimes B^{\perp}$
and $\gamma (B^{\perp})\subset K\otimes A^{\perp}$
as well. It follows from this that if we define subsheaves
$$A_0=A + B^{\perp}\subset V \qquad\text{ and }\qquad
B_0=B+A^{\perp}\subset V^*$$
then we have $B_0^{\perp}\subset A_0$,
$A_0^{\perp}\subset B_0$ and $\Phi(A_0\oplus B_0)\subset
K\otimes(A_0\oplus B_0)$.

We can apply Lemma \ref{primerpas} also to subsheaves by replacing any
subsheaf of $V$ or $V^*$ by its saturation, which is now a subbundle
of degree not less than that of the subsheaf.  Hence we deduce that
\begin{equation}
\deg A_0+\deg B_0=\deg(A + B^{\perp})+\deg(B+A^{\perp})\leq 0,
\label{desAB}
\end{equation}
and equality holds if and only if $A+B^{\perp}=V$ and $B+A^{\perp}=V^*$.

Now we compute (using repeatedly the formula in Remark \ref{formuleta})
\begin{align*}
\deg(A + B^{\perp}) &= \deg A+\deg B^{\perp}-\deg(A\cap B^{\perp}) \\
&=\deg A+\deg B-\deg V^*-\deg((A^{\perp}+B)^{\perp})\\
&=\deg A+\deg B-\deg V^*-\deg (A^{\perp}+B)+\deg V^*\\
&=\deg A+\deg B-\deg (A^{\perp}+B).
\end{align*}
Consequently $\deg A+\deg B=\deg(A + B^{\perp})+\deg (A^{\perp}+B)$,
so (\ref{desAB}) implies that
\begin{equation}
\deg A+\deg B\leq 0,
\label{desABfinal}
\end{equation}
with equality if and only if $A+B^{\perp}=V$ and $B+A^{\perp}=V^*$.

\noindent{\bf 2.}
Let now $A'=W'\cap V$ and $B'=W'\cap V^*$.
Using again that $\Phi(W')\subset K\otimes W'$ we prove
that $\beta (B')\subset K\otimes A'$ and
$\gamma (A')\subset K\otimes B'$. Now, the same reasoning
as above (considering $(A'+{B'}^{\perp})\oplus (B'+{A'}^{\perp})$
and so on) proves that
\begin{equation}
\deg A'+\deg B'\leq 0,
\label{desABfinal2}
\end{equation}
with equality if and only if
$A'+{B'}^{\perp}=V$ and ${A'}^{\perp}+B'=V^*$.

\noindent{\bf 3.} Observe that there are exact sequences of sheaves
$$0\to B'\to W'\to A\to 0
\qquad\text{ and }\qquad
0\to A'\to W'\to B\to 0,$$
from which we obtain the formulas
$$\deg W'=\deg A+\deg B'
\qquad\text{ and }\qquad
\deg W'=\deg B+\deg A'.$$
Adding up and using (\ref{desABfinal}) together with (\ref{desABfinal2})
we obtain the desired inequality
$$\deg W'\leq 0.$$

\noindent{\bf 4.} Finally we consider the case when $(V,\varphi)$ is
stable and simple.  Suppose that $\deg W'=0$. Then we have
equality both in (\ref{desABfinal}) and in (\ref{desABfinal2}).
Hence, $A+B^{\perp}=V$, $A^{\perp}+B=V^*$,
$A'+{B'}^{\perp}=V$ and ${A'}^{\perp}+B'=V^*$.
But $A^{\perp}+B=(A\cap B^{\perp})^{\perp}$
and ${A'}^{\perp}+B'=(A'\cap {B'}^{\perp})^{\perp}$, so we deduce that
$$A\oplus B^{\perp}=V
\qquad\text{ and }\qquad
A'\oplus {B'}^{\perp}=V.$$
If one of these decompositions were nontrivial then $V$ would
not be simple, in contradiction with our
assumptions. Consequently we must have
$A=V$, $B^{\perp}=0$ (because $W'\neq 0$) and similarly
$A'=0$, ${B'}^{\perp}=V^*$ (because $W'\neq V\oplus V^*$).
This implies that the projections $p:W'\to A$ and $q:W'\to B$
induce isomorphisms $u:W'\simeq V$ and $v:W'\simeq V^*$.
Finally, defining $f:=v\circ u^{-1}:V\to V^*$ we find
an isomorphism which satisfies $\beta f=f^{-1}\gamma$ because
$\Phi (W')\subset K\otimes W'$.

To prove that in this case $(W,\Phi)=H(V,\varphi)$ is strictly polystable
just observe that $W'=\{(u,fu)\mid u\in V\}$ and define
$W''=\{(u,-fu)\mid u\in V\}$. It is then straightforward to check that
$V\oplus V^*=W'\oplus W''$, that $\Phi (W')\subset K\otimes W'$ and that
$\Phi (W'')\subset K\otimes W''$.  Finally note that the Higgs bundle
$(W',\Phi)$ is stable: any $\Phi$-invariant subbundle $W_0 \subset W'$
is also a $\Phi$-invariant subbundle of $(V \oplus V^*, \Phi)$.
Hence, if $\deg W_0=0$ the argument of the previous paragraph shows that
$W_0$ has to have the same rank as $V$, so
$W_0 = W'$.  Analogously, one sees that $(W'',\Phi)$ is a stable Higgs
bundle.
\end{proof}

Combining Lemmas \ref{primerpas} and \ref{segonpas} we get the proof of
(3) of Theorem \ref{thm:stability-equivalence}.

We conclude with the following lemma, which proves
(4) of Theorem \ref{thm:stability-equivalence}.

\begin{lemma}\label{lem:semi-quad-higgs}
An $\Sp(2n,\RR)$-Higgs bundle $(V,\varphi=(\beta,\gamma))$ is
polystable if and only if $H(V,\varphi)$ is polystable.
\end{lemma}
\begin{proof}
  If $(V,\varphi)$ is polystable then, by Proposition
  \ref{prop:polystable-spnr-higgs} there is a splitting $(V,\varphi) =
  (V_1,\varphi_1) \oplus \dots \oplus (V_k,\varphi_k)$ such that each
  $(V_i,\varphi_i)$ is a stable and simple $G_i$-Higgs bundle
  belonging to one out of three possible types of Higgs bundles, which
  are labeled in Proposition \ref{prop:polystable-spnr-higgs} as types
  (1), (2) or (3) (recall that type (1) corresponds
  $G_i=\Sp(2n_i,\RR)$, type (2) to $G_i=\U(p_i,q_i)$ and type (3) to
  $G_i=\U(n_i)$).  By Lemma \ref{segonpas}, for each summand
  $(V_i,\varphi_i)$ belonging to type (1) the Higgs bundle
  $H(V_i,\varphi_i)$ is polystable. The corresponding statement for
  type (2) follows easily using the line of reasoning employed in
  \cite[Section~2.3]{gothen:2001},
  cf.~\cite[Remark~3.8]{bradlow-garcia-prada-gothen:2003}. Finally,
  the statement for type (3) summands is immediate.  This proves one
  direction of the equivalence.

We now prove the converse statement.
Assume that $(W,\Phi)=H(V,\varphi)$ is polystable, so that
$W=\bigoplus_{i=1}^N W_i$ with $\Phi W_i\subset K\otimes W_i$ and
every $(W_i,\Phi|_{W_i})$ is stable with $\deg W_i=0$.

\noindent{\bf 1.}
We claim that for any subbundle $U\subset W$ satisfying
$\deg U=0$ and $\Phi(U)\subset K\otimes U$ there is an isomorphism
$\psi:W\to W$ which commutes with $\Phi$
and a set $I\subset\{1,\dots,N\}$ such that
$U=\psi(\bigoplus_{i\in I}W_i)$. To prove the claim we
use induction on $N$ (the case $N=1$ being obvious).
Let $W_{\geq 2}=\bigoplus_{i\geq 2}W_i$ and denote by
$p_{\geq 2}:W\to W_{\geq 2}$ the projection. Then we have
an exact sequence
$$0\to W_1\cap U\to U\to p_{\geq 2}(U)\to 0.$$
Since both $W_1\cap U$ and $p_{\geq 2}(U)$ are invariant under
$\Phi$, by polystability their degrees must be $\leq 0$. And
since according to the exact sequence above the sum of
their degrees must be $0$, the only possibility is that
$$\deg W_1\cap U=0\qquad\text{ and }\qquad \deg p_{\geq 2}(U)=0.$$
Now we apply the induction hypothesis to the inclusion
$p_{\geq 2}(U)\subset W_{\geq 2}$ and deduce that there
is an isomorphism $\psi_2:W_{\geq 2}\to W_{\geq 2}$ commuting
with $\Phi$ and a
subset $I_2\subset\{2,\dots,N\}$ such that
$$p_{\geq 2}(U)=\psi_2(\bigoplus_{i\in I_2}W_i).$$
Since $\deg W_1\cap U=0$ and $W_1$ is stable, only two things can
happen. Either $W_1\cap U=W_1$ or $W_1\cap U=0$.
In the first case we have
$$U=W_1\oplus\bigoplus_{i\in I_2}\psi(W_i),$$
so putting $I=\{1\}\cap I_2$ and $\psi=\diag(1,\psi_2)$ the
claim is proved. If instead $W_1\cap U=0$ then there is a map
$\xi:p_{\geq 2}(U)\to W_1$ such that
$$U=\{(\xi(v),v)\in W_1\oplus p_{\geq 2}\}.$$
Since $U$ is $\Phi$-invariant we deduce that $\xi$ must commute with
$\Phi$.
If we now extend $\xi$ to $W_{\geq 2}$ by defining
$\xi(\psi_2(W_j))=0$ for any $j\in\{2,\dots,N\}\setminus I_2$ then
the claim is proved by setting $I=I_2$ and
$$\psi=\left(\begin{array}{cc}
1 & \xi\circ\psi_2 \\ 0 & \psi_2\end{array}\right).$$

\noindent{\bf 2.}
Define for any $W'\subset W$ the subsheaves
$R(W')=p(W')\oplus q(W')$ (recall that $p:W\to V$ and $q:W\to V^*$
are the projections) and $r(W')=(W'\cap V)\oplus (W'\cap V^*)$.
Reasoning as in the first step of the proof of Lemma \ref{segonpas}
we deduce that if $W'$ is $\Phi$-invariant then both $R(W')$ and
$r(W')$ are $\Phi$ invariant, so in particular we must have
$\deg R(W')\leq 0$ and $\deg r(W')\leq 0$. In case $\deg W'=0$
these inequalities imply $\deg R(W')=\deg r(W')=0$ (using the
exact sequences $0\to W'\cap V^*\to W'\to p(W')\to 0$ and
$0\to W'\cap V\to W'\to q(W')\to 0$).

Assume that there is some summand in $\{W_1,\dots,W_N\}$, say
$W_1$, such that $0\neq r(W_1)$ or $R(W_1)\neq W$. Suppose, for
example, that $W':=R(W_1)\neq W$ (the other case is similar).
Let $A=p(W_1)$ and $B=q(W_1)$, so that $W'=A\oplus B$.
By the observation above and the
claim proved in {\bf 1} we know that there is
an isomorphism $\psi:W\to W$ which commutes with $\Phi$
and such that, if we substitute
$\{W_i\}_{1\leq i\leq N}$ by $\{\psi(W_i)\}_{1\leq i\leq N}$
and we reorder the summands if necessary, then we may write
$W'=W_1\oplus\dots\oplus W_k$ for some $k<N$.
Now let $W''=W_{k+1}\oplus\dots\oplus W_N$.
We clearly have $W=W'\oplus W''$, so the inclusion of
$W''\subset W=V\oplus V^*$ composed with the projection
$V\oplus V^*\to V/A\oplus V^*/B=W/W'$ induces an isomorphism.
Consequently we have $V=A\oplus W''\cap V$. Let us rename
for convenience $V_1:=A$ and $V_2:=W''\cap V$. Then, using
the fact that each $W_i$ is $\Phi$-invariant we deduce that
we can split both $\beta$ and $\gamma$ as
\begin{align*}
\beta &= (\beta_1,\beta_2)
\in H^0(K\otimes S^2V_1)\oplus H^0(K\otimes S^2V_2), \\
\gamma &= (\gamma_1,\gamma_2)
\in H^0(K\otimes S^2V_1^*)\oplus H^0(K\otimes S^2V_2^*).
\end{align*}
Hence, if we put $\varphi_i=(\beta_i,\gamma_i)$ for $i=1,2$ then
we may write
$$(V,\varphi)=(V_1,\varphi_1)\oplus (V_2,\varphi_2).$$

\noindent{\bf 3.}
Our strategy is now to apply recursively the process described
in {\bf 2}. Observe that if $N\geq 3$ then for at least one
$i$ we have $R(W_i)\neq W$, because there must be a summand
whose rank is strictly less that the rank of $V$. Hence the
projection of this summand to $V$ is not exhaustive.

Consequently, we can apply the process and split $V$ in smaller
and smaller pieces, until we arrive at a decomposition
$$(V,\varphi)=(V_1,\varphi_1)\oplus\dots\oplus (V_j,\varphi_j)$$
such that we can not apply {\bf 2} to any $H(V_i,\varphi_i)$
For each  $(V_i,\varphi_i)$ there are two possibilities.
Either $H(V_i,\varphi_i)$ is stable, in which case $(V_i,\varphi_i)$
is stable (by Lemma \ref{primerpas}), or
$H(V_i,\varphi_i)$ splits in two stable Higgs bundles
$W_i'\oplus W_i''$ which satisfy:
$$R(W_i')=R(W_i'')=W\qquad\text{ and }\qquad
r(W_i')=r(W_i'')=0.$$
But in this case it is easy to check that $(V_i,\varphi_i)$ is
also stable.

By the preceding lemma, $(V,\varphi)$ is semistable.
Suppose it is not stable. Then there is a filtration
$0\subset V_1\subset V_2\subset V$ such that
$\Phi(V_2\oplus V_1^{\perp})\subset K\otimes (V_2\oplus V_1^{\perp})$
and $W':= V_2\oplus V_1^{\perp}=0$ has degree $\deg W'=0$.

Define $W_{\geq 2}=\bigoplus_{i\geq 2}W_i$, and
let $p_2:W\to W_{\geq 2}$ denote the projection.
We have an exact sequence
$$0\to W'\cap W_1\to W'\to p_2(W')\to 0.$$
It is easy to check that $\Phi(W'\cap W_1)\subset K\otimes (W'\cap
W_1)$ and that $\Phi(p_2(W'))\subset K\otimes p_2(W')$. Since both
$W_1$ and $W_{\geq 2}$ are polystable, we must have $\deg W'\cap
W_1\leq 0$ and $\deg p_2(W')\leq 0$. Finally, since $\deg W'=0$,
the exact sequence above implies that $\deg W'\cap W_1=0$ and
$\deg p_2(W')=0$. Now $W_1$ is stable, so $W'\cap W_1$ can only be
either $0$ or $W_1$. Reasoning inductively with $p_2(W')\subset
W_{\geq 2}$ in place of $W'\subset W$ we deduce that there must be
some $I\subset\{1,\dots,k\}$ such that
$$W'=\bigoplus_{i\in I}W_i.$$

Since each $(W_i,\Phi|_{W_i})$ is stable, it is easy to check
(for example using induction on $N$) that one must have
$ V_2\oplus V_1^{\perp}=W_j$ for some $j$.
This easily implies that
$V_2=V\cap W_j$ and if we define
$$V'=\bigoplus_{i\neq j} p(W_j)$$ then
$V=V'\oplus V_2$. Applying the same process to $V'$ and $V_2$
we arrive at the conclusion that $(V,\varphi)$ is polystable.
\end{proof}

%%%%%%%%%%%%%%%%%%%%%%%%%%%%%%%%%%%%%%%%%%%%%%%%%%%%
\subsection{Proof of Theorem~\ref{thm:stability-equivalence-sp2nC}}
\label{sec:proof-thm:stability-equivalence-sp2nC}
%%%%%%%%%%%%%%%%%%%%%%%%%%%%%%%%%%%%%%%%%%%%%%%%%%%%

An $\Sp(2n,\CC)$-Higgs bundle $((W,\Omega),\Phi)$ is stable if  the
$\SL(2n,\CC)$-Higgs bundle $(W,\Phi)$ is stable. Thus the result is
immediate from Theorem~\ref{thm:stability-equivalence}, unless we are
in Case~(\ref{item:intertwined}) of that Theorem. In that case, we have
seen in the last paragraph of the proof of Lemma~\ref{segonpas} that
  \begin{enumerate}
  \item There is an isomorphism $f$ as stated, except for the symmetry
    condition.
  \item There is an isomorphism
  \begin{math}
    V \oplus V^* = W' \oplus W'',
  \end{math}
  where $W' = \{(u,f(u)) \suchthat u \in V \}$ and $W'' = \{(u,-f(u))
  \suchthat u \in V \}$, and $W'$ and $W''$ are
  $\Phi$-invariant subbundles of $W$.
  \item The $\SL(2n,\CC)$-Higgs bundle $(W,\Phi)$ is strictly
  polystable, decomposing as the direct sum of stable $\GL(n,\CC)$-Higgs
  bundles:
  \begin{equation}\label{eq:decom-sl}
    (W,\Phi) = (W',\Phi') \oplus (W'',\Phi'').
  \end{equation}
  Note also that $(W',\Phi') \simeq (W'',\Phi'')$.
  \end{enumerate}
  Now, from Theorem~\ref{thm:sp(2n,C)-stability} we have that for the
  $\Sp(2n,\CC)$-Higgs bundle $((W,\Omega),\Phi)$ to be strictly
  semistable, it must have an isotropic $\Phi$-invariant subbundle of
  degree zero. But the decomposition (\ref{eq:decom-sl}) shows that
  the only degree zero $\Phi$-invariant subbundles are $W'$ and
  $W''$. The subbundle $W'$ is isotropic if and only if, for all local
  sections $u,v$ of $V$, we have
  \begin{displaymath}
      \Omega((u,f(u),(v,f(v)) = 0
      \iff \la u, f(v) \ra = \la v , f(u) \ra,
  \end{displaymath}
  that is, if and only if $f$ is symmetric. Analogously, $W''$ is
  isotropic if and only if $f$ is symmetric. The first part of the
  conclusion follows.

  For the second part, consider the $\GL(n,\RR)$-Higgs bundle
  $((V,f),\beta f)$. This is stable as a $\GL(n,\CC)$-Higgs bundle
  because $(V,\beta f) \simeq (W',\Phi')$, which is stable. Thus, in
  particular, $((V,f),\beta f)$ is stable as a $\GL(n,\RR)$-Higgs
  bundle (see Definition~\ref{thm:orthogonal-stability}).
  \qed

%%%%%%%%%%%%%%%%%%%%%%%%%%%%%%%%%%%%%%%%%%%%%%%%%%%%%%%%%%%
  \subsection{Milnor--Wood inequality and moduli space of
    $\Sp(2n,\RR)$-Higgs bundles}
\label{sec:spn-higgs-moduli}
%%%%%%%%%%%%%%%%%%%%%%%%%%%%%%%%%%%%%%%%%%%%%%%%%%%%%%%%%%%

The topological invariant attached to an $\Sp(2n,\RR)$-Higgs
bundle $(V,\beta,\gamma)$ is an element in the fundamental group
of $\U(n)$. Since $\pi_1(U(n))\simeq \ZZ$, this is an integer,
which coincides with the degree of $V$.

We have the following Higgs bundle incarnation of the Milnor--Wood
inequality \eqref{mw1} (see \cite[Proposition~3.21]{gothen:2001} or
\cite[Corollary~3.27]{bradlow-garcia-prada-gothen:2003}).

\begin{proposition}\label{bounds-semistability}
 Let $(V,\beta,\gamma)$ be a semistable $\Sp(2n,\RR)$-Higgs bundle and let
$d = \deg(V)$.  Then
\begin{eqnarray}
d&\leq \rank(\gamma) (g-1) \label{semistability}\\
-d&\leq \rank(\beta) (g-1),
\end{eqnarray}
\end{proposition}

This is proved by first using the equivalence between the
semistability of $(V,\beta,\gamma)$ and the $\SL(2n,\CC)$-Higgs bundle
$(W,\Phi)$ associated to it, and then applying the semistability
numerical criterion to special Higgs subbundles defined by the kernel
and image of $\Phi$.

As a consequence of Proposition~\ref{bounds-semistability} we have
the  following.

\begin{proposition}\label{mw-higgs}
  Let $(V,\beta,\gamma)$ be a semistable $\Sp(2n,\RR)$-Higgs bundle and let
$d= \deg(V)$.  Then
  \begin{displaymath}
    \abs{d} \leq n(g-1).
  \end{displaymath}
  Furthermore,
  \begin{itemize}
  \item[$(1)$] $d = n(g-1)$ holds if and only if
  $\gamma\colon V \to V^* \otimes K$ is an isomorphism;
  \item[$(2)$] $d = -n(g-1)$ holds if and only if
  $\beta\colon V^* \to V \otimes K$ is an isomorphism.
  \end{itemize}
\end{proposition}

Recall from our general discussion in Section~\ref{sec:g-higgs-defs} that
$\cM_d(\Sp(2n,\RR))$ denotes the moduli space of $\Sp(2n,\RR)$-Higgs
bundles $(V,\beta,\gamma)$ with $\deg(V) = d$. For brevity we shall
henceforth write simply $\cM_d$ for this moduli space.

Combining Theorem \ref{alg-moduli} with Proposition \ref{prop:dim-moduli}
we have the following.

\begin{proposition}
\label{prop:expected-dim-MSp2n}
The moduli space $\cM_d$ is a complex algebraic variety. Its
expected dimension is  $(g-1)(2n^2+n)$.
\end{proposition}

One has the following immediate duality result.

\begin{proposition}\label{prop:duality-iso}
  The map $(V,\beta,\gamma)\mapsto (V^*,\gamma^t,\beta^t)$ gives an
  isomorphism $\cM_d\simeq \cM_{-d}$.
\end{proposition}

As a corollary of Proposition \ref{mw-higgs}, we obtain
the following.

\begin{proposition}\label{non-emptiness-higgs}
The moduli space $\cM_d$ is empty unless
$$
\abs{d} \leq n(g-1).
$$
\end{proposition}

%%%%%%%%%%%%%%%%%%%%%%%%%%%%%%%%%%%%%%%%%%%%%%%%%%%%%%%%%%%%%%%%%%%%%%%%%%%
\subsection{Smoothness and polystability of  $\Sp(2n,\RR)$-Higgs bundles}
\label{sec:smoothness-polystability}
%%%%%%%%%%%%%%%%%%%%%%%%%%%%%%%%%%%%%%%%%%%%%%%%%%%%%%%%%%%%%%%%%%%%%%%%%%%

We study now the smoothness properties of the moduli space. As a
corollary of Proposition \ref{prop:smoothness} and
Theorem~\ref{thm:stability-equivalence-sp2nC} we have the following.

\begin{proposition}\label{cor:smooth-points}
  Let $(V,\varphi)$ be an $\Sp(2n,\RR)$-Higgs bundle which is stable
  and simple and assume that there is no symmetric isomorphism
  $f\colon V \xrightarrow{\simeq} V^*$ intertwining $\beta$ and
  $\gamma$.  Then $(V,\varphi)$ represents a smooth point of the
  moduli space of polystable $\Sp(2n,\RR)$-Higgs bundles.
\end{proposition}

So, a stable $\Sp(2n,\RR)$-Higgs bundle $(V,\varphi)$ in $\cM_d$
with $d\neq 0$ can only fail to be a smooth point of the moduli
space if it is not simple --- this gives rise to an orbifold-type
singularity --- or if, in spite of being simple, there is an
isomorphism $V \simeq V^*$ intertwining $\beta$ and $\gamma$. Of
course, this can only happen if $d=\deg(V)=0$.  Generally, if
$(V,\varphi)$ is polystable, but not stable it is also a singular
point of $\mathcal{M}_d$.

We shall  need the following analogue of
Proposition~\ref{cor:smooth-points} for $\U(n)$-, $\GL(n,\CC)$- $\U(p,q)$- and
$\GL(n,\RR)$-Higgs bundles

\begin{proposition}
  \label{prop:smooth-u-upq-glnr}
  \begin{enumerate}
  \item\label{item:u-smooth} A stable $\U(n)$-Higgs bundle represents
    a smooth point in the moduli space of $\U(n)$-Higgs bundles.
\item\label{item:gl-smooth} A stable $\GL(n,\CC)$-Higgs bundle represents
    a smooth point in the moduli space of $\GL(n,\CC))$-Higgs bundles.

  \item\label{item:upq-smooth} A stable $\U(p,q)$-Higgs bundle $(V,W,\beta,\gamma)$
represents a
    smooth point of the moduli space of $\U(p,q)$-Higgs bundles
    unless there exists an isomorphism $f:V\to W$ such that $\beta f=
    f^{-1}\gamma$. In this case $p=q$ and $(V,\beta f)$ defines a
    stable $\GL(p,\CC)$-Higgs bundle.

  \item\label{item:glnr-smooth} A $\GL(n,\RR)$-Higgs bundle which is stable as a
    $\GL(n,\CC)$-Higgs bundle represents a smooth point in the
    moduli space of $\GL(n,\RR)$-Higgs bundles.
  \end{enumerate}
\end{proposition}

\begin{proof}
  (\ref{item:u-smooth}) A stable $\U(n)$-Higgs bundles is nothing but
  a stable vector bundle, so this is classical.

  (\ref{item:gl-smooth}) This is also classical (\cite{hitchin:1987a}).

  (\ref{item:upq-smooth}) We first observe that a stable
  $\U(p,q)$-Higgs bundle is simple. After this we prove that a stable  $\U(p,q)$-Higgs
  bundle is stable as a $\GL(p+q,\CC)$-bundle unless there exists an isomorphism $f:V\to W$ such that $\beta f=
    f^{-1}\gamma$. The proof follows along the same lines as that
  of Theorem \ref{thm:stability-equivalence-sp2nC}, although it is
  simpler since $f$ is not symmetric. Now the stability as a
  $\GL(p+q,\CC)$-Higgs bundle ensures that the appropriate $\HH^2$ obstruction for
  smoothness vanishes.

  (\ref{item:glnr-smooth}) A stable $\GL(n,\RR)$-Higgs bundle is
  simple if and only if as a $\GL(n,\CC)$-Higgs bundle it is stable.
 The rest follows by a similar argument to the last one in
  (\ref{item:upq-smooth}).
\end{proof}

It will be convenient to make the following definition for
$\GL(n,\RR)$-Higgs bundles, analogous to the way we associate
$\Sp(2n,\RR)$-Higgs bundles to vector bundles and $\U(p,q)$-Higgs
bundles in (\ref{eq:un-sp2n}) and (\ref{eq:upq-sp2n}), respectively
(cf.\ Theorem~\ref{thm:stability-equivalence-sp2nC}).
Given a $\GL(n,\RR)$-Higgs bundle $((W,Q),\psi)$, let $f\colon W \to
W^*$ be the symmetric isomorphism associated to $Q$. Define an
associated $\Sp(2n,\RR)$-Higgs bundle
\begin{equation}
  \label{eq:gln-spn-higgs}
  (V,\varphi) = \upsilon_*^{\GL(n,\RR)}((W,Q),\psi)
\end{equation}
by setting
\begin{displaymath}
  V = W, \quad \beta = \psi \quad\text{and}\quad \gamma=f\psi f.
\end{displaymath}
Again we shall
slightly abuse language, saying simply that
$\upsilon_*^{\GL(n,\RR)}((W,Q),\psi)$ is a $\GL(n,\RR)$-Higgs
bundle, whenever no confusion is likely too occur.

Putting everything together we obtain our main result of this section:
a structure theorem for polystable $\Sp(2n,\RR)$-Higgs bundles.
\begin{theorem}
  \label{thm:polystable-spnr-higgs}
  Let $(V,\varphi)$ be a polystable $\Sp(2n,\RR)$-Higgs bundle. Then
  there is a decomposition $(V,\varphi) = (V_1,\varphi_1) \oplus \dots
  \oplus (V_k,\varphi_k)$, unique up to reordering, such that each of
  the $\Sp(2n_i,\RR)$-Higgs bundles $(V_i,\varphi_i)$ is one of the
  following:
  \begin{enumerate}
  \item A stable and simple $\Sp(2n_i,\RR)$-Higgs bundle, which is
  stable as a $\Sp(2n,\CC)$-Higgs bundle.
  \item A stable $\U(p_i,q_i)$-Higgs bundle with $n_i = p_i+q_i$,
  which is stable as a $\GL(p_i+q_i,\CC)$-Higgs bundle
  \item A stable $\U(n_i)$-Higgs bundle.
  \item A $\GL(n_i,\RR)$-Higgs bundle which is stable as a
    $\GL(n_i,\CC)$-Higgs bundle.
\item A stable $\GL(n_i,\CC)$-Higgs bundle.

  \end{enumerate}
  Each $(V_i,\varphi_i)$ is a smooth point in the moduli space of
  $G_i$-Higgs bundles, where $G_i$ is the corresponding real group
  $\Sp(2n_i,\RR)$, $\U(p_i,q_i)$, $\U(n_i)$ or $\GL(n_i,\RR)$.
\end{theorem}

\begin{proof}
  This follows from Propositions~\ref{prop:polystable-spnr-higgs},
  \ref{cor:smooth-points} and \ref{prop:smooth-u-upq-glnr} and
  Theorems~\ref{thm:stable-not-simple-spnr-higgs} and
  \ref{thm:stability-equivalence-sp2nC}
\end{proof}

%%%%%%%%%%%%%%%%%%%%%%%%%%%%%%%%%%%%%%%%%%%%%%%%%%%%%%%%%%%%%%%%%%%%%%%%%%%%%%%%%%%
\section{Maximal degree  $\Sp(2n,\RR)$-Higgs bundles and the Cayley correspondence}
\label{maximal}
%%%%%%%%%%%%%%%%%%%%%%%%%%%%%%%%%%%%%%%%%%%%%%%%%%%%%%%%%%%%%%%%%%%%%%%%%%%%%%%%%%%

%%%%%%%%%%%%%%%%%%%%%%%%%%%%%%%%%%%
\subsection{Cayley correspondence}
%%%%%%%%%%%%%%%%%%%%%%%%%%%%%%%%%%%

In this section we shall describe the $\spn$ moduli space for the
extreme value $\abs{d} = n(g-1)$.  In fact, for the rest of this
section we shall assume that $d = n(g-1)$.  This involves no loss of
generality, since, by Proposition \ref{prop:duality-iso}, $(V,\varphi)
\mapsto (V^*,\varphi^t)$ gives an isomorphism between the $\spn$
moduli spaces for $d$ and $-d$. The main result is
Theorem~\ref{thm:cayley}, which we refer to as the \emph{Cayley
 correspondence}. This is stated as Theorem~\ref{cayley} in the
Introduction, where the reason for the name is also explained.

Let  $(V,\beta,\gamma)$ be an $\Sp(2n,\RR)$-Higgs bundle with
$d=n(g-1)$ such that $\gamma\in H^0(K\otimes S^2V^*)$ is an
isomorphism. Let $L_0=K^{-1/2}$ be a fixed  square root of $K^{-1}$, and
define $ W:=V\otimes L_0$. Then $q:=\gamma\otimes I_{L_0}:
W \to W^*$ is a symmetric isomorphism defining a non-degenerate 
symmetric bilinear  form $Q$ on $W$,  in other words,  $(W,Q)$ is an
$\O(n,\CC)$-holomorphic bundle. The $K^2$-twisted endomorphism
$\psi:W\to W \otimes K^2$ defined by $\psi: = 
\beta \otimes I_{L_0^{-1}}\circ  (\gamma\otimes I_{L_0})$ is $Q$-symmetric and
hence $(W,Q,\psi)$ defines a $K^2$-twisted $\GL(n,\RR)$-Higgs
pair (in the sense of
Section \ref{glnr-higgs}), 
from which we can recover the original $\Sp(2n,\RR)$-Higgs bundle.

We shall need the following Lemma. In the statement  we use the notions
of $\beta$- and $\gamma$-invariance of filtrations introduced in
Definition~\ref{def:beta-gamma-invariant}.

\begin{lemma}
  \label{lem:cayley-filtrations}
  Let $(V,\beta,\gamma)$ be a $\Sp(2n,\RR)$-Higgs bundle with
  $d=n(g-1)$ such that $\gamma$ is an isomorphism. Let $(W,Q,\psi)$ be
  the corresponding $K^2$-twisted $\GL(n,\RR)$-Higgs pair. Let
  \begin{equation}
    \label{eq:V-filtration}
    0 \subset V_1 \subset V_2 \subset V
  \end{equation}
  be a filtration and let $W_1 \subset W$ be the subbundle defined by
  $W_1 = V_1 \otimes L_0$. Then the following statements hold:
  \begin{enumerate}
  \item The filtration (\ref{eq:V-filtration}) is
    $\gamma$-invariant if and only if $W_1 \subset W$ is isotropic.
  \item Assume that the filtration (\ref{eq:V-filtration}) is
    $(\beta,\gamma)$-invariant. Then $W_1
    \subset W$ is $\psi$-invariant.
  \item Assume that $V_2 = V_1^{\perp_\gamma}$. Then the filtration
    (\ref{eq:V-filtration}) is $(\beta,\gamma)$-invariant if and only $W_1\subset
    W$ is isotropic and $\psi$-invariant.
  \end{enumerate}
\end{lemma}

\begin{proof}
  Statement (1) is immediate using that $\gamma$-invariance of the
  filtration is equivalent to $V_1^{\perp_\gamma} \subset V_2$ (see
  Remark~\ref{K-pairing}).

  For Statement (2), note that $\psi=(\beta\otimes
  \Id_{L_0^{-1}})\circ(\gamma\otimes\Id_{L_0})$. Hence, using the conditions for
  $(\beta,\gamma)$-invariance given in Remark~\ref{K-pairing}, we obtain
  \begin{displaymath}
    \psi(W_1) = (\beta\otimes
    \Id_{L_0^{-1}})(\gamma(V_1)\otimes L_0)
    \subset (\beta\otimes \Id_{L_0^{-1}})(V_2^{\perp}\otimes L_0^{-1})
    \subset V_1 \otimes L_0 K^2
    = W_1\otimes K^2.
  \end{displaymath}

  In order to prove Statement (3), in view of (1) and (2), we
  can use (\ref{eq:perp-gamma}) of Remark~\ref{K-pairing},
  $\psi$-invariance of $W_1$ and $V_2^{\perp_\gamma} = V_1 =
  W_1\otimes L_0^{-1}$ to deduce that
  \begin{displaymath}
    \beta(V_2^{\perp}) = \beta
    (\gamma\otimes\Id_{K^{-1}}(V_2^{\perp_\gamma}\otimes K^{-1}))
    = (\psi\otimes\Id_{L_0})(W_1\otimes L_0)
    \subset W_1 \otimes L_0K^2
    = V_1\otimes K.
  \end{displaymath}
  Thus the result follows from the $\beta$-invariance condition given
  in Remark~\ref{rem:matrix-form}.

\end{proof}

\begin{theorem}\label{equivalence-stability}
Let  $(V,\beta,\gamma)$ be a   $\Sp(2n,\RR)$-Higgs bundle with $d=n(g-1)$
such that $\gamma$ is an isomorphism. Let $(W,Q,\psi)$ be the corresponding
$K^2$-twisted $\GL(n,\RR)$-Higgs pair. Then $(V,\beta,\gamma)$ is semistable
(respectively stable, polystable)
if and only if $(W,Q,\psi)$ is semistable
(respectively stable, polystable).
\end{theorem}

\begin{proof}
We start by proving the equivalence of the semistability conditions.

Assume that $(V,\beta,\gamma)$ is a semistable $\Sp(2n,\RR)$-Higgs
bundle with $d=n(g-1)$, so that $\gamma$ is an isomorphism. Let
$W_1\subset W$ be a $\psi$-invariant $Q$-isotropic subbundle, i.e.
\begin{equation}\label{Q-isotropic}
 W_1\subset W_1^{\perp_Q}.
\end{equation}
We shall prove that $\deg(W_1)\leq 0$. Define
\begin{displaymath}
  V_1:=W_1\otimes L_0^{-1}\quad \text{and} \quad V_2:=W_1^{\perp_Q}\otimes L_0^{-1},
\end{displaymath}
and note that
\begin{equation}\label{deg-g}
  \deg(W_1) = \deg(V_1)-\rk(V_1)(g-1).
\end{equation}
By (\ref{Q-isotropic}) we have a filtration $0\subset V_1\subset
V_2\subset V$, which moreover satisfies
$V_2=V_1^{\perp_\gamma}$. Thus, by (3) of
Lemma~\ref{lem:cayley-filtrations}, it is $(\beta,\gamma)$-invariant and so, by
semistability of $(V,\beta,\gamma)$, we have
\begin{equation}
  \label{eq:beta-gamma-ss}
  \deg(V)-\deg(V_1)-\deg(V_2) \geq 0.
\end{equation}

Now, we have an exact sequence
\begin{equation}
0\to V_1^{\perp_{\gamma}}\to V\to V_1^*\otimes K\to 0,
\label{eq:deg-V-perp-gamma}
\end{equation}
defined by composing the isomorphism $\gamma$ with the projection 
$V^*\otimes K\to V_1^*\otimes K$.
Thus
\begin{equation}
\label{eq:deg-V-perp-gamma}
\deg(V_1^{\perp_{\gamma}})=\deg(V)+\deg(V_1)-(2g-2)\rk(V_1)
\end{equation}
and, since $V_2=V_1^{\perp_\gamma}$, it follows that
\begin{equation}
  \label{eq:deg-V-V1}
  \deg(V)-\deg(V_1)-\deg(V_2) = -2\deg(V_1)+(2g-2)\rk(V_1)
  = -2\deg(W_1),
\end{equation}
where we have used (\ref{deg-g}) in the second identity. Hence
(\ref{eq:beta-gamma-ss}) implies that $\deg(W_1)\leq 0$ and therefore,
according to Definition~\ref{spnr-stability}, $(W,Q,\psi)$ is
semistable.

To prove the converse, assume now that $(W,Q,\psi)$ is 
semistable. Consider a $(\beta,\gamma)$-invariant filtration 
$$
0\subset V_1\subset V_2\subset V.
$$
Let
\begin{displaymath}
  W_1 := V_1 \otimes L_0.
\end{displaymath}
Then (1) and (2) of Lemma~\ref{lem:cayley-filtrations} imply that
$W_1$ is isotropic and $\psi$-invariant. Hence, by semistability of
$(W,Q,\psi)$, we have $\deg(W_1) \leq 0$, i.e.,
\begin{equation}
  \label{eq:deg-V1}
  \deg(V_1) \leq \rk(V_1)(g-1).
\end{equation}

In order to complete the proof of semistability of $(V,\beta,\gamma)$,
we shall consider a second $\psi$-invariant $Q$-isotropic subbundle
$W_2\subset W$. For this, we view $V_2\cap V_2^{\perp_\gamma} \subset
V$ as a subbundle (i.e., define it in sheaf theoretic terms and take
its saturation) and let
\begin{displaymath}
  W_2 = V_2\cap V_2^{\perp_\gamma} \otimes L_0 \subset W.
\end{displaymath}
Similarly, we consider the subbundle $V_2+V_2^{\perp_\gamma} \subset
V$. Define
\begin{displaymath}
  V_1'= V_2\cap V_2^{\perp_{\gamma}}\quad\text{and}\quad
  V_2'= V_2+V_2^{\perp_\gamma}.
\end{displaymath}
Then we have a filtration
\begin{equation}
  \label{eq:pprime-filtration}
  0 \subset V_1' \subset V_2' \subset V.
\end{equation}
We shall see that this filtration is $(\beta,\gamma)$-invariant and that
$V_2' \subset V_1'^{\perp_\gamma}$. Once this is established, it
will follow from (3) of Lemma~\ref{lem:cayley-filtrations}
that $W_2$ is isotropic and $\psi$-invariant.

We start by showing $\gamma$-invariance. There is an inclusion
\begin{equation}
V_2+V_2^{\perp_{\gamma}}\subset (V_2\cap
V_2^{\perp_{\gamma}})^{\perp_{\gamma}},
\label{eq:VVperp-inc}
\end{equation}
which, in turn follows from $V_2\subset (V_2\cap
V_2^{\perp_{\gamma}})^{\perp_{\gamma}}$ and
$V_2^{\perp_{\gamma}}\subset (V_2\cap
V_2^{\perp_{\gamma}})^{\perp_{\gamma}}$, which are obvious. Hence
$V_2' \subset V_1'^{\perp_\gamma}$, from which $\gamma$-invariance
of (\ref{eq:pprime-filtration})  follows (cf.\ Remark~\ref{K-pairing}). To show $\beta$-invariance,
note that we have inclusions
$$V_1\subset V_1'=V_2\cap V_2^{\perp_{\gamma}},\qquad V_2\subset V_2'=V_2+V_2^{\perp_{\gamma}}.$$
The first one follows from $V_1\subset V_2$ and from applying $^{\perp_{\gamma}}$ to the
inclusion $V_2\subset V_1^{\perp_{\gamma}}$, which gives
$V_1\subset V_2^{\perp_{\gamma}}$. The second inclusion is
obvious. From the first inclusion we have $V_1'^{\perp} \subset
V_1^{\perp}$ and hence, using the second inclusion and $\beta$-invariance of the original
filtration $0 \subset V_1 \subset V_2 \subset V$ we have
\begin{displaymath}
  \beta(V_1'^{\perp}) \subset \beta(V_1^\perp) \subset V_2\otimes K
  \subset V_2' \otimes K.
\end{displaymath}
Thus (\ref{eq:pprime-filtration}) is also $\beta$-invariant (using
Remark~\ref{K-pairing}) and we have established $(\beta,\gamma)$-invariance of
(\ref{eq:pprime-filtration}). Note that we did not use that $\gamma$
is non-degenerate for this.

Now, since $\gamma$ is non-degenerate, the sequence
\begin{equation}
\label{eq:V2V2perp-ses}
  0\to V_2\cap V_2^{\perp_{\gamma}}\to V_2\oplus V_2^{\perp_{\gamma}}
  \to V_2+V_2^{\perp_{\gamma}}\to 0
\end{equation}
is generically short exact. Hence both bundles in
(\ref{eq:VVperp-inc}) have the same rank and, therefore,
\begin{displaymath}
  V_2+V_2^{\perp_{\gamma}} = (V_2\cap V_2^{\perp_{\gamma}})^{\perp_{\gamma}}.
\end{displaymath}
In other words, $V_2' = V_1'^{\perp_\gamma}$ and, as noted above, it
follows from (3) of Lemma~\ref{lem:cayley-filtrations} that $W_2
\subset W$ is isotropic and $\psi$-invariant. Thus, by semistability
of $(W,Q,\psi)$ we have that $\deg(W_2) \leq 0$. This, arguing as in
(\ref{eq:deg-V-V1}), is equivalent to
\begin{equation}
  \label{eq:deg-V-primes}
  \deg(V) - \deg(V_1') - \deg(V_2') \geq 0.
\end{equation}
But, using the generically short exact sequence
(\ref{eq:V2V2perp-ses}), we have that
\begin{displaymath}
  \deg(V_1') + \deg(V_2') \geq \deg(V_2) + \deg(V_2^{\perp_\gamma})
\end{displaymath}
Hence (\ref{eq:deg-V-primes}) implies that
\begin{equation}
  \label{eq:deg-V2-V2perp}
  \deg(V) - \deg(V_2) - \deg(V_2^{\perp_\gamma}) \geq 0
\end{equation}
and so, using (\ref{eq:deg-V-perp-gamma}) for $V_2^{\perp_\gamma}$, we obtain
\begin{displaymath}
  \deg(V_2) \leq (g-1)\rk(V_2).
\end{displaymath}
Adding this to (\ref{eq:deg-V1}) we get
$$\deg(V_1)+\deg(V_2) \leq (g-1)\rk(V_1)+(g-1)\rk(V_2).$$
Now we can estimate, using the inclusion $V_2\subset V_1^{\perp_{\gamma}}$,
\begin{equation}
  \label{eq:rkV2-estimate}
  \rk(V_2)\leq \rk(V_1)^{\perp_{\gamma}}=\rk(V)-\rk(V_1),
\end{equation}
where $\rk(V_1^{\perp_{\gamma}})=\rk(V)-\rk(V_1)$ follows from the fact that
$\gamma$ is nondegenerate.
Consequently,
$$\deg(V_1)+\deg(V_2) \leq (g-1)\rk(V).$$
On the other hand, the nondegeneracy of $\gamma$ implies that $V\simeq
V^*\otimes K$, which gives $\deg(V)=(g-1)\rk(V)$. Summing up, we
finally get
$$\deg(V_1)+\deg(V_2) \leq \deg(V),$$
concluding the proof of semistability of $(V,\beta,\gamma)$

The proof of the statement for stability is essentially the same,
observing that the trivial filtration $0=V_1\subset V_2=V$ corresponds
to the trivial subbundle $0\subset W$.
 
We turn now to the proof of the equivalence of the polystability
conditions. We will first prove that if $(V,\beta,\gamma)$ is
polystable then $(W,Q,\psi)$ is also polystable. For this, suppose
that $W'\subset W$ is a $\psi$-invariant $Q$-isotropic non-trivial
subbundle such that $\deg(W')=0$. According to Definition
\ref{thm:orthogonal-stability}, we have to show that there exists a
$Q$-coisotropic $\psi$-invariant subbundle $W''\subset W$ such that
$W=W'\oplus W''$.  We proceed as in the proof that semistability of
$(V,\beta,\gamma)$ implies semistability of $(W,Q,\psi)$. So, let
$V_1:=W'\otimes L_0^{-1}$ and $V_2:=W'^{\perp_Q}\otimes L_0^{-1}$. Our
previous arguments now show that the filtration $0\subset V_1\subset
V_2\subset V$ is a non-trivial $(\beta,\gamma)$-invariant filtration,
and that $\deg(W')=0$ is equivalent to
$$
\deg(V)-\deg(V_1) -\deg (V_2)=0.
$$
Moreover, since $(V,\beta,\gamma)$ is polystable, from Definition
\ref{spn-poly} and Remark \ref{f-decomposition}, there are subbundles
$F_1$ $F_2$ and $F_3$ of $V$, such that $V=F_1\oplus F_2\oplus F_3$,
with $V_1=F_1$, $V_2=F_1\oplus F_2$, and $\beta(F_i^*)\subset
F_{4-i}\otimes K$ and $\gamma(F_i)\subset F_{4-i}^*\otimes K$ for
$i=1,2,3$.  Let $W'':=(F_2\oplus F_3)\otimes L_0$. The
$\psi$-invariance of $W''$ follows clearly from the
$(\beta,\gamma)$-invariance of the $F_i$. Moreover, the
$\gamma$-invariance of the $F_i$ implies that $(F_2 \oplus
F_3)^{\perp_\gamma} = F_3$, i.e. $W''^{\perp_Q}\subset W''$, which is
the coisotropic condition.

To prove the converse, suppose that $(W,Q,\psi)$ is
polystable. Consider a general $(\beta,\gamma)$-invariant filtration
$0\subset V_1\subset V_2\subset V$ with
\begin{equation}
  \label{eq:deg-VV1V2=0}
  \deg(V)-\deg(V_1)-\deg(V_2) = 0.
\end{equation}
The arguments of the proof that semistability of $(W,Q,\psi)$ implies
semistability of $(V,\beta,\gamma)$ of course continue to hold and,
moreover, the condition
(\ref{eq:deg-VV1V2=0}) forces all inequalities in that proof to be
equalities. In particular, equality must hold in
(\ref{eq:rkV2-estimate}) and so we conclude that in fact
\begin{displaymath}
V_2=V_1^{\perp_\gamma}.
\end{displaymath}
Therefore, it suffices to consider $(\beta,\gamma)$-invariant
filtrations of the form $0\subset V_1\subset V_1^{\perp_\gamma}\subset
V$ with
$$\deg(V)-\deg(V_1)-\deg(V_1^{\perp_\gamma})=0.$$
Now, running backwards the argument given above that polystability of
$(V,\beta,\gamma)$ implies polystability of $(W,Q,\psi)$, we see that
there is a decomposition $V= F_1 \oplus F_2 \oplus F_3$ with
$V_1=F_1$, $V_1^{\perp_\gamma}=F_1\oplus F_2$ and satisfying the
$(\beta,\gamma)$-invariance conditions.

This concludes the proof.

\end{proof}

\begin{theorem}
\label{thm:cayley}
Let $\cM_{\max}$ be the moduli space of polystable
$\Sp(2n,\RR)$-Higgs bundles with $d=n(g-1)$ and let $\cM'$ be
the moduli space of polystable $K^2$-twisted $\GL(n,\RR)$-Higgs pairs.
The map $(V,\beta,\gamma)\mapsto (W,Q,\psi)$ defines an
isomorphism of complex algebraic varieties
$$
\cM_{{\mathrm{max}}}\simeq \cM'.
$$
\end{theorem}

\begin{proof}
  Let $(V,\beta,\gamma)$ be a semistable $\Sp(2n,\RR)$-Higgs bundle
  with $d=n(g-1)$.  By Proposition \ref{mw-higgs}, $\gamma$ is an
  isomorphism and hence the map $(V,\beta,\gamma)\mapsto (W,Q,\psi)$
  is well defined.  The result follows now from
  Theorem~\ref{equivalence-stability} 
and the existence of local
  universal families (see \cite{schmitt:2008}).
\end{proof}

%%%%%%%%%%%%%%%%%%%%%%%%%%%%%%%%%%%%%%%%%%%%%%%%%%%%
\subsection{Invariants of $\GL(n,\RR)$-Higgs pairs}
%%%%%%%%%%%%%%%%%%%%%%%%%%%%%%%%%%%%%%%%%%%%%%%%%%%%

To a $K^2$-twisted $\GL(n,\RR)$-Higgs pair $(W,Q,\psi)$ one can attach
topological invariants corresponding to the first and second
Stiefel-Whitney classes of a reduction to $\O(n)$ of the $\O(n,\CC)$
bundle defined by $(W,Q)$. The first class $w_1\in H^1(X,\ZZ_2)\simeq
\ZZ_2^{2g}$ measures the obstruction for the $\O(n)$-bundle to have an
orientation, i.e.\ to the existence of a reduction to a $\SO(n)$
bundle, while the second one $w_2\in H^2(X,\ZZ_2)\simeq \ZZ_2$ measures
the obstruction to lifting the $\O(n)$-bundle to a $\Pin(n)$-bundle,
where
\begin{displaymath}
  1 \to \ZZ_2 \to \Pin(n)  \to \O(n) \to 1.
\end{displaymath}
If we define
$$
\mathcal{\cM'}_{w_1,w_2}:=\{(W,Q,\psi)\in \mathcal{M}'\;\;\;\mbox{such that}
\;\;\; w_1(W)=w_1 \;\;\; \mbox{and}\;\;\; w_2(W)=w_2\},
$$
we have that
\begin{equation}\label{stiefel-whitney}
\cM'=\bigcup_{w_1,w_2} \cM'_{w_1,w_2}.
\end{equation}

We thus have, via the isomorphism given by Theorem \ref{thm:cayley},
that the moduli space $\cM_{\max}$ is partitioned in disjoint closed
subvarieties corresponding to fixing $(w_1,w_2)$.

%%%%%%%%%%%%%%%%%%%%%%%%%%%%%%%%
\section{The Hitchin functional}
\label{sec:connected-components}
%%%%%%%%%%%%%%%%%%%%%%%%%%%%%%%%

%%%%%%%%%%%%%%%%%%%%%%%%%%%%%%%%%%%%%
\subsection{The Hitchin functional}
\label{sec:morse}
%%%%%%%%%%%%%%%%%%%%%%%%%%%%%%%%%%%%%

In order to define this functional, we consider the moduli space
of $\Sp(2n,\RR)$-Higgs bundles $(V,\varphi)$ from the gauge theory
point of view, i.e., we use the identification of $\mathcal{M}_d$
with the moduli space $\Mg_d$ of solutions $(A,\varphi)$ to the
Hitchin equations given by Theorem \ref{higgs-hk}. There is an action of
$S^1$ on $\mathcal{M}_d$ via multiplication of $\varphi$ by
scalars:
\begin{math}
  (A,\varphi) \mapsto (A,e^{\mathbf{i}\theta}\varphi)
\end{math}.  Restricted to the smooth locus $\mathcal{M}^s_d$ this
action is Hamiltonian with symplectic moment map $-f$, where the
\emph{Hitchin functional} $f$ is defined by
\begin{equation}
  \label{eq:2.7}
  \begin{aligned}
    f\colon\mathcal{M}_d &\to \RR,\\ (A,\varphi) &\mapsto
    \tfrac{1}{2}\norm{\varphi}^2
    =\tfrac{1}{2}\norm{\beta}^2+\tfrac{1}{2}\norm{\gamma}^2.
  \end{aligned}
\end{equation}
Here $\norm{\cdot}$ is the $L^2$-norm obtained by using the
Hermitian metric in $V$ and integrating over $X$. The function $f$
is well defined on the whole moduli space (not just on the smooth
locus).  It was proved by Hitchin
\cite{hitchin:1987a,hitchin:1992} that $f$ is proper and therefore
it has a minimum on every closed subspace of $\mathcal{M} =
\bigcup_d \mathcal{M}_d$. Thus we have the following result.

\begin{proposition}\label{prop:ps-top}
  Let $\mathcal{M}' \subseteq \mathcal{M}$ be any closed subspace and
  let $\mathcal{N}' \subseteq \mathcal{M}'$ be the subspace of local
  minima of $f$ on $\mathcal{M}'$.  If $\mathcal{N}'$ is connected
  then so is $\mathcal{M}'$. \qed
\end{proposition}
The following observation was also made by Hitchin \cite{hitchin:1992}.
\begin{proposition}\label{prop:hitchin-additive}
  The Hitchin functional is additive with respect to direct sum of
  $\Sp(2n,\RR)$-Higgs bundles, in other words,
  \begin{displaymath}
    f(\bigoplus (V_i,\varphi_i)) = \sum f(V_i,\varphi_i).
  \end{displaymath}
\end{proposition}

Let $(V,\varphi)$ represent a smooth point of $\mathcal{M}_d$. Then the
moment map condition shows that the critical points of $f$ are exactly
the fixed points of the circle action.  These can be identified as
follows (cf.\ \cite{hitchin:1987a,hitchin:1992,simpson:1992}).

\begin{proposition}
  \label{prop:VHS}
  An $\Sp(2n,\RR)$-Higgs bundle $(V,\varphi)$ represents a fixed point
  of the circle action on $\mathcal{M}_d$ if and only if it is a
  \emph{complex variation of Hodge structure} (also called a
  \emph{Hodge bundle}): this means that there is a decomposition in
  holomorphic subbundles
  \begin{displaymath}
    V = \bigoplus F_{i}
  \end{displaymath}
  for real indices, or \emph{weights}, $i$ such that, attributing
  weight $-i$ to $F_i^*$, $\varphi=(\beta,\gamma)$ has weight one with
  respect to this decomposition; more explicitly this means that
  \begin{displaymath}
    \gamma\colon F_{i} \to F^*_{-i-1}\otimes K
    \qquad\text{and}\qquad
    \beta\colon F^*_{i} \to F_{-i+1}\otimes K.
  \end{displaymath}
\end{proposition}

The decomposition $V = \bigoplus F_i$ of Proposition~\ref{prop:VHS}
gives rise to corresponding decompositions
\begin{align}
  \label{eq:deceV}
  \End(V)_k&=\bigoplus_{i-j=k}F_i\otimes F_j^* , \\
  \label{eq:decV}
  (S^2V\otimes K)_{k+1}&=\bigoplus_{\substack{i+j=k+1 \\i<j}}
  F_i\otimes F_j\otimes K\oplus S^2F_{\frac{k+1}{2}}\otimes K, \\
  \label{eq:decVd}
  (S^2V^*\otimes K)_{k+1}&=\bigoplus_{\substack{-i-j=k+1\\ i<j}}
  F_i^*\otimes F_j^*\otimes K\oplus S^2F^*_{-\frac{k+1}{2}}\otimes K.
\end{align}
The map $\ad(\varphi)$ in the deformation complex
(\ref{eq:def-complex}) has weight $1$ with respect to these
decompositions, so that we can define complexes
\begin{equation}
\label{eq:mapAd}
  C^\bullet_k(V,\varphi) \colon \End(V)_k \xrightarrow{\ad(\varphi)}
  (S^2V\otimes K\oplus S^2V^*\otimes K)_{k+1}, \\
\end{equation}
for any $k$.  The deformation complex (\ref{eq:def-complex}) decomposes
accordingly as
\begin{displaymath}
  C^\bullet(V,\varphi) = \bigoplus C^\bullet_k(V,\varphi).
\end{displaymath}
We shall also need the positive weight subcomplex
\begin{equation}\label{eq:def-C-minus}
  C^\bullet_-(V,\varphi) = \bigoplus_{k>0} C^\bullet_k(V,\varphi).
\end{equation}
It can be shown (see, e.g., \cite[\S 3.2]{garcia-gothen-munoz:2007})
that $\HH^1(C^\bullet_k(V,\varphi))$ is the weight $-k$-subspace of
$\HH^1(C^\bullet(V,\varphi))$ for the infinitesimal circle
action. Thus $\HH^1(C^\bullet_-(V,\varphi))$ is the positive weight
space for the infinitesimal circle action.

\begin{proposition}
  \label{prop:minima-HH1}
  Let $(V,\varphi)$ be a polystable $\Sp(2n,\RR)$-Higgs bundle whose isomorphism
  class is fixed under the circle action.

  \begin{enumerate}
  \item Assume that $(V,\varphi)$ is simple and stable as an
    $\Sp(2n,\CC)$-Higgs bundle.  Then $(V,\varphi)$ represents a local
    minimum of $f$ if and only if $\HH^1(C^\bullet_-(V,\varphi))=0$.
  \item\label{item:non-minima} Suppose that there is a family
    $(V_t,\varphi_t)$ of polystable $\Sp(2n,\RR)$-Higgs bundles,
    pa\-ra\-me\-trized by $t$ in the open unit disk $D$, deforming
    $(V,\varphi)$ (i.e., such that $(V_0,\varphi_0) = (V,\varphi)$)
    and that the corresponding infinitesimal deformation is a non-zero
    element of $\HH^1(C^\bullet_-(V,\varphi))$.  Then $(V,\varphi)$ is
    not a local minimum of $f$ on $\mathcal{M}_d$.
  \end{enumerate}
\end{proposition}

\begin{proof}
  (1) From Proposition~\ref{prop:smoothness}, when the hypotheses are
  satisfied, $(V,\varphi)$ represents a smooth point of the moduli
  space. Then one can use the moment map condition on $f$ to show that
  $\HH^1(C^\bullet_k(V,\varphi))$ is the eigenvalue $-k$ subspace of
  the Hessian of $f$ (cf.\ \cite[\S 3.2]{garcia-gothen-munoz:2007};
  this goes back to Frankel \cite{frankel:1959}, at least).  This
  proves (1).

  (2) Take a corresponding family of solutions to Hitchin's equations.
  One can then prove that the second variation of $f$ along this
  family is negative in certain directions (see Hitchin
  \cite[\S~8]{hitchin:1992}).
\end{proof}

%%%%%%%%%%%%%%%%%%%%%%%%%%%%%%%%%%%%%%%%%%%%%%%%%%%%%
\subsection{A cohomological criterion for minima}
\label{sec:cohom-crit-minima}
%%%%%%%%%%%%%%%%%%%%%%%%%%%%%%%%%%%%%%%%%%%%%%%%%%%%%

The following result was proved in \cite[Proposition~4.14\footnote{a
  corrected proof can be found in
  \cite[Lemma~3.11]{bradlow-garcia-prada-gothen:homotopy}} and
Remark~4.16]{bradlow-garcia-prada-gothen:2003}.  It is the key to
obtaining the characterization of the minima of the Hitchin functional
$f$.

\begin{proposition}
  \label{prop:ad-iso}
  Let $(V,\varphi)$ be a polystable $\Sp(2n,\RR)$-Higgs bundle whose
  isomorphism class is fixed under the circle action.  Then for any
  $k$ we have $\chi(C^\bullet_k(V,\varphi)) \leq 0$ and equality holds
  if and only if
  \begin{displaymath}
    \ad(\varphi)\colon \End(V)_k \to
  (S^2V\otimes K\oplus S^2V^*\otimes K)_{k+1}
  \end{displaymath}
  is an isomorphism.
\end{proposition}

\begin{corollary}\label{cor:ad-iso-minima}
  Let $(V,\varphi)$ be a simple $\Sp(2n,\RR)$-Higgs bundle which is
  stable as an $\Sp(2n,\CC)$-Higgs bundle. If $(V,\varphi)$ is fixed
  under the circle action then it represents a local minimum of $f$ if
  and only if the map
  \begin{displaymath}
    \ad(\varphi)\colon \End(V)_k \to
  (S^2V\otimes K\oplus S^2V^*\otimes K)_{k+1}
  \end{displaymath}
  is an isomorphism for all $k>0$.
\end{corollary}

\begin{proof}
  We have the vanishing
  $\HH^0(C_k^\bullet(V,\varphi))=\HH^2(C_k^\bullet(V,\varphi))=0$ for all
  $k>0$ from Proposition~\ref{prop:hh-vanishing}.  Hence $\dim
  \HH^1(C^\bullet_-(V,\varphi))=-\chi(C^\bullet_-(V,\varphi))$. Now the
  result is immediate from Proposition~\ref{prop:ad-iso} and (1) of
  Proposition~\ref{prop:minima-HH1}.
\end{proof}

%%%%%%%%%%%%%%%%%%%%%%%%%%%%%%%%%%%%%%%%%%%%
\subsection{Minima of the Hitchin functional}
\label{sec:minima-hitchin-functional}
%%%%%%%%%%%%%%%%%%%%%%%%%%%%%%%%%%%%%%%%%%%%

In order to describe the minima, it is convenient to define
the following subspaces of $\mathcal{M}_d$.

\begin{definition}
  \label{def:N_d}
  For each $d$, define the following subspace of $\mathcal{M}_d$.
  \begin{displaymath}
    \mathcal{N}_d = \{(V,\beta,\gamma) \in \mathcal{M}_d
    \suchthat \text{$\beta=0$ or $\gamma=0$}\}.
  \end{displaymath}
\end{definition}

\begin{remark}
\label{rem:N(d)-sign}
It is easy to see that polystability of $(V,\varphi)$ implies that, in
fact,
\begin{align*}
  \mathcal{N}_d &= \{(V,\beta,\gamma)
    \suchthat \beta=0\} &\text{for $d>0$,}\\
  \mathcal{N}_d &= \{(V,\beta,\gamma)
    \suchthat \gamma=0\} &\text{for $d<0$,}\\
  \mathcal{N}_d &= \{(V,\beta,\gamma)
    \suchthat \beta=\gamma=0\} &\text{for $d=0$.}
  \end{align*}
  Note, in particular, that for $d=0$ the vanishing of one of the
  sections $\beta$ or $\gamma$ implies the vanishing of the other one.
\end{remark}

\begin{proposition}
  \label{prop:N(d)-minima}
  Let $(V,\varphi)$ be a polystable $\Sp(2n,\RR)$-Higgs bundle with $\beta=0$ or
  $\gamma=0$.  Then $(V,\varphi)$ represents the absolute minimum of $f$ on
  $\mathcal{M}_d$. Thus $\mathcal{N}_d$ is contained in the subspace
  of local minima of $f$ on $\mathcal{M}_d$.
\end{proposition}

\begin{proof}
  This can be seen in a way similar to the proof of
  \cite[Proposition~4.5]{bradlow-garcia-prada-gothen:2003}.
\end{proof}

\begin{theorem}\label{thm:all-minima}
  Let $(V,\beta,\gamma)$ be a polystable $\Sp(2n,\RR)$-Higgs bundle
  and assume that $n \geq 3$.  Then $(V,\beta,\gamma)$ represents a
  minimum of the Hitchin functional if and only if one of the
  following situations occurs:
  \begin{enumerate}
  \item $(V,\beta,\gamma)$ belongs to $\mathcal{N}_d$.

  \item The degree $d = - n(g-1)$ with $n=2q+1$ odd, and there exists a
    square root $L$ of $K$ such that the bundle $V$ is of the form
  \begin{displaymath}
    V = \bigoplus_{\lambda=-q}^{q}L^{-1}K^{-2\lambda}.
  \end{displaymath}
  With respect to this decomposition of $V$ and the
  corresponding decomposition of $V^*$, the maps $\beta$ and $\gamma$
  are of the form:
  \begin{displaymath}
    \beta =
      \begin{pmatrix}
        0 &  \cdots & 0 & 1 \\
        \vdots & \iddots &  \iddots & 0 \\
        0 & 1 &  \iddots & \vdots \\
        1 & 0 & \cdots & 0
      \end{pmatrix}
      \qquad\text{and}\qquad
    \gamma =
       \begin{pmatrix}
         0 &  \cdots & 0 & 0 \\
         \vdots & \iddots  & \iddots & 1 \\
         0 & 0 & \iddots & \vdots \\
         0 & 1 & \cdots & 0
      \end{pmatrix}
  \end{displaymath}
  where, in the matrix for $\beta$, we denote by $1$ the canonical
  section of
  $$
    \Hom((L^{-1}K^{-2\lambda})^*, L^{-1}K^{2\lambda})\otimes K \simeq
    \mathcal{O}
  $$
  and analogously for $\gamma$.

\item The degree $d = - n(g-1)$ with $n=2q+2$ even, and there exists a
  square root $L$ of $K$ such that the bundle $V$ is of the form
  \begin{displaymath}
    V = \bigoplus_{\lambda=-q}^{q+1}LK^{-2\lambda}.
  \end{displaymath}
  With respect to this decomposition of $V$ and the
  corresponding decomposition of $V^*$, the maps $\beta$ and $\gamma$
  are of the form given above.

\item The degree $d = n(g-1)$ and the dual $\Sp(2n,\RR)$-Higgs bundle
  $(V',\beta',\gamma') = (V^*,\gamma^t,\beta^t)$ is of the form given
  in (2) or (3) above.

\end{enumerate}
\end{theorem}

\begin{definition}
  If $(V,\beta,\gamma)$ is a minimum which does not belong to
  $\mathcal{N}_d$ we say that it is a \textbf{quiver type} minimum.
\end{definition}

\begin{remark}
  The cases $n=1$ and $n=2$ are special and were treated in
  \cite{hitchin:1987a} and \cite{gothen:2001}, respectively (cf.\ (1)
  of Corollary~\ref{cor:minima} and Remark~\ref{rem:n-one-hodge}).
\end{remark}

\begin{proof}[Proof of Theorem~\ref{thm:all-minima}]
  This proof relies on the results of
  Sections~\ref{sec:stable-simple-minima} and
  \ref{sec:poly-stable-minima} below.

  Consider first the case of simple $\Sp(2n,\RR)$-Higgs bundles
  $(V,\varphi)$ which are stable as $\Sp(2n,\CC)$-Higgs bundles. In
  this case, the analysis of the minima is based on
  Corollary~\ref{cor:ad-iso-minima} and is carried out in
  Section~\ref{sec:stable-simple-minima} below. The main result is
  Theorem~\ref{thm:smooth-minima}, which says that
  Theorem~\ref{thm:all-minima} holds for such $(V,\varphi)$.

  Next, consider a polystable $\Sp(2n,\RR)$-Higgs bundle $(V,\varphi)$
  which is not simple and stable as an $\Sp(2n,\CC)$-Higgs bundle. Then
  the decomposition $(V,\varphi) = \bigoplus (V_i,\varphi_i)$ given in
  the structure Theorem~\ref{thm:polystable-spnr-higgs} is
  non-trivial. The main result of
  Section~\ref{sec:poly-stable-minima},
  Proposition~\ref{prop:strict-polystable-minima}, says that if such a
  $(V,\varphi)$ is a local minimum then it belongs to $\mathcal{N}_d$,
  i.e., $\beta=0$ or $\gamma=0$. This concludes the proof.
 \end{proof}

%%%%%%%%%%%%%%%%%%%%%%%%%%%%%%%%
\section{Minima in the smooth locus of the moduli space}
\label{sec:stable-simple-minima}
%%%%%%%%%%%%%%%%%%%%%%%%%%%%%%%%

In this section we consider simple $\Sp(2n,\RR)$-Higgs bundles
$(V,\phi)$ which are stable as $\Sp(2n,\CC)$-Higgs bundles.  Thus, by
Proposition~\ref{prop:smoothness}, they belong to the smooth locus of
the moduli space $\mathcal{M}_d$. In Theorem~\ref{thm:smooth-minima}
below we prove that the statement of Theorem~\ref{thm:all-minima}
holds in this case.

Our results are based on a careful analysis of the structure of
$\Sp(2n,\RR)$-Higgs bundles $(V,\varphi)$ satisfying the criterion of
Corollary~\ref{cor:ad-iso-minima}.

\subsection{Hodge bundles}
\label{sec:hodge-bundles}

In this subsection we give a description of simple
$\Sp(2n,\RR)$-Higgs bundles which are complex variations of Hodge
structure (cf.\ Proposition~\ref{prop:VHS}). Assume that the
$\Sp(2n,\RR)$-Higgs bundle $(V,\varphi)=(V,\beta,\gamma)$ is a
Hodge bundle, so that there is a splitting $V=\bigoplus_{i\in\RR}
F_i$ and
\begin{equation}
\label{eq:peces-beta-gamma} \beta\in
H^0(\bigoplus_{i+j=1}F_i\otimes F_j\otimes K),\qquad \gamma\in
H^0(\bigoplus_{-i-j=1}F^*_i\otimes F^*_j\otimes K),
\end{equation}
as described in Proposition~\ref{prop:VHS} (these tensor products
should be interpreted as subbundles of $S^2V\otimes K$ and
$S^2V^*K$, so for example when $i=j=\frac{1}{2}$ the summand
$F_i\otimes F_j\otimes K$ is to be thought of as the symmetric
product $S^2F_{\frac{1}{2}}\otimes K$). It is important to bear in
mind that the indices $i$ of the summands $F_i$ are in general
real numbers, not necessarily integers (in fact, we will deduce
below from the condition that $V$ is simple that $F_i$ is zero
unless $i$ belongs to $\frac{1}{2}+\ZZ$).

The following definitions will be useful in the subsequent
arguments. Let $\Gamma$ be the group of maps from $\RR$ to itself
generated by the functions $f,g:\RR\to\RR$ given by $f(x)=1-x$ and
$g(x)=-1-x$. Let $\OOO\subset\RR$ be an orbit of the action of
$\Gamma$. A parametrization of $\OOO$ is a surjective map
$r:\ZZ\to\OOO$ which satisfies $r(2k+1)=f(r(2k))$ and
$r(2k+2)=g(r(2k+1))$ for each integer $k$. Since the maps $f,g$
are involutions, any orbit of $\Gamma$ admits a parametrization.
We now have:

\begin{lemma}
\label{lemma:gamma-orbits} Let $\OOO\subset\RR$ be any orbit of
the action of $\Gamma$. Then $\OOO$ belongs to one of the
following sets of orbits:
\begin{enumerate}
\item $\ZZ$, \item $\frac{1}{2}+2\ZZ$, \item $-\frac{1}{2}+2\ZZ$,
\item $(\alpha+2\ZZ)\cup ((1-\alpha)+2\ZZ)$, where
$0<\alpha<\frac{1}{2}$ is a real number, \item $(-\alpha+2\ZZ)\cup
((\alpha-1)+2\ZZ)$, where $0<\alpha<\frac{1}{2}$ is a real number.
\end{enumerate}
Furthermore, any parametrization $r:\ZZ\to\OOO$ is bijective
unless $\OOO$ is either $\frac{1}{2}+2\ZZ$ or $-\frac{1}{2}+2\ZZ$.
\end{lemma}
\begin{proof}
If two real numbers $x,y\in\RR$ satisfy $x-y\in 2\ZZ$ then
$f(x)-f(y)\in 2\ZZ$ and $g(x)-g(y)\in 2\ZZ$, so the action of
$\Gamma$ on $\RR$ descends to any action on $\RR/2\ZZ$. Since
$f(g(x))=2+x$, for any $\Gamma$-orbit $\OOO\subset\RR$ and any
$x\in\OOO$ we have $x+2\ZZ\subset\Gamma$. It follows that the
quotient map $\RR\to\RR/2\pi\ZZ$ gives a bijection between
$\Gamma$-orbits. Consequently, to classify the orbits of $\Gamma$
acting on $\RR$ is equivalent to classify the orbits on
$\RR/2\ZZ$. Such classification can be easily made by hand, so the
first statement of the lemma follows. The second statement can
also be checked directly in a straightforward way.
\end{proof}

\begin{lemma}
\label{lemma:orbites-possibles} Assume that $(V,\beta,\gamma)$ is
simple. Then there exists a unique $\Gamma$-orbit
$\OOO\subset\RR$, which is either $\frac{1}{2}+2\ZZ$ or
$-\frac{1}{2}+2\ZZ$, such that
$$V=\bigoplus_{i\in\OOO}F_i.$$
In other words, $F_i=0$ unless $i\in\OOO$.
\end{lemma}
\begin{proof}
For any two reals $i,j\in\RR$ let $\beta_{ij}$ be the piece of
$\beta$ contained in $H^0(F_i\otimes F_j\otimes K)$, and define
similarly $\gamma_{ij}\in H^0(F^*_i\otimes F^*_j\otimes K)$. It
follows from (\ref{eq:peces-beta-gamma}) that both $\beta_{ij}$
and $\gamma_{ij}$ vanish unless $i,j$ belong to the same
$\Gamma$-orbit. We now prove that there is a unique $\Gamma$-orbit
$\OOO$ such that $F_i\neq 0\Rightarrow i\in\OOO$. Suppose that
this is not the case. Then there exists a $\Gamma$-orbit $\OOO$
such that both bundles
$$V'=\bigoplus_{i\in\OOO}F_i\qquad\qquad\text{and}\qquad\qquad
V''=\bigoplus_{i\notin\OOO}F_i$$ are nonzero. Clearly, $V=V'\oplus
V''$. Furthermore, by the previous observation, defining
$$\beta'=\bigoplus_{i,j\in\OOO}\beta_{ij},\qquad
\beta''=\bigoplus_{i,j\notin\OOO}\beta_{ij},\qquad
\gamma'=\bigoplus_{i,j\in\OOO}\gamma_{ij},\qquad
\gamma''=\bigoplus_{i,j\notin\OOO}\gamma_{ij},$$ we have
$\beta=\beta'+\beta''$ and $\gamma=\gamma'+\gamma''$. It follows
that the automorphism of $V$ defined as
$\sigma=\Id_{V'}-\Id_{V''}$ fixes both $\beta$ and $\gamma$, so
$(V,\beta,\gamma)$ is not simple, contradicting our hypothesis.
Now let $\OOO$ be the $\Gamma$-orbit satisfying
$V=\bigoplus_{i\in\OOO}F_i$, and let $r:\ZZ\to\OOO$ be a
parametrization. Assume that $\OOO$ is not of the form
$\frac{1}{2}+2\ZZ$ nor of the form $-\frac{1}{2}+2\ZZ$. Then, by
Lemma \ref{lemma:gamma-orbits}, the map $r$ is a bijection. Define
then
$$V'=\bigoplus_{k\in\ZZ}F_{r(2k)},\qquad\qquad\text{and}\qquad\qquad
V''=\bigoplus_{k\in\ZZ}F_{r(2k+1)}.$$ Then we have
$$\beta\in H^0(V'\otimes V''\otimes K),\qquad\qquad
\gamma\in H^0((V')^*\otimes (V'')^*\otimes K).$$ Hence, any
automorphism of $V$ of the form $\sigma=\theta \Id_{V'}+
\theta^{-1}\Id_{V''}$, with $\theta\in\CC^*$, fixes both $\beta$
and $\gamma$, contradicting the assumption that $(V,\beta,\gamma)$
is simple. It follows that $\OOO$ is equal either to
$\frac{1}{2}+2\ZZ$ or to $-\frac{1}{2}+2\ZZ$, so the lemma is
proved.
\end{proof}

\subsection{Simple minima with $\beta\neq 0$ and $\gamma\neq 0$}
\label{sec:apply-criterion}

Assume, as in the previous subsection, that $(V,\beta,\gamma)$ is
simple and a Hodge bundle. Assume additionally that $\beta\neq 0$ and
$\gamma\neq 0$.

Denote as before by $\OOO\subset\RR$ the $\Gamma$-orbit satisfying
$V=\bigoplus_{i\in\OOO}F_i$. We claim that there are at least two
nonzero summands in the previous decomposition. Indeed, if there
is a unique nonzero summand $F_i$, then $\beta\neq 0$ implies
$2i=1$, whereas $\gamma\neq 0$ implies $2i=-1$. Since these
assumptions are mutually contradictory, the claim follows.

Now define $M_+=\sup\{i\mid F_i\neq 0\}$ and $M_-=\inf\{i\mid
F_i\neq 0\}$. We claim that $|M_+|\neq |M_-|$. Indeed, by Lemma
\ref{lemma:orbites-possibles} we have either
$\OOO=\frac{1}{2}+2\ZZ$ or $\OOO=-\frac{1}{2}+2\ZZ$. Suppose we
are in the first case. Then we can write $M_+=\frac{1}{2}+2k$,
$M_-=\frac{1}{2}+2l$ for some integers $k,l$. The equality
$|M_+|=|M_-|$ implies that $M_+=M_-$, so we conclude that there is
a unique nonzero $F_i$, contradicting our previous observation.
The case $\OOO=-\frac{1}{2}+2\ZZ$ is completely analogous.

In view of the preceding observation, we may distinguish two cases:
either $|M_+|>|M_-|$ or $|M_+|<|M_-|$.  Henceforth we shall assume,
for definiteness, that we are in the situation $|M_+|>|M_-|$.

\begin{remark}
  \label{rem:dual-minima}
  Recall from Proposition~\ref{prop:duality-iso} that, for each $d$,
  there is an isomorphism $\mathcal{M}_d \xrightarrow{\simeq}
  \mathcal{M}_{-d}$, given by the duality $(V,\beta,\gamma) \mapsto
  (V^*, \gamma^t,\beta^t)$.  Under this duality the two cases
  $|M_+|>|M_-|$ and $|M_+|<|M_-|$ get interchanged (in fact, as we
  shall see, the former situation corresponds to $d<0$, whereas the
  latter corresponds to $d>0$).
\end{remark}

Let $M=M_+$. We have $M=p+\frac{1}{2}$ for some integer $p$.
Define $m=-p+\frac{1}{2}$. We can write
  \begin{equation}\label{eq:V-decomp}
    V=\bigoplus_{\lambda=0}^p F_{M-2\lambda}.
  \end{equation}
A priori, in this decomposition there might be some summands which
are zero. Nevertheless, we will sea below that this is not the
case.

\begin{theorem}\label{thm:minima}
  Let $(V,\beta,\gamma)$ be simple and a Hodge bundle with $\beta\neq
  0$ and $\gamma \neq 0$. Assume additionally that $|M_+|>|M_-|$ so
  that $(V,\beta,\gamma)$ is of the form (\ref{eq:V-decomp}). Then the
  map
  \begin{displaymath}
    \ad(\varphi)\colon \End(V)_k \to
    (S^2V\otimes K\oplus S^2V^*\otimes K)_{k+1}
  \end{displaymath}
  is an isomorphism for all $k>0$ if and only if the following holds:
  \begin{itemize}
  \item[(i)] For any $0\leq\lambda\leq p$
    the rank of $F_{M-2\lambda}$ is $1$ (in particular, it is nonzero);
  \item[(ii)] for any $0\leq\lambda\leq p-1$ the piece of $\beta$ in
    $$F_{M-2\lambda}\otimes F_{m+2\lambda}\otimes K\subset S^2V\otimes
    K$$
    never vanishes;
  \item[(iii)] for any $1\leq\lambda\leq p-1$ the piece of $\gamma$ in
    $$F^*_{M-2\lambda}\otimes F^*_{m+2\lambda-2}\otimes K\subset
    S^2V^*\otimes K$$
    never vanishes.
  \end{itemize}
  Analogous statements hold in the case $|M_+|<|M_-|$ (cf.\
  Remark~\ref{rem:dual-minima}).
%  An analogous statement holds for Hodge bundles of the type
%  described in \eqref{eq:cadena-gamma} of
%  Proposition~\ref{prop:cadena}.
\end{theorem}

\begin{proof}
  We already proved that the assumption $\beta\neq 0$ and $\gamma \neq 0$ implies that $p\geq 1$
  (for otherwise in the decomposition \eqref{eq:V-decomp} we would only have one summand).
  If we take the
  piece in degree $k=2p$ of the map $\ad(\varphi)$, we get
$$A:=\ad(\varphi)_{2p}:F_M\otimes F_m^*\to S^2F_M\otimes K,$$
which by assumption is an isomorphism. Computing the ranks $r_i =
\rk(F_i)$, we
deduce
$$r_Mr_m=\frac{r_M(r_M+1)}{2}.$$
To prove that $r_M=r_m=1$, we assume the contrary and show that
this leads to a contradiction. If $r_M>1$ then by the formula
above we must have $r_m<r_M$.
Let $b$ be the piece of $\beta$ in $F_M\otimes F_m\otimes K\subset
(S^2V\otimes K)_{2p}$. Then the map $A$ sends any
$e\in F_M\otimes F_m^*$ to
$$A(e)=eb+be^*.$$
The first summand denotes the composition of maps
$$F_M^*\stackrel{b}{\longrightarrow}
F_m\stackrel{e}{\longrightarrow}
F_M$$
and the second summand
$$F_M^*\stackrel{e^*}{\longrightarrow}
F_m^*\stackrel{b}{\longrightarrow}
F_M.$$
Take a basis $u_1,\dots,u_{r_M}$ of $F_M$ whose first
$r_m$ elements are a basis of $b(F_m^*)$, and take on
$F_M^*$ the dual basis. If we write the matrices of
$eb$ and $be^*$ with respect to these basis, one readily
checks that the $(r_M-r_m)\times(r_M-r_m)$ block in the
bottom left of both matrices vanishes. Consequently,
an element in $S^2F_M$ represented by a symmetric
matrix whose entry at the bottom left corner is nonzero
cannot belong to the image of $A$. Hence $A$ is not an isomorphism,
in contradiction to our assumption, so we deduce that
$$r_M=r_m=1.$$
One also deduces that the section $b\in H^0(F_M\otimes F_m\otimes K)$
never vanishes.
This proves statements (i) and (ii) when $\lambda=0$
or $p$.

\emph{Observation.} The following observation will be useful:
if $e\in F_i\otimes F_j^*\subset \End(V)$,
then any nonzero piece of $\ad(\varphi)(e)$ in the decomposition
(\ref{eq:decV}) belongs to a summand of the form
$F_i\otimes F_u\otimes K$, and any nonzero piece in
(\ref{eq:decVd}) belongs to a summand of the form
$F_j^*\otimes F_v^*\otimes K$ (in both cases the symmetrization
should be understood if the two indices coincide).
This follows from the fact that $\ad(\varphi)(e)$ is the sum
of compositions of $e$ with another map
(either on the right and on the left).
Hence each summand in $\ad(\varphi)(e)$ must share with
$e$ at least the domain or the target.

Now let us take any $k=2p-2\lambda\geq 1$, such that $\lambda\geq 1$,
so that $1\leq\lambda\leq p-1$.  Then we have
\begin{equation}
\End(V)_{2p-2\lambda}=
F_{M}\otimes F_{m+2\lambda}^*\oplus
F_{M-2}\otimes F_{m+2\lambda-2}^*\oplus
\dots\oplus
F_{M-2\lambda}\otimes F_{m}^*.
\label{eq:endo}
\end{equation}
We claim that there is no nonzero block in
$(S^2V^*\otimes K)_{2p-2\lambda+1}$ of the form
$F_{m+2\lambda}^*\otimes F_v^*\otimes K$.
Indeed, for that one should take
$v=-(2p-2\lambda+1)-(m+2\lambda)=-M-1$, but
$F_{-M-1}=0$, because $-M-1<m$. On the other hand,
$(S^2V^*\otimes K)_{2p-2\lambda+1}$ contains the block
$F_M\otimes F_{M-2\lambda}\otimes K$ and no other block
involving $F_M$. Hence we must have
$$\ad(\varphi)_k(F_{M}\otimes F_{m+2\lambda}^*)\subset
F_M\otimes F_{M-2\lambda}\otimes K.$$
Taking ranks and using the fact that $\ad(\varphi)_k$ is injective,
we deduce that
$$r_{m+2\lambda}\leq r_{M-2\lambda}.$$
Since $1\leq\lambda\leq p-1\Longleftrightarrow
1\leq p-\lambda\leq p-1$, we automatically deduce that
$$r_{m+2p-2\lambda}\leq r_{M-2p+2\lambda}.$$
But $m+2p=M$, so we conclude that
\begin{equation}
r_{m+2\lambda}=r_{M-2\lambda}.
\label{eq:simetria}
\end{equation}

Let us distinguish two possibilities.

\emph{Case (1).} Suppose that $\lambda=2l+1$ is odd.
Then we have
$$S^2F_{m+\lambda-1}^*\otimes K
\subset (S^2V^*\otimes K)_{2p-2\lambda+1},$$
and the observation above implies that
$$\ad(\varphi)_{2p-2\lambda}^{-1}(S^2F_{m+\lambda-1}^*\otimes K)
\subset F_{M-\lambda-1}\otimes F_{m+\lambda-1}^*.$$
The argument given above for $\lambda=0$ proves now that
the piece of $\gamma$ in
$$F_{M-\lambda-1}^*\otimes F_{m+\lambda-1}^*\otimes K$$
never vanishes.

\emph{Case (2).} Suppose that $\lambda=2l$ is even.
Then we have
$$S^2F_{M-\lambda}\otimes K
\subset (S^2V\otimes K)_{2p-2\lambda+1},$$
and the observation above implies that
$$\ad(\varphi)_{2p-2\lambda}^{-1}(S^2F_{M-\lambda}\otimes K)
\subset F_{M-\lambda}\otimes F_{m+\lambda}^*.$$
The argument given above for $\lambda=0$ proves now that
the piece of $\beta$ in
$$F_{M-\lambda}\otimes F_{m+\lambda}\otimes K$$
never vanishes.

These arguments prove statements (ii) and (iii).

We are now going to prove that for any
$1\leq\lambda\leq p/2$ the ranks $r_{M-2\lambda}=r_{m+2\lambda}=1$
using induction. Fix such a $\lambda$ and assume that for any
$0\leq l<\lambda$ we have $r_{M-2l}=r_{m+2l}=1$ (when $l=0$
we already know this is true).
Since $2p-2\lambda\geq 1$ we must have
\begin{equation}
\rk\End(V)_{2p-2\lambda}=
\rk (S^2V\otimes K\oplus S^2V^*\otimes K)_{2p-2\lambda+1}.
\label{eq:mateixrank}
\end{equation}
Using induction we can compute the left hand side:
\begin{align*}
\rk\End(V)_{2p-2\lambda}&=r_Mr_{m+2\lambda}+r_{M-2}r_{m+2\lambda-2}
+\dots+r_{M-2\lambda+2}r_{m+2}+r_{M-2\lambda}r_{m}\\
&=r_{m+2\lambda}+r_{M-2\lambda}+(\lambda-1).
\end{align*}
We now distinguish again two cases.

\emph{Case (1).}
Suppose that $\lambda=2l+1$ is odd. Then we compute
\begin{align*}
\rk(S^2V)_{2p-2\lambda+1}
&=r_Mr_{M-2\lambda}+r_{M-2}r_{M-2\lambda+2}+\dots+
r_{M-\lambda+1}r_{M-\lambda-1}\\
&=r_{M-2\lambda}+l
\end{align*}
and
\begin{align*}
\rk(S^2V^*)_{2p-2\lambda+1}
=&r_mr_{m+2\lambda-2}+r_{m+2}r_{m+2\lambda-4}+\dots
+r_{m+\lambda-3}r_{m+\lambda+1}\\
&+\binom{r_{m+\lambda-1}+1}{2}=l+1.
\end{align*}
Comparing the two computations it follows from (\ref{eq:mateixrank})
that
$$r_{m+2\lambda}=1,$$
and using (\ref{eq:simetria}) we deduce that
$$r_{M-2\lambda}=1.$$

\emph{Case (2).}
Now suppose that $\lambda=2l$ is even. Then we have
\begin{align*}
\rk(S^2V)_{2p-2\lambda+1}
=&r_Mr_{M-2\lambda}+r_{M-2}r_{M-2\lambda+2}+\dots+
r_{M-\lambda+2}r_{M-\lambda-2}\\
&+\binom{r_{M-\lambda}}{2}=r_{M-2\lambda}+l
\end{align*}
and
\begin{align*}
\rk(S^2V^*)_{2p-2\lambda+1}
&=r_mr_{m+2\lambda-2}+r_{m+2}r_{m+2\lambda-4}+\dots
+r_{m+\lambda-2}r_{m+\lambda}\\
&=l.
\end{align*}
Comparing again the two computations we deduce that
$$r_{m+2\lambda}=r_{M-2\lambda}=1.$$
This finishes the proof of statement (i) and thus the proof of the
Theorem in the case $|M_+| > |M_-|$.

Finally, in the case $|M_+| < |M_-|$ the analysis is completely analogous.
\end{proof}

\begin{corollary}\label{cor:minima}
  Let $(V,\beta,\gamma)$ be simple and a Hodge bundle with $\beta\neq
  0$ and $\gamma \neq 0$. Assume additionally that $|M_+|>|M_-|$ so
  that $(V,\beta,\gamma)$ is of the form (\ref{eq:V-decomp}).
  Assume that the map
  \begin{displaymath}
    \ad(\varphi)\colon \End(V)_k \to
    (S^2V\otimes K\oplus S^2V^*\otimes K)_{k+1}
  \end{displaymath}
  is an isomorphism for all $k>0$.  Then the following holds.
  \begin{enumerate}
  \item\label{cor:minima-n=2} If $n=2$ then $F_{\frac{3}{2}}\otimes F_{-\frac{1}{2}}\otimes
    K\simeq \OOO$.
  \item\label{cor:minima-odd}  If $n=2q+1 \geq 3$ is odd then $\beta\colon
  F^*_{\frac{1}{2}-2\lambda}\overset{\simeq}{\to}
  F_{\frac{1}{2}+2\lambda}K$ for any integer $-q\leq \lambda\leq q$.
  In particular, there exists a square root $L$ of $K$ such that for
  any integer $-q\leq \lambda\leq q$ we have
  $$
    F_{M-2(q-\lambda)}\simeq F_{m+2(\lambda+q)} \simeq
    F_{\frac{1}{2}+2\lambda} \simeq L^{-1}\otimes K^{-2\lambda},
  $$
  and the bundle $V$ is of the form
  \begin{displaymath}
    V = \bigoplus_{\lambda=-q}^{q}L^{-1}K^{-2\lambda}.
  \end{displaymath}
  \item\label{cor:minima-even}   If $n=2q +2\geq 4$ then $\gamma\colon
  F_{-\frac{1}{2}}\overset{\simeq}{\to} F^*_{-\frac{1}{2}}K$ and
  $\beta\colon F^*_{-\frac{1}{2}-2\lambda}\overset{\simeq}{\to}
  F_{-\frac{1}{2}+2\lambda}K$ for any integer $ -q\leq \lambda \leq
  q+1$.  In particular, there exists a square root $L$ of $K$ such
  that for any integer $ -q\leq \lambda \leq q+1$ we have
  $$
    F_{-\frac{1}{2}+2 \lambda}\simeq L\otimes K^{-2\lambda}
    \simeq F_{M-2(q+1-\lambda)}\simeq F_{m+2(\lambda+q)},
  $$
  and the bundle $V$ is of the form
  \begin{displaymath}
    V = \bigoplus_{\lambda=-q}^{q+1}LK^{-2\lambda}.
  \end{displaymath}
  \item\label{item:4}  For any $n\geq 2$, the degree of $V$ is $\deg(V)=n(1-g)$.
  \item\label{item:5} For any $n\geq 2$, an $\Sp(2n,\RR)$-Higgs bundle
    of the form described in
    (\ref{cor:minima-n=2})--(\ref{cor:minima-even}) above is stable as
    an $\SL(2n,\CC)$-Higgs bundle, and thus also as an
    $\Sp(2n,\CC)$-Higgs bundle.
  \end{enumerate}
  Analogous statements hold in the case $|M_+|<|M_-|$. In particular,
  in this case the degree of $V$ is $\deg(V)=n(g-1)$ (cf.\
  Remark~\ref{rem:dual-minima}).
\end{corollary}

\begin{remark}\label{rem:n-one-hodge}
  In the case $n=1$ it is not possible for $(V,\varphi)$ to be a Hodge bundle with $\beta\neq 0$ and $\gamma\neq 0$.
\end{remark}

\begin{proof}[Proof of Corollary~\ref{cor:minima}]
  First we observe that, since the $F_i$ are all line bundles, we have
  $n=p+1$, $M=p +\frac{1}{2}$ and $m=-p +\frac{1}{2}$.

  (\ref{cor:minima-n=2}) In this case we have $n=2$, $p=1$, $M=3/2$, $m=-1/2$. Then,
  taking $\lambda=0$ in (ii) of Theorem \ref{thm:minima} we get
  $F_{\frac{3}{2}}\otimes F_{-\frac{1}{2}}\otimes K\simeq \OOO$.

  (\ref{cor:minima-odd}) In this case we have $n=p+1=2q+1$ so that $M=2q+1/2$ and
  $m=-2q+1/2$.  Hence, using (ii) and (iii) of Theorem
  \ref{thm:minima}, we can describe the structure of the maps $\beta$
  and $\gamma$ in the following diagram:
  
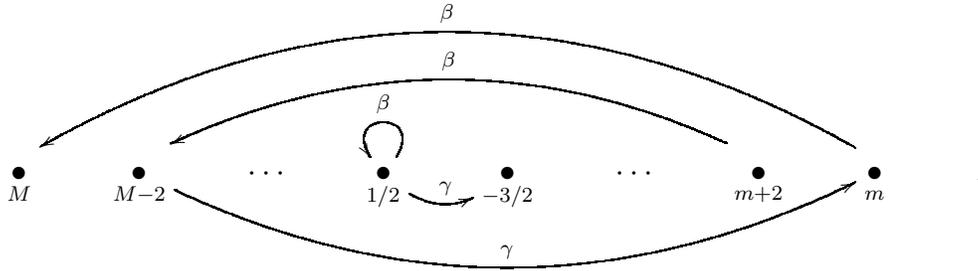
\begin{displaymath}
  \xymatrix{
    \underset{\scriptstyle M}{\bullet} &
    \underset{\scriptstyle M-2}{\bullet} \ar@/_3pc/[rrrrrr]^{\gamma}&
    \cdots &
    \underset{\scriptstyle 1/2}{\bullet}\ar@(ur,ul)[]_{\beta}
      \ar@/_1pc/[r]^{\gamma}&
    \underset{\scriptstyle -3/2}{\bullet} &
    \cdots &
    \underset{\scriptstyle m+2}{\bullet} \ar@<-1ex>@/_2.5pc/[lllll]_{\beta} &
    \underset{\scriptstyle m}{\bullet} \ar@<-1ex>@/_4pc/[lllllll]_{\beta} &
  },
  \end{displaymath}
  where an arrow
  $\xymatrix{\underset{\smash{\scriptstyle i}}{\bullet}\ar@<1ex>[r]^{\beta}
    & \underset{\smash{\scriptstyle j}}{\bullet}} $
  means that there is an isomorphism $\beta\colon F^*_{i}\to
  F_j\otimes K$ (and thus $j=-i+1$); similarly, an arrow
  $\xymatrix{\underset{\scriptstyle i}{\bullet}\ar@<1ex>[r]^{\gamma}
    & \underset{\scriptstyle j}{\bullet}} $
  means that there is an isomorphism $\gamma\colon F_{i}\to
  F^*_j\otimes K$.  In particular, we see that the isomorphism
  $\beta\colon F^*_{\frac{1}{2}} \overset{\simeq}{\to}
  F_{\frac{1}{2}}\otimes K$ means that $F_{\frac{1}{2}} \simeq L^{-1}$
  for a square root $L$ of $K$. This proves the case $\lambda=0$ of
  (2). Now repeated application of (ii) and (iii) of Theorem
  \ref{thm:minima} proves the general case.  Note that this argument
  can be phrased as saying that the graph above is
  connected and its only closed loop is the one at $1/2$:  thus the
  remaining $F_i$ are uniquely determined by $F_{\frac{1}{2}}$.

  (\ref{cor:minima-even}) In this case we have $n=p+1=2q+2$ so that $M=2q+3/2$ and
  $m=-2q-1/2$ and, as above, we have a diagram
  
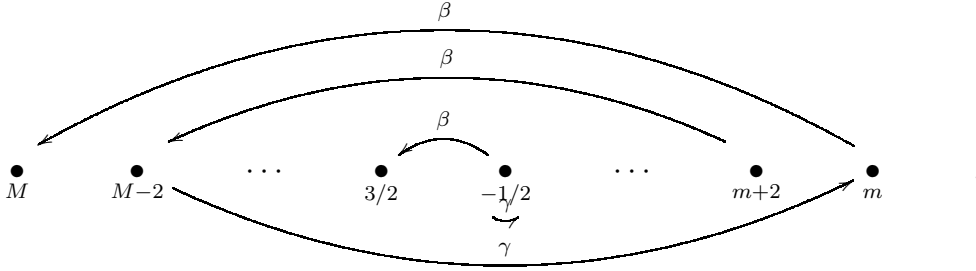
\begin{displaymath}
  \xymatrix{
    \underset{\scriptstyle M}{\bullet} &
    \underset{\scriptstyle M-2}{\bullet} \ar@/_3pc/[rrrrrr]^{\gamma}&
    \cdots &
    \underset{\scriptstyle 3/2}{\bullet}&
      \underset{\scriptstyle -1/2}{\bullet}\ar@(dl,dr)[]^{\gamma}
      \ar@/_1pc/[l]_{\beta}&
    \cdots &
    \underset{\scriptstyle m+2}{\bullet} \ar@<-1ex>@/_2.5pc/[lllll]_{\beta} &
    \underset{\scriptstyle m}{\bullet} \ar@<-1ex>@/_4pc/[lllllll]_{\beta} &
  }.
  \end{displaymath}
  The argument is now analogous to the previous case.

  (\ref{item:4}) Easy from the formulas for $V$ given in
  (\ref{cor:minima-odd}) and (\ref{cor:minima-even}).

  (\ref{item:5}) Let $(V,\varphi)$ be of the kind described in
  (\ref{cor:minima-n=2})--(\ref{cor:minima-even}), and consider the
  associated $\SL(2n,\CC)$-Higgs bundle $(V \oplus V^*,\Phi) =
  H(V,\varphi)$. The $\Phi$-invariant subbundles of $V \oplus V^*$ are
  of the form $\bigoplus_{i \geq i_0} (F_i \oplus F_{-i}^*)$. From the
  given description, it is easy to check that such a subbundle, when
  proper and non-zero, has degree strictly negative.

  Finally, in the case $|M_+| < |M_-|$ the analysis is completely analogous.
\end{proof}

\subsection{Simple minima: final characterization}
Finally, we use the analysis carried out so far to determine the
minima of the Hitchin functional on the locus of the moduli
space corresponding to simple $\Sp(2n,\RR)$-Higgs bundles which are
stable as $\Sp(2n,\CC)$-Higgs bundles.

\begin{theorem}\label{thm:smooth-minima}
  Let $(V,\beta,\gamma)$ be a simple
  $\Sp(2n,\RR)$-Higgs bundle which is stable as an $\Sp(2n,\CC)$-Higgs
  bundle.
  \begin{enumerate}
  \item If $\abs{d} < n(g-1)$ then $(V,\beta,\gamma)$ represents a
  minimum of the Hitchin functional if and only if it belongs to
  $\mathcal{N}_d$.

\item If $\abs{d} = n(g-1)$ and $n \geq 3$ then $(V,\beta,\gamma)$
  represents a minimum of the Hitchin functional if and only if one of
  the following situations occurs:
      \begin{enumerate}
      \item[(i)] the $\Sp(2n,\RR)$-Higgs bundle $(V,\beta,\gamma)$ belongs to
        $\mathcal{N}_d$;
      \item[(ii)]the $\Sp(2n,\RR)$-Higgs bundle
        $(V,\beta,\gamma)$ is of the type described in
        (\ref{cor:minima-odd}) or (\ref{cor:minima-even}) of
        Corollary~\ref{cor:minima}.
      \item[(iii)]the dual $\Sp(2n,\RR)$-Higgs bundle
        $(V^*,\gamma^t,\beta^t)$ is of the type described in
        (\ref{cor:minima-odd}) or (\ref{cor:minima-even}) of
        Corollary~\ref{cor:minima} (cf.\
        Remark~\ref{rem:dual-minima}).
      \end{enumerate}
      In cases (ii) and (iii) we say that
      $(V,\beta,\gamma)$ is a \emph{quiver type} minimum.
  \end{enumerate}
\end{theorem}

\begin{proof}
  If $(V,\beta,\gamma)$ belongs to $\mathcal{N}_d$ then we know from
  Proposition~\ref{prop:N(d)-minima} that it represents a minimum.
  And, if $(V,\beta,\gamma)$ (or the dual $(V^*,\gamma^t,\beta^t)$) is
  of the type described in (\ref{cor:minima-odd}) or
  (\ref{cor:minima-even}) of Corollary~\ref{cor:minima}, then
  Corollary~\ref{cor:ad-iso-minima} and Theorem~\ref{thm:minima} show
  that it represents a minimum.

  On the other hand, if $(V,\beta,\gamma)$ is a minimum which does not
  belong to $\mathcal{N}_d$, then Corollary~\ref{cor:ad-iso-minima},
  Theorem~\ref{thm:minima} and Corollary~\ref{cor:minima} show that it
  (or the dual $(V^*,\gamma^t,\beta^t)$) is of the type described in
  (\ref{cor:minima-odd}) or (\ref{cor:minima-even}) of
  Corollary~\ref{cor:minima}.
\end{proof}

%%%%%%%%%%%%%%%%%%%%%%%%%%%%%%%
\section{Minima on the entire moduli space}
\label{sec:poly-stable-minima}
%%%%%%%%%%%%%%%%%%%%%%%%%%%%%%%

%%%%%%%%%%%%%%%%%%%%%%%%%%%%%%%
\subsection{Main result and strategy of proof.}
%%%%%%%%%%%%%%%%%%%%%%%%%%%%%%%

In Section \ref{sec:stable-simple-minima} we characterized the minima
of the Hitchin functional on the locus of $\mathcal{M}_d$
corresponding to simple $\Sp(2n,\RR)$-Higgs bundles $(V,\varphi)$
which are stable as $\Sp(2n,\CC)$-Higgs bundles. In this section we
provide the remaining results required to extend this characterization
to the whole moduli space, thus completing the proof of
Theorem~\ref{thm:all-minima}. As explained in the proof of that
Theorem, what is required is to rule out certain type of potential
minima of the Hitchin functional. In each case this is done by using
(\ref{item:non-minima}) of Proposition~\ref{prop:minima-HH1}.  The
main result of this Section is the following.
\begin{proposition}
  \label{prop:strict-polystable-minima}
  Let $(V,\varphi=\beta+\gamma)$ be a polystable $\Sp(2n,\RR)$-Higgs bundle and
  assume that the decomposition $(V,\varphi) = (V_1,\varphi_1) \oplus
  \dots \oplus (V_k,\varphi_k)$ of
  Theorem~\ref{thm:polystable-spnr-higgs} is non-trivial. If
  $(V,\varphi)$ is a local minimum of the Hitchin functional then
  either $\beta=0$ or $\gamma=0$.
\end{proposition}

\begin{proof}
  The starting point is the structure
  Theorem~\ref{thm:polystable-spnr-higgs}. Recall that this describes
  a polystable $\Sp(2n,\RR)$-Higgs bundle as a direct sum
  \begin{equation}\label{eq:decomposition-stable-simple}
    (V,\varphi)
    = \bigoplus (V_i,\varphi_i),
  \end{equation}
  where each $\Sp(2n,\RR)$-Higgs bundle
  $(V_i,\varphi_i)$ comes from a $G_i$-Higgs bundle which is a smooth
  point in its respective moduli space. If $(V,\varphi)$ is a minimum,
  then Proposition~\ref{prop:hitchin-additive} implies that each
  $(V_i,\varphi_i)$ is a minimum on the corresponding moduli space of
  $G_i$-Higgs bundles. Consider each of the possible $G_i$'s in turn.

  \emph{The case $G_i = \Sp(2n_i,\RR)$.} This is the case
  covered by Theorem~\ref{thm:smooth-minima}. (Except for the case
  $n_i=2$, which will require special attention.)

  \emph{The case $G_i = \U(n_i)$.} In this case
  $\varphi_i=0$ for any $G_i$-Higgs bundle, as we have already seen.

  \emph{The case $G_i = \U(p_i,q_i)$.} In this case, the
  minima of the Hitchin functional were determined in
  \cite[Theorem~4.6]{bradlow-garcia-prada-gothen:2003}. There it is shown that a
  $\U(p_i,q_i)$-Higgs bundle $(\tilde V_i,\tilde W_i,
  \tilde\beta+\tilde\gamma)$ is a minimum if and only if
  $\tilde\beta=0$ or $\tilde\gamma=0$. Hence $(V_i,\varphi_i) =
  \upsilon_*^{\U(p_i,q_i)}(\tilde V_i,\tilde W_i,
  \tilde\beta+\tilde\gamma)$ (cf.\ (\ref{eq:upq-sp2n})) is a minimum
  if and only if $\beta_i=0$ or $\gamma_i=0$

  \emph{The case $G_i = \GL(n_i,\RR)$.} The moduli space of such Higgs
  bundles was studied in
  \cite{bradlow-garcia-prada-gothen:2004b}. Using the results of that
  paper we show in Lemma~\ref{lem:glnr-spnr-minima} below that a
  $\Sp(2n_i,\RR)$-Higgs bundle $(V_i,\varphi_i)$ coming from a
  $\GL(n_i,\RR)$-Higgs bundle is a minimum if and only if
  $\varphi_i=0$.

  A quiver type minimum $(V,\varphi)$ is simple and stable as a
  $\Sp(2n,\CC)$-Higgs bundle by (\ref{item:5}) of
  Corollary~\ref{cor:minima}.  Thus, to conclude the proof of the
  Proposition, it remains to show that if $(V,\varphi)$ is a minimum and
  the decomposition (\ref{eq:decomposition-stable-simple}) is
  non-trivial, then it belongs to $\mathcal{N}_d$, i.e., $\beta=0$ or
  $\gamma=0$. By the above analysis of the minima coming from
  $G_i$-Higgs bundles, it therefore suffices to show that
  $(V,\varphi)$ is not a minimum when the decomposition
  (\ref{eq:decomposition-stable-simple}) falls in one of the following
  cases:
  \begin{enumerate}
  \item There is a $(V_i,\varphi_i)$ in $\mathcal{N}_{d_i}$ with
    $\beta_i\neq 0$ and a $(V_j,\varphi_j)$ in $\mathcal{N}_{d_j}$
    with $\gamma_j \neq 0$.
  \item There is a $(V_i,\varphi_i)$ which is a quiver type minimum
    and a $(V_j,\varphi_j)$ which lies in $\mathcal{N}_{d_i}$.
  \item There are (distinct) $(V_i,\varphi_i)$ and $(V_j,\varphi_j)$
    which are quiver type minima.
  \end{enumerate}
  In order to accommodate the possibility $n_i=2$, the quiver type
  minima must here be understood to include all minima with $\beta\neq
  0$ and $\gamma \neq 0$ (cf.\ (\ref{cor:minima-n=2}) of
  Corollary~\ref{cor:minima}). The case $n_i=1$ is included since such
  minima must have $\beta=0$ or $\gamma=0$ (cf.\
  Remark~\ref{rem:n-one-hodge}).

  Note that, by Proposition~\ref{prop:hitchin-additive}, in fact it
  suffices to consider the case when $k=2$ in
  (\ref{eq:decomposition-stable-simple}). With this in mind, the
  results of Lemmas~\ref{lem:beta-gamma-minima},
  \ref{lem:beta-quiver-minima} and \ref{lem:quiver-2-minima} below
  conclude the proof.
\end{proof}

%%%%%%%%%%%%%%%%%%%%%%%%%%%%%%%%%%%%%%%%%%%%%%%%%%%%%%%
\subsection{Deforming a sum of minima in $\mathcal{N}_d$.}
%%%%%%%%%%%%%%%%%%%%%%%%%%%%%%%%%%%%%%%%%%%%%%%%%%%%%%%

\begin{lemma}
  \label{lem:beta-gamma-minima}
  Let $(V,\varphi)$ be a polystable $\Sp(2n,\RR)$-Higgs bundle which
  decomposes as a direct sum $(V,\varphi) = (V',\varphi') \oplus
  (V'',\varphi'')$ with $\varphi'=(\beta',\gamma')$ and
  $\varphi''=(\beta'',\gamma'')$. Suppose that $\beta'=0$,
  $\gamma'\neq 0$, $\beta''\neq 0$ and $\gamma''=0$. Suppose
  additionally that $(V',\varphi')$ and $(V'',\varphi'')$ are stable
  $\Sp(2n,\RR)$-Higgs bundles or stable $\U(p,q)$-Higgs bundles. Then
  $(V,\varphi)$ is not a minimum of $f$ on $\mathcal{M}_d$.
%  The same is true if $\beta'\neq0$,
%  $\gamma'= 0$, $\beta''= 0$ and $\gamma''\neq 0$.
\end{lemma}

\begin{proof}
  We prove the Lemma by applying the criterion in
  (\ref{item:non-minima}) of Proposition~\ref{prop:minima-HH1}.  As a
  first step, we identify the complex $C^{\bullet}_-$ defined in
  (\ref{eq:def-C-minus}), and for that we need to know the weights of
  each piece $V',V''$. Recall that the weight of $\varphi',\varphi''$
  is always $1$.
\begin{enumerate}
\item Since $\gamma'\colon V' \to V'^*K$, the weight on $V'^*$ is
  $1+\lambda'=-\lambda'$, where $\lambda'$ is the weight on $V'$. Thus
  $\lambda'=-1/2$.
\item Similarly, the weight on $V''$ is $\lambda''=1/2$.
\end{enumerate}
{}From this it follows immediately that the complex $C^\bullet_-$ is
given by
\begin{displaymath}
  C^\bullet_- \colon \Hom(V',V'') \to 0,
\end{displaymath}
so that
\begin{displaymath}
  \HH^1(C^\bullet_-) = H^1(\Hom(V',V'')).
\end{displaymath}
Recall from Remark~\ref{rem:N(d)-sign} that $d'=\deg(V') \geq 0$ and
$d'' \leq 0$ so, by Riemann--Roch,
\begin{displaymath}
  H^1(\Hom(V',V'')) \neq 0.
\end{displaymath}
This proves that $C^{\bullet}_-$ has nonzero first hypercohomology. To
finish the argument we need to integrate any element of
$\HH^1(C^{\bullet}_-)$ to a deformation of $(V,\varphi)$ through
polystable $\Sp(2n,\RR)$-Higgs bundles.

Chose any\footnote{when one of $(V',\varphi')$ and $(V'',\varphi'')$
  is a $\U(p,q)$-Higgs bundle, this choice is not completely
  arbitrary, cf.\ the proof of
  Lemma~\ref{lem:beta-gamma-family-polystable} below.} nonzero element
$a\in H^1(\Hom(V',V''))$. Denote by $D$ the open unit disk. Define
$\VV'=D\times V'$ and $\VV''=D\times V''$, which we view as vector
bundles over $X\times D$. We denote by $\gamma'_D:\VV'\to\VV'^*\otimes
K$ (here $K$ denotes the pullback to $X\times D$) the extension of
$\gamma'$ which is constant on the $D$ direction, and we define
similarly $\beta''_D:\VV''^*\to\VV''\otimes K$. Take the extension
$$
0 \to \VV'' \to \VV \to \VV' \to 0
$$
classified by
$$
a \otimes 1 \in H^1(\Hom(\VV',\VV'')) =  H^1(X;\Hom(V',V'')) \otimes
H^0(D;\CC).
$$
The restriction of this to $X\times\{t\}$ is the extension
\begin{equation}\label{eq:ext-primed}
0 \to V'' \to V_t \to V' \to 0
\end{equation}
classified by $ta \in H^1(\Hom(V',V''))$.
Define $\gamma_D:\VV\to \VV^*\otimes K$ as the composition
$$\VV\longrightarrow \VV'\stackrel{\gamma'_D}{\longrightarrow}
\VV'^*\otimes K\to \VV^*\otimes K,$$ where the first arrow comes
from the exact sequence defining $\VV$ and the third one comes
from dualizing the same exact sequence and tensoring by the
pullback of $K$. Similarly, define $\beta_D:\VV^*\to\VV\otimes K$
as the composition
$$\VV^*\longrightarrow \VV''^*\stackrel{\beta''_D}{\longrightarrow}
\VV''\otimes K\to \VV\otimes K.$$ The resulting triple
$(\VV,\beta_D,\gamma_D)$ is a family of symplectic Higgs bundles
parametrized by the disk, whose restriction to the origin
coincides with $(V,\varphi)$, and which integrates the element $a$
in the deformation complex.

It remains to show that each member of the family
$(\VV,\beta_D,\gamma_D)$ is a polystable $\Sp(2n,\RR)$-Higgs
bundle. This is done in Lemma~\ref{lem:beta-gamma-family-polystable}
below. We have thus proved that $(V,\varphi)$ is not a local minimum.
\end{proof}

\begin{lemma}
  \label{lem:beta-gamma-family-polystable}
  The $\Sp(2n,\RR)$-Higgs bundle $(V_t,\varphi_t= \beta_t + \gamma_t)$
  on $X$, obtained by restricting to $X \times \{t\}$ the family
  $(\VV,\beta_D,\gamma_D)$ constructed in the proof of
  Lemma~\ref{lem:beta-gamma-minima}, is polystable.
\end{lemma}

\begin{proof}
  It will be convenient to use the stability condition for
  $\Sp(2n,\RR)$-Higgs bundles as given in Lemma~\ref{primerpas}. Thus,
  if $(V_t,\varphi_t)$ is not stable, there are subbundles $A \subset
  V_t$ and $B \subset V_t^*$ such that $\gamma_t(A) \subset B \otimes
  K$ and $\beta_t(B) \subset A \otimes K$, and with $\deg(A \oplus B)
  = 0$. Since $X$ is a Riemann surface, the kernel of the restriction
  to $A$ of the sheaf map $V_t \to V''$ is locally free and
  corresponds to a subbundle $A' \subset A$. The quotient $A'' := A /
  A'$ then gives a subbundle $A'' \subset V''$ so that we have a
  commutative diagram with exact rows and columns:
  \begin{equation}
    \label{eq:cd-A}
    \begin{CD}
      & & 0 & & 0 & & 0  \\
      @. @VVV  @VVV  @VVV  @.\\
      0 @>>> A'' @>>> A @>>> A' @>>> 0 \\
      @. @VVV  @VVV  @VVV  @.\\
      0 @>>> V'' @>>> V_t @>>> V' @>>> 0.
    \end{CD}
  \end{equation}
  Similarly, we obtain subbundles $B'' \subset {V''}^{*}$ and $B'
  \subset {V'}^{*}$ and a diagram:
  \begin{equation}
    \label{eq:cd-B}
    \begin{CD}
      & & 0 & & 0 & & 0  \\
      @. @VVV  @VVV  @VVV  @.\\
      0 @<<< B' @<<< B @<<< B'' @<<< 0 \\
      @. @VVV  @VVV  @VVV  @.\\
      0 @<<< {V'}^{*} @<<< V_t @<<< {V''}^{*} @<<< 0.
    \end{CD}
  \end{equation}
  One easily checks that ${B'}^{,\perp} \subset A'$ and
  ${B''}^{,\perp} \subset A''$. By definition of $\gamma_t$, the
  diagram
  \begin{displaymath}
    \begin{CD}
      0 @>>> V' @>>> V_t @>>> V'' @>>> 0 \\
      @. @VV{\gamma'}V  @VV{\gamma_t}V  @.  @.\\
      0 @<<< {V'}^{*} @<<< V_t @<<< {V''}^{*} @<<< 0.
    \end{CD}
  \end{displaymath}
  commutes. Thus, since $\gamma_t(A) \subset B \otimes K$, we have
  that $\gamma'(A') \subset B' \otimes K$. Similarly, $\beta''(B'')
  \subset A''\otimes K$. It follows that the pair of subbundles $A'
  \subset V'$ and $B' \subset {V'}^*$ destabilizes $(V',\varphi')$ and
  that the pair of subbundles $A'' \subset V''$ and $B'' \subset {V''}^*$
  destabilizes $(V'',\varphi'')$.

  Consider now the case in which both $(V',\varphi')$ and
  $(V'',\varphi'')$ are stable $\Sp(2n,\RR)$-Higgs bundles. Then we
  must have $A' \oplus B' = V' \oplus V'^*$ or $A' \oplus B' = 0$ and
  similarly for $A'' \oplus B''$. The only case in which the original
  destabilizing subbundle $A \oplus B \subset V_t \oplus V_t^*$ is
  non-trivial is when $A' \oplus B' = V' \oplus V'^*$ and $A'' \oplus
  B'' = 0$ (or vice-versa). But, in this case, $V' \simeq A'\simeq A$
  and hence (\ref{eq:cd-A}) shows that the non-trivial extension
  (\ref{eq:ext-primed}) splits, which is a contradiction. Hence there
  is no non-trivial destabilizing pair of subbundles of
  $(V_t,\varphi_t)$, which is therefore stable.

  It remains to deal with case in which one, or both, of
  $(V',\varphi')$ and $(V'',\varphi'')$ are stable $\U(p,q)$-Higgs
  bundles. The remaining cases being similar, for definiteness we
  consider the case in which $(V'',\varphi'')$ is a stable
  $\Sp(2n'',\RR)$-Higgs bundle and
  $(V',\varphi')$ is a stable $\U(n_1',n_2')$-Higgs bundle, i.e.,
  \begin{displaymath}
    V' = V_1' \oplus V_2', \qquad \varphi' = \gamma' \in
    H^0(V_1'\otimes V_2' \otimes K).
  \end{displaymath}
  In addition to the cases considered above, we now also need to
  consider the case when $A' \oplus B'$ is non-trivial, say $A' \oplus B' = V_1' \oplus {V_2'}^*$. There are now two possibilities for $A'' \oplus B''$:
  either it is zero or it equals $V'' \oplus {V''}^*$; we leave the
  first (simpler) case to the reader and consider the second one. In
  this case, the element
  \begin{displaymath}
    a = a_1 + a_2 \in H^1(\Hom(V',V'')
      = H^1(\Hom(V_1',V'')) \oplus H^1(\Hom(V_2',V''))
  \end{displaymath}
  chosen in the proof of Lemma~\ref{lem:beta-gamma-minima} above must
  be taken such that both $a_1$ and $a_2$ are non-zero (this is
  possible by Riemann--Roch). Thus, for $i=1,2$ we have a commutative diagram
  \begin{displaymath}
    \begin{CD}
      0 @>>> V'' @>>> V_{t_i} @>>> V_i'  @>>> 0 \\
      @. @| @VVV @VVV \\
      0 @>>> V'' @>>> V_{t} @>>> V_1' \oplus V_2'  @>>> 0
    \end{CD}
  \end{displaymath}
  of non-trivial extensions, where the two vertical maps on the right
  are inclusions. This, together with (\ref{eq:cd-B}) for $B' =
  {V_2'}^*$ and $B''= {V''}^*$, gives rise to
  the commutative diagram
  \begin{displaymath}
    \begin{CD}
      0 @>>> {V_2'}^* @>>> B @>>> {V''}^* @>>> 0 \\
      @. @VVV @VVV @| \\
      0 @>>> {V_1'}^* \oplus {V_2'}^* @>>> V_t^* @>>> {V''}^* @>>> 0 \\
      @. @VVV @VVV @| \\
      0 @>>> {V_2'}^* @>>> V_{t_2}^* @>>> {V''}^* @>>> 0.
    \end{CD}
  \end{displaymath}
  The composites of the vertical maps on the left and on the right are
  isomorphisms. Hence the composite of the middle vertical maps is
  also an isomorphism and this provides a splitting of the extension
  \begin{displaymath}
    0 \to {V_1'}^* \to V_t^* \to V_{t_2}^* \to 0.
  \end{displaymath}
  Denote the splitting maps in the dual split extension by
  \begin{displaymath}
    i\colon V_1' \to V_t\qquad\text{and}\qquad p\colon V_t \to V_{t_2}.
  \end{displaymath}
  We now have a diagram
  \begin{displaymath}
    \begin{CD}
    0 @>>> V'' @>>> V_{t_1} @>>> V_1' @>>> 0 \\
    @. @| @VVV @VVV \\
    0 @>>> V'' @>>> V_t @>>> V_1' \oplus V_2' @>>> 0 \\
    @. @| @VV{p}V @VVV \\
    0 @>>> V'' @>>> V_{t_2} @>>> V_2' @>>> 0,
    \end{CD}
  \end{displaymath}
  where the vertical maps on the right are the natural inclusion and
  projection, respectively. Using the existence of the splitting map
  $i\colon V_1' \to V_t$ and the inclusion $V_{t_2} \to V_t$ one
  readily sees that this diagram commutes. This finally gives us the
  commutative diagram
  \begin{displaymath}
    \begin{CD}
    0 @>>> 0 @>>> V_t/V_{t_1} @>{\simeq}>> V_2' @>>> 0 \\
    @. @| @VVV @| \\
    0 @>>> V'' @>>> V_{t_2} @>>> V_2' @>>> 0,
    \end{CD}
  \end{displaymath}
  which shows that the sequence at the bottom is split, a
  contradiction.
\end{proof}

%%%%%%%%%%%%%%%%%%%%%%%%%%%%%%%%%%%%%%%%%%%%%%%%%%%%%%%%%%%%%%%%%%%
\subsection{Deforming a sum of a quiver type minimum and a minimum
  in $\mathcal{N}_d$. }
%%%%%%%%%%%%%%%%%%%%%%%%%%%%%%%%%%%%%%%%%%%%%%%%%%%%%%%%%%%%%%%%%%%

\begin{lemma}
  \label{lem:beta-quiver-minima}
  Let $(V,\varphi)$ be a polystable $\Sp(2n,\RR)$-Higgs bundle which
  decomposes as a direct sum $(V,\varphi) = (V',\varphi') \oplus
  (V'',\varphi'')$ with $\varphi'=(\beta',\gamma')$ and
  $\varphi''=(\beta'',\gamma'')$. Suppose that
  \begin{enumerate}
  \item $(V',\varphi')$ is a quiver type minimum,
  \item $(V'',\varphi'')$ is a minimum with $\beta''= 0$ or
    $\gamma''=0$ which is a stable $G''$-Higgs bundle for $G''$ one of
    the following groups: $\Sp(2n'',\RR)$, $\U(p'',q'')$, $\U(n'')$ or
    $\GL(n'',\RR)$.
  \end{enumerate}
  Then $(V,\varphi)$ is not a minimum of $f$ on $\mathcal{M}_d$.
\end{lemma}

\begin{proof}
  Consider for definiteness the case in which $(V',\varphi')$ is a
  quiver type minimum with $\deg(V') = n'(1-g)$ and $(V'',\varphi'')$
  has $\gamma''=0$ and $\beta''\neq 0$. The case in which $\beta''=0$
  and $\gamma''\neq 0$ can be treated along the same lines as the
  present case, so we will not give the details. The case in which
  $(V',\varphi')$ is a quiver type minimum with $\deg(V') = n'(g-1)$
  is obtained by symmetry. Note that some degenerate cases can occur,
  namely:
  \begin{enumerate}
  \item\label{item:1} $(V',\varphi')$ is a quiver type minimum with
    $\rk(V') = 2$ (cf.\ (\ref{cor:minima-n=2}) of
    Corollary~\ref{cor:minima}).
  \item\label{item:2} $(V'',\varphi'')$ has $\beta''=\gamma''=0$.
  \end{enumerate}
  With respect to Case~(\ref{item:1}), all we need for the arguments
  below is that $\beta\colon F_{\frac{3}{2}}^* \xra{\simeq}
  F_{-\frac{1}{2}}\otimes K$ is an isomorphism, which is guaranteed by
  (\ref{cor:minima-n=2}) of Corollary~\ref{cor:minima}. In what
  concerns Case~(\ref{item:2}), slight modifications are required in
  the arguments given below; we leave these to the reader.

  With these introductory remarks out of the way,
  Corollary~\ref{cor:minima} tells us that $V'$ decomposes as a direct
  sum of line bundles
  $V'=F_m\oplus\dots\oplus F_M$ and that restricting $\beta'$ we get
  an isomorphism
$$\beta':F_m^*\stackrel{\simeq}{\longrightarrow} F_M\otimes K.$$
Our first task is to identify nonzero elements in the first
hypercohomology of $C^{\bullet}_-$. A good place to look for them
is in the hypercohomology of the piece of highest weight in the
deformation complex, which is
\begin{equation}
\label{eq:1} V''^*\otimes F_M\oplus V''\otimes F_m^*\to V''\otimes
F_M\otimes K.
\end{equation}
This morphism cannot be an isomorphism, because the ranks do not
match. Thus Proposition~\ref{prop:ad-iso} implies that $\HH^1$ of this
complex is non-vanishing.

In the hypercohomology long exact sequence (cf.\
(\ref{eq:hyper-les})) of the complex (\ref{eq:1}), the map
$$H^0(V''^*\otimes F_M\oplus V''\otimes F_m^*)
= H^0(V''^*\otimes F_M)\oplus H^0(V''\otimes F_m^*) \to
H^0(V''\otimes F_M\otimes K)$$ is always onto because the map
$f:H^0(V''\otimes F_m^*)\to H^0(V''\otimes F_M\otimes K)$ is
induced by tensoring $\beta':F_m^*\to F_M\otimes K$ (which is an
isomorphism) with the identity map $V''\to V''$, so $f$ is also an
isomorphism. Hence the image of $H^0(V''\otimes F_M\otimes K) \to
\HH^1$ is zero, and this by exactness implies that $\HH^1 \to
H^1(V''^*\otimes F_M\oplus V''\otimes F_m^*)$ is injective. We now
want to characterize the image of this inclusion. Tensoring the
Higgs fields $\beta''$ and $\beta'$ with the identity on $F_M$
and $V''$ respectively, we get maps
$$\beta''\colon V''^*\otimes F_M \to V''\otimes F_M\otimes K,$$
and
$$\beta'\colon V''\otimes F_m^* \xra{\simeq} V''\otimes F_M\otimes K.$$
Now the map $\zeta$ in the long exact sequence
$$\HH^1 \to
H^1(V''^*\otimes F_M\oplus V''\otimes F_m^*)
\stackrel{\zeta}{\longrightarrow} H^1(V''\otimes F_M\otimes K) \to
\HH^2$$ can be interpreted as follows: given elements
$(\delta,\epsilon)\in H^1(V''^*\otimes F_M)\oplus H^1(V''\otimes
F_m^*)$,
$$\zeta(\delta,\epsilon)=-\beta''(\delta)-\beta'(\epsilon)\in
H^1(V''\otimes F_M\otimes K).$$ Hence we may take a nonzero pair
$(\delta,\eta)$ satisfying $\beta''(\delta)+\beta'(\epsilon)=0$
and corresponding to a nonzero element in the hypercohomology of
the complex (\ref{eq:1}). We next prove that the deformation along
$(\delta,\eta)$ is unobstructed, by giving an explicit
construction of a family of Higgs bundles $(V_t,\beta_t,\gamma_t)$
parametrized by $t\in\CC$ and restricting to $(V'\oplus
V'',\varphi'+\varphi'')$ at $t=0$.

Pick Dolbeault representatives
$a_{\delta}\in\Omega^{0,1}(V''^*\otimes F_M)$ and $a_{\epsilon}\in
\Omega^{0,1}(F_m^*\otimes V'')$ of $\delta$ and $\epsilon$. We are
going to construct a pair $(W_t,\nu_t)$ satisfying the following.
\begin{itemize}
\item There is a $C^{\infty}$ isomorphism of vector bundles
$W_t\simeq F_M\oplus V''\oplus F_m$ with respect to which the
$\dbar$ operator of $W_t$ can be written as
$$\dbar_{W_t}=\left(\begin{array}{ccc}
\dbar_{F_M} & ta_{\delta} & t^2\gamma \\
0 & \dbar_{V''} & t a_{\epsilon} \\
0 & 0 & \dbar_{F_m}\end{array}\right)=\dbar_0+t a_1+t^2 a_2,$$
where $\gamma\in \Omega^{0,1}(F_m^*\otimes F_M)$ will be specified
later, \item $\nu_t$ is a holomorphic section of
$H^0(S^2W_t\otimes K)$ of the form
$$\nu_t=\beta'+\beta''+t\nu_1.$$
\end{itemize}
Now the condition $\dbar_{W_t}\nu_t=0$ translates into
\begin{align*}
\dbar_0(\beta'+\beta'') & = 0,\\
\dbar_1\nu_1+a_1(\beta'+\beta'') & =0, \\
a_1\nu_1+a_2(\beta'+\beta'') & = 0. \end{align*} The first
equation is automatically satisfied. As for the second equation
note that
$$a_1(\beta'+\beta'')=\beta''(a_\delta)+\beta'(a_{\epsilon})\in
\Omega^{1,1}(V''\otimes_S F_M).$$ Since by hypothesis the
Dolbeault cohomology class represented by
$\beta''(a_\delta)+\beta'(a_{\epsilon})$ is equal to zero, we may
chose a value of $\nu_1\in\Omega^{0,1}(V''\otimes_S F_M)$ solving
the second equation. It remains to consider the third equation.
Note that $a_2\beta''=0$ and that
$a_2\beta'=\gamma(\beta')\in\Omega^{1,1}(F_M\otimes F_M)$. Since
$\beta'$ is an isomorphism, for any
$\eta\in\Omega^{1,1}(F_M\otimes F_M)$ there exist some $\gamma$
such that $\gamma(\beta')=\eta$. Taking $\eta=-a_1\nu_1$, we
obtain a value of $\gamma$ solving the third equation above.

It follows from the construction that there are short exact
sequences of holomorphic bundles
$$0\to F_M\to W_t\to Z_t\to 0,\qquad\qquad
0\to V''\to Z_t\to F_m\to 0.$$ Dualizing both sequences we have
inclusions $F_m^*\to Z_t^*$ and $Z_t^*\to W_t^*$ which can be
composed to get an inclusion
\begin{equation}
\label{eq:inclusion-F-m} F_m^*\to W_t^*.
\end{equation}
Now let
$$V_t=W_t\oplus \bigoplus_{m<\lambda<M} F_\lambda.$$
To finish the construction of the family of Higgs bundles we have
to define holomorphic maps
$$\beta_t:V_t^*\to V_t\otimes K,\qquad\qquad
\gamma_t:V_t\to V_t^*\otimes K$$ defining sections in
$H^0(S^2V_t\otimes K)$ and $H^0(S^2V_t^*\otimes K)$ respectively.
The following conditions are in fact satisfied by a unique choice
of maps $(\beta_t,\gamma_t)$:
\begin{itemize}
\item the restriction of $\beta_t$ to $W_t$ is equal to $\nu_t$,
\item the restriction of $\beta_t$ to $\bigoplus_{m<\lambda<M}
F_\lambda$ is equal to $\beta'$, \item the restriction of
$\gamma_t$ to $W_t$ is equal to $0$, \item the restriction of
$\gamma_t$ to $F_M\subset V_t$ is $0$, \item the restriction of
$\gamma_t$ to $F_{M-2}\subset V_t$ is the composition of
$\gamma':F_{M-2}\to F_m^*\otimes K$ with the inclusion
(\ref{eq:inclusion-F-m}) tensored by the identity on $K$, \item
the restriction of $\gamma_t$ to $\bigoplus_{m<\lambda<M-2}
F_\lambda$ is equal to $\gamma'$.
\end{itemize}

The proof of the lemma is completed by using
Lemma~\ref{lem:poly-stable-quiver-beta-deform}.
\end{proof}

\begin{lemma}
  \label{lem:poly-stable-quiver-beta-deform}
  The $\Sp(2n,\RR)$-Higgs bundle $(V_t,\varphi_t)$, obtained by
  restricting the family constructed in the proof of
  Lemma~\ref{lem:beta-quiver-minima} to $X \times \{t\}$, is
  polystable.
\end{lemma}

\begin{proof}
  Analogous to the proof of Lemma~\ref{lem:beta-gamma-family-polystable}.
\end{proof}

%%%%%%%%%%%%%%%%%%%%%%%%%%%%%%%%%%%%%%%%%%%%%%%%%
\subsection{Deforming a sum of two quiver type minima.}
%%%%%%%%%%%%%%%%%%%%%%%%%%%%%%%%%%%%%%%%%%%%%%%%%

\begin{lemma}
  \label{lem:quiver-2-minima}
  Let $(V,\varphi)$ be a polystable $\Sp(2n,\RR)$-Higgs bundle which
  decomposes as a direct sum $(V,\varphi) = (V',\varphi') \oplus
  (V'',\varphi'')$ with $\varphi'=(\beta',\gamma')$ and
  $\varphi''=(\beta'',\gamma'')$. Suppose that both $(V',\varphi')$
  and $(V'',\varphi'')$ are quiver type minima.  Then $(V,\varphi)$ is
  not a minimum of $f$ on $\mathcal{M}_d$.
\end{lemma}

\begin{proof}
  Suppose we have two minima which are quiver pairs (minimal degree)
$$V'=F'_{m'}\oplus\dots\oplus F'_{M'}=\bigoplus F_\lambda'
\qquad\text{and}\qquad V''=F''_{m''}\oplus\dots\oplus
F''_{M''}=\bigoplus F_\mu''.$$ All morphisms
$\beta',\beta'',\gamma',\gamma''$ are isomorphisms. We want to
deform $V'\oplus V''$.

The same ideas as before tell us (looking at the negative
deformation complex) that we should look at the piece of the exact
sequence of maximal weight, which is
$$C^{\bullet}:F'^*_{m'}\otimes F''_{M''}\oplus
F''^*_{m''}\otimes F'_{M'}\to F'_{M'}\otimes F''_{M''}\otimes K.$$
Define $V_0'':=F''_{m''}\oplus F''_{M''}$. The restriction of the
$\beta''$ to $V_0''$ defines an isomorphism
$$\beta_0'':V_0^*\to V_0''\otimes K,$$
so we can apply exactly the same construction as before, replacing
$V''$ by $V_0''$, and obtain a deformation $W_{t\delta,t\epsilon}$
of the bundle
$$F'_{m'}\oplus F'_{M'}\oplus V_0''=F'_{m'}\oplus F'_{M'}\oplus
F''_{m''}\oplus F''_{M''}.$$ A very important point, however, is
that now the extension classes of the bundles $W_\delta$ and
$W_\epsilon$ are more restricted, since they belong respectively
to the groups $H^1(F''^*_{m''}\otimes F'_{M'})$ and
$H^1(F'^*_{m'}\otimes F''_{M''})$. In particular, to define
$W_{t\epsilon}$ the line bundle $F'_{m'}$ only {\it merges} with
$F''_{M''}$, and not with $F''_{m''}$. This implies that there is
a map
\begin{equation}
\label{eq:4} W_{t\epsilon}\to F''_{m''}
\end{equation}
which deforms the projection $V_0''\to F''_{m''}$.

We leave all the remaining $F'_\lambda$ and $F''_\mu$ untouched.
There are only two maps which have to be deformed (apart from the
$\beta$'s which are internal in $W_{\delta,\epsilon}$). These are
$$\gamma':F'_{m'} \to F'^*_{M'-2}\otimes K
\qquad\text{and}\qquad \gamma'':F''_{m''} \to F''^*_{M''-2}\otimes
K.$$ The first one can be deformed to a map
$$\gamma'_{\delta,\epsilon}:W_{t\delta,t\epsilon}\to
F'^*_{M'-2}\otimes K$$ exactly as in the previous section. As for
$\gamma''$, we combine the projection $W_{t\delta,t\epsilon}\to
W_{t\epsilon}$ with the map in (\ref{eq:4}) and with $\gamma''$ to
obtain the desired deformation
$$W_{t\delta,t\epsilon}\to F''^*_{M''-2}\otimes K.$$
Lemma~\ref{lem:poly-stable-quiver-2-deform} below completes the
proof.
\end{proof}

\begin{lemma}
  \label{lem:poly-stable-quiver-2-deform}
  The $\Sp(2n,\RR)$-Higgs bundle $(V_t,\varphi_t)$, obtained by
  restricting the family constructed in the proof of
  Lemma~\ref{lem:quiver-2-minima} to $X \times \{t\}$, is
  polystable.
\end{lemma}

\begin{proof}
  Analogous to the proof of Lemma~\ref{lem:beta-gamma-family-polystable}.
\end{proof}

%%%%%%%%%%%%%%%%%%%%%%%%%%%%%%%%%%%%%%%%%%%%%%%%%%%%%%%
\subsection{$\GL(n,\RR)$-Higgs bundles.}
%%%%%%%%%%%%%%%%%%%%%%%%%%%%%%%%%%%%%%%%%%%%%%%%%%%%%%%

In this section, we will assume that
$$
(V,\varphi) = \upsilon_*^{\GL(n,\CC)}((W,Q),\psi)
$$
is an $\Sp(2n,\RR)$-Higgs bundle associated to a $\GL(n,\RR)$-Higgs
bundle $((W,Q),\psi)$. Recall that $d = \deg(V)=0$ in this case.
\begin{lemma}
  \label{lem:glnr-spnr-minima}
  Let $(V,\varphi)$ be the $\Sp(2n,\RR)$-Higgs bundle associated to a
  $\GL(n,\RR)$-Higgs bundle $((W,Q),\psi)$ as in
  (\ref{eq:gln-spn-higgs}). If $(V,\varphi)$ is a minimum of $f$ on
  $\mathcal{M}_0$ then $\varphi=0$.
\end{lemma}

\begin{proof}
  In \cite{bradlow-garcia-prada-gothen:2004b} it is shown that there
  are two types of minima on the moduli space $\GL(n,\RR)$-Higgs
  bundles $((W,Q),\psi)$. The first type has $\psi=0$.  The second
  type corresponds to the minimum on the Hitchin--Teichm\"uller
  component and has non-vanishing Higgs field. They are of the
  form:
  \begin{displaymath}
    W = F_{-m} \oplus \dots \oplus F_{m}
  \end{displaymath}
  for line bundles $F_i$, indexed by integers for $n=2m+1$ odd and
  half-integers for $n=2m+1$ even. More precisely, $F_i \simeq K^{-i}$
  so that, in particular, $F_i \simeq F_{-i}^*$. With respect to this
  decomposition of $W$,
  \begin{displaymath}
    Q =
    \begin{pmatrix}
        0 &  \cdots& \cdots & 0 & 1 \\
        \vdots & & & \iddots & 0 \\
        \vdots & & 1 & & \vdots \\
        0 & \iddots & & & \vdots \\
        1 & 0 & \cdots& \cdots & 0
    \end{pmatrix} \qquad\text{and}\qquad
    \psi =
    \begin{pmatrix}
      0 &\cdots &\cdots &\cdots & 0 \\
      1 & 0 & \cdots & \cdots& 0\\
      0 & 1 & 0 & \cdots & 0 \\
      \vdots  &  & \ddots & & \vdots \\
      0 & \cdots & 0 & 1 & 0
    \end{pmatrix}.
  \end{displaymath}
  We shall apply the criterion in (\ref{item:non-minima}) of
  Proposition~\ref{prop:minima-HH1} to show that
  $\upsilon_*^{\GL(n,\CC)}((W,Q),\psi)$ is not a minimum of the
  Hitchin functional for such $((W,Q),\psi)$.

  Recall that $V=W$, $\beta=\psi f^{-1}$ and $\gamma=f\psi$, where
  $f\colon V \to V^*$ is the symmetric isomorphism associated to
  $Q$. Hence the components of $\beta$ and $\gamma$ are the canonical
  sections
  \begin{displaymath}
    \beta\colon F_i^* \to F_{-i+1}\otimes K\quad\text{and}\quad
    \gamma\colon F_i \to F_{-i-1}^* \otimes K.
  \end{displaymath}

  Since $\varphi$ has weight one, the weight of $F_i$ is $i$ (cf.\
  Proposition~\ref{prop:VHS}). It follows that the highest weight
  piece of the complex $C^{\bullet}_-$ defined in
  (\ref{eq:def-C-minus}) is
  \begin{displaymath}
    C^{\bullet}_{2m}\colon \Hom(F_{-m},F_{m}) \to 0.
  \end{displaymath}
  Hence
  \begin{displaymath}
    \HH^1(C^{\bullet}_{2m}) = H^1(\Hom(F_{-m},F_{m}))
    = H^1(K^{-2m}),
  \end{displaymath}
  which is non-vanishing. Take a non-zero $a \in
  H^1(\Hom(F_{-m},F_{m}))$. Let $D$ be the open unit disk and let
  $\FF_j$ be the pull-back of $F_j$ to $X \times D$. Let
  \begin{equation}\label{eq:6}
    0 \to \FF_{m} \to \WW_a \to \FF_{-m} \to 0
  \end{equation}
  be the extension with class
  \begin{displaymath}
    a \otimes 1 \in H^1(\Hom(\FF_{-m},\FF_{m})) \simeq
    H^1(X;\Hom(F_{-m},F_{m})) \otimes H^0(D;\CC).
  \end{displaymath}
  Then $\VV_a = \WW_a \oplus \bigoplus_{i < m} \FF_i$ is a family
  deforming $V$ which is tangent to $a$ at $t=0 \in D$. To obtain the
  required deformation of $(V,\varphi)$ it thus remains to define the
  Higgs field $\varphi_D \in H^0(S^2\VV_a \otimes K)$ deforming
  $\varphi$. The only pieces of $\varphi$ which do not automatically
  lift are the ones involving $F_{-m}$ and $F_{m}$, i.e., $\beta \in
  H^0(\Hom(F_{-m+1}^*,F_m)\otimes K)$ and $\gamma\in
  H^0(\Hom(F_{-m},F_{m-1}^*)\otimes K)$. In order to lift $\beta$,
  clearly we should define $\beta_D$ to be the composition
  \begin{displaymath}
    \FF_{-m+1}^* \xrightarrow{\beta} \FF_{m} \to \WW_a,
  \end{displaymath}
  where the last map is induced from the injection in (\ref{eq:6}). A
  similar construction gives the lift $\gamma_D$ of $\gamma$. We have
  thus constructed a family $(\VV_a,\beta_D,\gamma_D)$ which is
  tangent to $a \in H^1(C^{\bullet}_{2m}(V,\varphi))$ for $t = 0 \in
  D$. Hence Lemma~\ref{lem:poly-stable-gln-deform} below completes the
  proof.
\end{proof}

\begin{lemma}
  \label{lem:poly-stable-gln-deform}
  The $\Sp(2n,\RR)$-Higgs bundle $(V_t,\varphi_t)$, obtained by
  restricting $(\VV_a,\beta_D,\gamma_D)$ constructed in the proof of
  Lemma~\ref{lem:glnr-spnr-minima} above to $X \times \{t\}$, is
  polystable.
\end{lemma}

\begin{proof}
  Analogous to the proof of Lemma~\ref{lem:beta-gamma-family-polystable}.
\end{proof}

%%%%%%%%%%%%%%%%%%%%%%%%%%%%%%%%%%%%%%%%%%%%%%%%%%
\section{Counting components: main results}
\label{sec:count-comp-main}
%%%%%%%%%%%%%%%%%%%%%%%%%%%%%%%%%%%%%%%%%%%%%%%%%%

%%%%%%%%%%%%%%%%%%%%%%%%%%%%%%%%%%%%%%%%%%%%%%%%%%%%%%%%%%%%%%%%%%%%%%%%%%%%%%%%%
\subsection{Connected components of $\cM_d$ for $d=0$ and  $\abs{d}=n(g-1)$}
\label{counting}
%%%%%%%%%%%%%%%%%%%%%%%%%%%%%%%%%%%%%%%%%%%%%%%%%%%%%%%%%%%%%%%%%%%%%%%%%%%%%%%%%%

With the description of the minima of the Hitchin functional given in
Theorem~\ref{thm:all-minima} at our disposal we are now in a position
to complete the count of connected components of the moduli space in
the situation of $d=0$ and $\abs{d} = n(g-1)$.

\begin{proposition}
  The quiver type minima belong to a Hitchin--Teichm\"uller component
  of the moduli space.  In particular, they are stable and simple and
  correspond to smooth points of the moduli space.
\end{proposition}

\begin{proof}
  This is immediate from the description of the $\Sp(2n,\RR)$-Higgs
  bundles of the Hitchin--Teichm\"uller component given in
  \cite{hitchin:1992}.
\end{proof}

\begin{proposition}\label{prop:char-classes}
  Assume that $d= -n(g-1)$ and let $(V,\beta,\gamma)$ be a quiver type
  minimum for the Hitchin functional.  Let $L_0$ be a fixed square
  root of the canonical bundle, giving rise to the Cayley
  correspondence isomorphism $\mathcal{M}_{-n(g-1)}
  \xrightarrow{\simeq} \mathcal{M}'$ of Theorem~\ref{thm:cayley}, via
  $V \mapsto W \otimes L_0$.  Then the following holds.
  \begin{enumerate}
  \item The second Stiefel--Whitney class $w_2(W) \in H^2(X,\ZZ_2)$
    vanishes.
  \item If $n$ is odd, the first Stiefel--Whitney class $w_1(W)$
    corresponds to the two-torsion point $L^{-1}L_0$ in the Jacobian
    of $X$ under the standard identification $J_2 \simeq H^1(X,\ZZ_2)$.
  \item If $n$ is even, the first Stiefel--Whitney class $w_1(W) \in
    H^1(X,\ZZ_2)$ vanishes.
  \end{enumerate}
\end{proposition}

\begin{proof}
  Easy (similar to the arguments given in \cite{hitchin:1992} for $G =
  \SL(n,\RR)$).
\end{proof}

\begin{theorem}\label{connected-components-higgs}
Let $X$ be a compact Riemann surface of genus $g$.
Let $\cM_d$ be the moduli space of polystable
$\Sp(2n,\RR)$-Higgs bundles of degree $d$.  Let $n\geq 3$. Then

\begin{enumerate}
\item $\cM_0$ is non-empty and connected;
\item $\cM_{\pm n(g-1)}$ has $3. 2^{2g}$ non-empty connected components.
\end{enumerate}
\end{theorem}

\begin{proof}
  (1) When $d=0$, we have from Theorem~\ref{thm:all-minima} that the
  subspace of minima of the Hitchin functional on $\mathcal{M}_0$ is
  $\mathcal{N}_0$.  It is immediate from
  Theorem~\ref{thm:stability-equivalence} that $\mathcal{N}_0$ is
  isomorphic to the moduli space of poly-stable vector bundles of
  degree zero. This moduli space is well known to be non-empty and
  connected and hence $\mathcal{M}_0$ is non-empty and connected.

  (2) For definiteness assume that $d = -n(g-1)$.  The decomposition
  (\ref{stiefel-whitney}) given by the Cayley correspondence gives a
  decomposition
  \begin{equation}\label{stiefel-whitney2}
    \cM_{-n(g-1)}=\bigcup_{w_1,w_2} \cM_{w_1,w_2},
  \end{equation}
  where $\cM_{w_1,w_2}$ corresponds to $\cM'_{w_1,w_2}$ under the
  Cayley correspondence.

  For each possible value of $(w_1,w_2)$, there may be one or more
  corresponding Hitchin-Teichm\"uller components contained in
  $\mathcal{M}_{w_1,w_2}$ (cf.\ Proposition~\ref{prop:char-classes});
  denote by $\tilde{\mathcal{M}}_{w_1,w_2}$ the complement to these.
  Since minima in $\mathcal{N}_{-n(g-1)}$ (i.e.\ with $\gamma = 0$)
  clearly do not belong to Hitchin--Teichm\"uller components, we see
  that the subspace of minima of $\tilde{\mathcal{M}}_{w_1,w_2}$
  consists of those $(V,\beta,\gamma)$ which have $\gamma=0$.  Thus,
  under the Cayley correspondence, this subspace of minima is
  identified with the moduli space of poly-stable
  $\mathrm{O}(n,\CC)$-bundles with the given Stiefel--Whitney classes
  $(w_1,w_2)$. 
  
  If $G$ is a connected reductive algebraic group over $\CC$,
  Ramanathan proves in \cite[Proposition~4.2]{ramanathan:1975} that the
  moduli space of $G$-principal bundles of a fixed topological type is
  connected. However, the hypothesis that $G$ is connected is not used
  in Ramanathan's proof (all that is required for his argument is that,
  when considering a holomorphic family of $G$-principal bundles, 
  semistability is an open condition in the analytic topology), so the
  statement extends immediately to
  $G=\mathrm{O}(n,\CC)$. It
  follows that the subspace of minima on
  $\tilde{\mathcal{M}}_{w_1,w_2}$ is connected and, hence, this space
  itself is connected by Proposition~\ref{prop:ps-top}.  Additionally,
  each $\tilde{\mathcal{M}}_{w_1,w_2}$ is non-empty (see, e.g.,
  \cite{oliveira:2005}). Therefore, there is one connected component
  $\tilde{\mathcal{M}}_{w_1,w_2}$ for each of the $2^{2g+1}$ possible
  values of $(w_1,w_2)$.  Adding to this the $2^{2g}$
  Hitchin--Teichm\"uller components gives a total of $3.2^{2g}$
  connected components, as stated.

  This accounts for all the connected components of
  $\mathcal{M}_{-n(g-1)}$ since there are no other minima of the
  Hitchin functional.
\end{proof}

%%%%%%%%%%%%%%%%%%%%%%%%%%%%%%%%%%%%%%%%%%%%%%%%%%%%%%%%%%%%
\subsection{Representations and $\Sp(2n,\RR)$-Higgs bundles}
%%%%%%%%%%%%%%%%%%%%%%%%%%%%%%%%%%%%%%%%%%%%%%%%%%%%%%%%%%%%

Let $\cR:=\cR(\Sp(2n,\RR))$ be the moduli space of reductive
representations of $\pi_1(X)$ in $\Sp(2n,\RR)$. Since
$\U(n)\subset \Sp(2n,\RR)$ is a maximal compact subgroup, we have
$$
\pi_1(\Sp(2n,\RR))\simeq \pi_1(\U(n))\simeq \ZZ,
$$
and the topological invariant attached to a representation
$\rho\in\cR$ is hence an element $d=d(\rho)\in\ZZ$. This integer is
called the \textbf{Toledo invariant} and coincides with the first
Chern class of a reduction to a $\U(n)$-bundle of the flat
$\Sp(2n,\RR)$-bundle associated to $\rho$.

Fixing the invariant $d\in \ZZ$ we consider, as in \eqref{eq:RdG},
$$
\mathcal{R}_d:=\{\rho \in \mathcal{R}\;\;\;\mbox{such that}
\;\;\; d(\rho)=d\}.
$$

\begin{proposition}\label{rep-duality}
The transformation  $\rho\mapsto {(\rho^t)}^{-1}$ in $\cR$ induces an
isomorphism of the moduli  spaces $\cR_d$ and $\cR_{-d}$.
\end{proposition}

As shown in Turaev \cite{turaev:1984} (cf.\ also Domic--Toledo
\cite{domic-toledo:1987}, the Toledo invariant $d$ of a representation
satisfies the Milnor--Wood type inequality
\begin{equation}\label{mw2}
\abs{d} \leq n(g-1).
\end{equation}

As a consequence we have the following.

\begin{proposition}
The moduli space $\cR_d$ is empty unless
$$
\abs{d} \leq n(g-1).
$$
\end{proposition}

As a special case of
Theorem \ref{na-Hodge} we have the following.

\begin{proposition}\label{Md-Rd}
The moduli spaces $\cR_d$ and $\cM_d$ are homeomorphic.
\end{proposition}

{}From Proposition \ref{Md-Rd} and Theorem
\ref{connected-components-higgs} we have the main
result of this paper regarding the connectedness properties of $\cR$
given by the following.

\begin{theorem}\label{connected-components-rep}
Let $X$ be a compact oriented surface of genus $g$.
Let $\cR_d$ be the moduli space of reductive representations of $\pi_1(X)$
in $\Sp(2n,\RR)$.  Let $n\geq 3$. Then
\begin{enumerate}
\item $\cR_0$ is non-empty and connected;
\item $\cR_{\pm n(g-1)}$ has $3. 2^{2g}$ non-empty connected components.
\end{enumerate}
\end{theorem}

%%%%%%%%%%%%%%
\end{document}